\documentclass[
preprint, 3p, 
number, 
sort&compress,
]{elsarticle}
\pdfoutput=1

\usepackage[utf8]{luainputenc}
\usepackage[english]{babel}
\usepackage{csquotes}

\usepackage[plainpages=false,pdfpagelabels,hidelinks,unicode]{hyperref}

\usepackage{lipsum}
\usepackage{amsfonts}
\usepackage{graphicx}
\usepackage{epstopdf}

\usepackage{amsmath}
\allowdisplaybreaks
\usepackage{amssymb}
\usepackage{commath}
\usepackage{mathtools}
\usepackage{bbm}

\usepackage{siunitx}
\usepackage{adjustbox}

\usepackage{algorithm}
\usepackage[noend]{algpseudocode}

\usepackage[normalem]{ulem}

\usepackage{amsthm}
\theoremstyle{plain}
  \newtheorem{theorem}{Theorem} 
  \newtheorem{lemma}[theorem]{Lemma} 
	
\theoremstyle{definition}
  
  \newtheorem{remark}[theorem]{Remark}
  \newtheorem{example}[theorem]{Example}
  
\usepackage{color}
\usepackage{graphicx}
\usepackage[small]{caption}
\usepackage{subcaption}

\ifx\useTikzForPlotting\undefined
\else
  \usepackage{pgfplots}
  \pgfplotsset{compat=1.11}
  \usetikzlibrary{external}
\fi

\usepackage{booktabs}
\usepackage{rotating}
\usepackage{multirow}

\usepackage{multicol}
\usepackage{enumitem}

\usepackage{calc}
\usepackage{xparse}

\renewcommand{\vec}[1]{\underline{#1}}
\NewDocumentCommand{\mat}{mo}{%
  \IfValueTF{#2}{%
    \underline{\underline{#1}}{#2}
  }{%
    \underline{\underline{#1}}\,
  }%
}

\renewcommand{\d}{\mathrm{d}} 
\newcommand{\intd}{\, \mathrm{d}}

\newcommand{\fnum}{f^{\mathrm{num}}}

\newcommand{\Oref}{\Omega_{\mathrm{ref}}}

\renewcommand{\epsilon}{\varepsilon}
\renewcommand{\phi}{\varphi}
\renewcommand{\rho}{\varrho}
\newcommand{\N}{\mathbb{N}}
\newcommand{\R}{\mathbb{R}}

\newsavebox{\DelimiterBox}
\newlength{\DelimiterHeight}
\newlength{\DelimiterDepth}
\newsavebox{\ArgumentBox}
\newlength{\ArgumentHeight}
\newlength{\ArgumentDepth}
\newlength{\ResizedDelimiterHeight}
\newlength{\ResizedDelimiterDepth}

\begin{document}

\begin{frontmatter}

\title{Shock Capturing by Bernstein Polynomials for Scalar Conservation Laws}

\author[myu]{Jan Glaubitz\corref{cor1}}
\ead{j.glaubitz@tu-bs.de}

\cortext[cor1]{Corresponding author}

\address[myu]{
Institute for Computational Mathematics, TU Braunschweig, Universitaetsplatz 2, 38106 Braunschweig,   
Germany.
}

\begin{abstract}
  A main disadvantage of many high-order methods for hyperbolic conservation laws lies in the famous 
Gibbs--Wilbraham phenomenon, once discontinuities appear in the solution.
Due to the Gibbs--Wilbraham phenomenon, the numerical approximation will be polluted by spurious 
oscillations, which produce unphysical numerical solutions and might finally blow up the 
computation. 
In this work, we propose a new shock capturing procedure to stabilise high-order spectral element approximations. 
The procedure consists of going over from the original (polluted) approximation to a convex combination of the 
original approximation and its Bernstein reconstruction, yielding a stabilised approximation. 
The coefficient in the convex combination, and therefore the procedure, is steered by a discontinuity sensor and is 
only activated in troubled elements. 
Building up on classical Bernstein operators, we are thus able to prove that the resulting Bernstein procedure is 
total variation diminishing and preserves monotone (shock) profiles. 
Further, the procedure can be modified to not just preserve but also to enforce certain bounds for the solution, 
such as positivity. 
In contrast to other shock capturing methods, e.\,g.\ artificial viscosity methods, the new procedure does not reduce 
the time step or CFL condition and can be easily and efficiently implemented into any existing code.
Numerical tests demonstrate that the proposed shock-capturing procedure is able to stabilise 
and enhance spectral element approximations in the presence of shocks. 
\end{abstract}


\end{frontmatter}

\section{Introduction}
\label{sec:introduction}

In this work, we introduce a shock capturing procedure for spectral element (SE) approximations of scalar hyperbolic 
conservation laws 
\begin{equation}\label{eq:cl}
    \partial_t u + \partial_x f(u) = 0, \quad x \in \Omega \subset \R,  
\end{equation} 
with appropriate initial and boundary conditions. 
Here, $u(t,x)$ is the unknown function and $f(u)$ is a $C^1$ flux function. 
We assume that the initial data $u_0$ is a function of bounded total variation in $\Omega$. 
Equation \eqref{eq:cl} might as well be written in its quasi-linear form 
\begin{equation}
    \partial_t u + a(u) \partial_x u = 0
\end{equation} 
with $a(u) = f'(u)$. 
We further recall that solutions of \eqref{eq:cl} may develop spontaneous jump discontinuities (shock waves and contact discontinuities) even when the initial data are smooth \cite{lax1973hyperbolic}. 
This important discovery has first been made by Riemann \cite{riemann1860fortpflanzung}.
Hence, the more general class of weak solutions is admitted, where \eqref{eq:cl} is satisfied in the sense of distribution theory. 
Since there are many possible weak solutions, however, equation \eqref{eq:cl} is augmented with an 
additional entropy condition, requiring 
\begin{equation}\label{eq:entr-cond}
    \partial_t U(u) + \partial_x F(u) \leq 0 
\end{equation}
to hold. 
$U$ is an entropy function and $F$ is a corresponding entropy flux satisfying $U' f' = F'$. 
A strict inequality in \eqref{eq:entr-cond} reflects the existence of physically reasonable shock waves. 

Next, we consider spectral/$hp$ element approximations of \eqref{eq:cl}. 
Also see \cite{karniadakis2013spectral} and references therein. 
SE approximations have been introduced in 1984 by Patera \cite{patera1984spectral} and can be interpreted 
as a formulation of finite elements (FE) that uses piecewise polynomials of high degrees. 
SE methods thus combine the advantages of (Galerkin) spectral methods with those of FEs by a simple 
application of the spectral method to subdomains (elements) $\Omega_i$ of $\Omega$. 
The elements $\Omega_i$ are mapped to a reference element $\Oref$ and all computations are carried out there. 
In one space dimension, this reference element is typically given by $\Oref=[-1,1]$. 
Thus, the SE approximation is based on a polynomial approximation of degree at most $N$, expressed w.\,r.\,t.\ some 
nodal or modal basis of $\mathbb{P}_N(\Oref)$. 
This polynomial approximation, in particular, is used to extrapolate the solution to the element boundaries (if these values are not already given as coefficients in a nodal basis). 
Subsequently, common numerical fluxes $\fnum$ are computed at each element boundary 
\cite{toro2013riemann}. 
Numerical fluxes allow information to cross element boundaries and are often essential for the stability of the method. 
They are incorporated into the approximation by some correction procedure at the boundaries. 
In discontinuous Galerkin (DG) methods, this is done by applying integration by parts to a weak form of \eqref{eq:cl}, 
see \cite{hesthaven2007nodal}. 
In flux reconstruction (FR) methods, which are based on the strong form of \eqref{eq:cl}, this is done by applying 
correction functions, see \cite{huynh2007flux}. 
Finally, $\partial_x f(u)$ is approximated using exact differentiation for the polynomial approximation. 

Convergence of the resulting SE method is either obtained by increasing the degree of the polynomials 
($p$-refinement) or by increasing the number of elements ($h$-refinement). 
SE methods are known to provide highly accurate approximations in smooth regions. 
But also from a computational point of view, the main difficulty of solving hyperbolic conservation laws is that discontinuities may arise:  
Resulting from the famous Gibbs--Wilbraham phenomenon \cite{hewitt1979gibbs}, the polynomial approximation is polluted 
by spurious oscillations, which might produce unphysical numerical solutions and finally blow up the computation. 
To overcome these problems, many researchers have extended shock capturing procedures from finite difference (FD) 
as well as finite volume (FV) methods to (high-order) SE schemes. 
The idea behind many of these procedures is to add dissipation to the numerical solution. 
This idea dates back to the pioneering work \cite{vonneumann1950method} of von Neumann and Richtmyer during the 
Manhattan project in the 1940's at Los Almos National Laboratory, where they constructed stable FD schemes for the 
equations of hydrodynamics by including artificial viscosity (AV) terms. 
Since then, AV methods have attracted the interest of many researchers and were investigated 
also for SE approximations in a large number of works 
\cite{persson2006sub,klockner2011viscous,glaubitz2018application,ranocha2018stability,glaubitz2019smooth}. 
Despite providing a robust and accurate way to capture (shock) discontinuities, AV terms are 
not trivial to include in SE approximations. 
Typically, they are nonlinear and consist of higher (second and fourth order) derivatives. 
Another drawback arises from the fact that AV terms can introduce additional harsh time step 
restrictions, when not constructed with care, 
and thus decrease the efficiency of the numerical method \cite{hesthaven2007nodal,glaubitz2019smooth}. 
Finally, we mention those methods based on order reduction \cite{baumann1999discontinuous,burbeau2001problem}, 
mesh adaptation \cite{dervieux2003theoretical}, weighted essentially non-oscillatory (WENO) concepts 
\cite{shu1988efficient,shu1989efficient}, and $\ell^1$ regularisation applied to high order approximations of the jump 
function \cite{glaubitz2019high}. 
Yet, a number of issues still remains unresolved. Often, the extension of these methods to multiple dimensions is 
an open question or they are regarded as too computational expensive.

Here, we propose an approach that combines the good properties of SE approximations in smooth regions with the total variation diminishing and shape preserving properties of Bernstein approximations for resolving shocks without spurious oscillations. 
In this new strategy, the approximation of $u$ in each element may vary from usual (high-order) interpolation polynomials to a Bernstein reconstruction of the solution. 
Further, by employing a discontinuity sensor, here based on comparing polynomial annihilation operators 
\cite{archibald2005polynomial} of increasing orders as proposed in \cite{glaubitz2019high}, the order of the 
approximations is reduced to one only in elements where the solution is not smooth. 
For instance mesh adaptation is hence not mandatory and (shock) discontinuities can be captured without modifying the number of degrees of freedom, the mesh topology, or even the method. 
Moreover, a slight modification of the proposed Bernstein reconstruction also allows the preservation of 
bounds, such as positivity of pressure and water height. 

By now, Bernstein polynomials as polynomial bases have been successfully applied to solvers for partial differential 
equations (PDEs) in a number of works. 
Especially in \cite{kirby2011fast,ainsworth2011bernstein}, efficient FE operators have been constructed by using 
Bernstein polynomials as shape functions. 
Also in \cite{beisiegel2015quasi}, Bernstein polynomials have been proposed as basis functions in a third-order 
quasi-nodal DG method for solving flooding and drying problems with shallow water equations. 
Finally, \cite{lohmann2017flux,anderson2017high} have investigated the potential of imposing discrete maximum principles for polynomials expressed in a basis of Bernstein polynomials. 
Yet, while Bernstein polynomials as basis functions have been studied in a few works by now, the associated approximation procedure resulting from the Bernstein operator is still of very limited use. 
In contrast to the above mentioned works, we do not only utilise Bernstein polynomials as basis 
functions, but we further propose to use the associated Bernstein operatore as a building stone for 
new shock capturing procedures.

The rest of this work is organised as follows. 
In \S \ref{sec:B-polynomials}, we revise bases of Bernstein polynomials and their associated Bernstein 
approximation operator. 
This operator will later be used to obtain 'smoother' reconstructions of polynomial approximations near discontinuities 
and provides certain structure-preserving and approximation properties. 
Building up on these properties, \S \ref{sec:Berstein-procedure} introduces the novel Bernstein procedure, 
which replaces polluted interpolation polynomial by convex combinations of the original (polluted) approximation 
and its Bernstein reconstruction. 
This process is steered by a polynomial annihilation sensor, which is discussed in \S \ref{sub:sensor}. 
\S \ref{sec:stability} investigates some analytical properties of the proposed Bernstein procedure, 
such as its effect on entropy, total variation and monotone (shock) profiles of the numerical solution. 
We stress that the Bernstein reconstruction is proven to be total variation diminishing (nonincreasing). 
Finally, \S \ref{sec:tests} provides numerical demonstrations for a series of different scalar test problems. 
We close this work with concluding thoughts in \S \ref{sec:conclusion}.
\section{Bernstein Polynomials and the Bernstein Operator}
\label{sec:B-polynomials}

In this section, we introduce Bernstein polynomials as well as some of their more important properties. 
On an interval $[a,b]$, the $N+1$ \emph{Bernstein basis polynomials of degree $N$} are defined as 
\begin{equation}
    b_{n,N}(x) = \binom{N}{n} \frac{(x-a)^n (b-x)^{N-n}}{(b-a)^n}, 
    \quad n = 0,\dots,N,
\end{equation}
and form a basis of $\mathbb{P}_N$. 
Thus, every polynomial of degree at most $N$ can be written as a linear combination of Bernstein basis polynomials, 
\begin{equation}
    B_N(x) = \sum_{n=0}^N \beta_n b_{n,N}(x), 
\end{equation}
called \emph{Bernstein polynomial} or \emph{polynomial in Bernstein form of degree $N$}. 
The coefficients are referred to as \emph{Bernstein coefficients} or \emph{B{\'e}zier coefficients}. 
We further define the linear \emph{Bernstein operator of degree $N$} for a function $u:[a,b] \to \R$ by 
\begin{equation}\label{eq:Bernstein-op}
    B_N[u](x) 
        = \sum_{n=0}^N u\left( a+\frac{n}{N}(b-a) \right) b_{n,N}(x).
\end{equation}
$B_N[\cdot]$ maps a function $u$ to a Bernstein polynomial of degree $N$ with Bernstein coefficients 
\begin{equation}
    \beta_n = u\left( a+\frac{n}{N}(b-a) \right).
\end{equation}
Without loss of generality, we can restrict ourselves to the intervals $[0,1]$ and $[-1,1]$. 
For sake of simplicity, the interval $[0,1]$ subsequently will be used for theoretical investigations. 
The interval $[-1,1]$, on the other hand, is typically used as a reference element in SE 
approximations. 
Thus, the proposed Bernstein procedure will be explained for this case.

\subsection{Structure-Preserving Properties}

Let us consider Bernstein polynomials on $[0,1]$.
We start by noting that the Bernstein basis polynomials form a partition of unity, i.\,e. 
\begin{equation}\label{eq:partition-of-unity}
    \sum_{n=0}^N b_{n,N}(x) = 1 
\end{equation} 
for all $N \in \N$. 
This is a direct consequence of the binomial theorem. 
Thus, we can immediately note 

\begin{lemma}\label{lem:boundary_pres}
    Let $B_N(x) = \sum_{n=0}^N \beta_n b_{n,N}(x)$ be a Bernstein polynomial of degree $N$ with Bernstein coefficients $\beta_1,\dots,\beta_N$.  
    Then 
    \begin{equation}
        m \leq \beta_n \leq M \quad \forall n=0,\dots,N 
        \quad \implies \quad
        m \leq B_N(x) \leq M 
    \end{equation}
    holds for the Bernstein polynomial. 
\end{lemma}

\begin{proof}
    Let $m \leq \beta_n \leq M$ for all $n=0,\dots,N$. 
    We therefore have 
    \begin{equation}
        m \sum_{n=0}^N b_{n,N}(x) \leq \sum_{n=0}^N \beta_n b_{n,N}(x) 
        \leq M \sum_{n=0}^N b_{n,N}(x)
    \end{equation} 
    and the assertion follows from \eqref{eq:partition-of-unity}.
\end{proof}

In particular, Lemma \ref{lem:boundary_pres} ensures that the Bernstein operator \eqref{eq:Bernstein-op} 
preserves the bounds of the underlying function $u$. 
In fact, Lemma \ref{lem:boundary_pres} not only ensures preservation of bounds by the Bernstein procedure, 
but also allows us to enforce such bounds. 
Moreover, the Bernstein operator preserves the boundary values of $u$, i.\,e.  
\begin{equation}
    B_N[u](0) = u(0) \quad \text{and} \quad B_N[u](1) = u(1) 
\end{equation} 
hold. 
This makes the later proposed Bernstein procedure a reasonable shock-capturing method not just in discontinuous FE 
approximations, such as DG methods, but also in continuous FE approximations, where the numerical solution is required 
to be continuous across element interfaces. 
Further, we revise the formula 
\begin{equation}
    B_N^{(k)}[u](x) = 
        N(N-1) \dots (N-k+1) \sum_{n=0}^{N-k} \Delta^k u\left( \frac{n}{N} \right) \binom{N-k}{n} x^n (1-x)^{N-n-k} 
\end{equation} 
with forward difference operator 
\begin{equation}
    \Delta^k u\left( \frac{n}{N} \right) 
        = \Delta \left( \Delta^{k-1} u\left( \frac{n}{N} \right) \right) 
        = u\left( \frac{n+k}{N} \right) - \binom{k}{1} u\left( \frac{n+k-1}{N} \right) 
            + \dots + (-1)^k u\left( \frac{n}{N} \right)
\end{equation}
for derivatives of \eqref{eq:Bernstein-op}; see \cite[Chapter 1.4]{lorentz2012bernstein}.
In particular, we have 
\begin{equation}\label{eq:derivative}
    B_N'[u](x) = N \sum_{n=0}^{N-1} \left[ u\left( \frac{n+1}{N} \right) - u\left( \frac{n}{N} \right) \right] 
        \binom{N-1}{n} x^n (1-x)^{N-1-n}
\end{equation} 
for the first derivative of $B_N[u]$.
This formula will be important later in order to show that the Bernstein procedure is able to preserve monotone (shock) profiles.

\subsection{Approximation Properties}
\label{subsub:approx}

Bernstein polynomials were first introduced by Bernstein \cite{bernstein1912demonstration} in a constructive proof of 
the famous Weierstrass theorem, which states that every continuous function on a compact interval can be approximated 
arbitrarily accurate by polynomials \cite{weierstrass1885analytische}. 
Hence, the sequence of Bernstein polynomials $\left( B_N[u] \right)_{N \in \N}$ converges uniformly to the 
continuous function $u$. 
Assuming that $u$ is bounded in $[0,1]$ and that the second derivative $u''$ exists at the point $x \in [0,1]$, 
we have
\begin{equation}\label{eq:Bernstein-approx}
    \lim_{N \to \infty} N \left( u(x) - B_N[u](x) \right) 
        = - \frac{x(1-x)}{2}u''(x) 
    \quad \text{for} \quad x \in [0,1];
\end{equation}
see \cite[chapter 1.6.1]{lorentz2012bernstein}. 
In particular, at points $x$ where the second derivative exists, the error of the Bernstein polynomial $B_N[u]$ 
therefore is of first order, i.\,e. 
\begin{equation}
  \left| u(x) - B_N[u](x) \right| \leq C N^{-1} u''(x) 
\end{equation}
for a $C > 0$. 
However, we should remember that solutions of scalar hyperbolic conservation laws 
-- which we intend to approximate -- 
might contain discontinuities. 
The structure of these solutions has been determined by 
Oleinik \cite{oleinik1957discontinuous,oleinik1964cauchy}, 
Lax \cite{lax1957hyperbolic}, 
Dafermos \cite{dafermos1977generalized}, 
Schaeffer \cite{schaeffer1973regularity}, 
Tadmor and Tassa \cite{tadmor1993piecewise}, 
and many more. 
Most notably for our purpose, Tadmor and Tassa \cite{tadmor1993piecewise} showed for scalar convex conservation laws 
that if the initial speed has a finite number of decreasing inflection points then it bounds the number of future shock 
discontinuities. 
Thus, in most cases the solution consists of a finite number of smooth pieces, each of which is as smooth as the 
initial data. 
Note that it is this type of regularity which is often assumed -- sometimes implicitly -- in the numerical treatment of 
hyperbolic conservation laws.
Hence, it appears to be more reasonable to investigate convergence of the sequence of Bernstein polynomials $\left( 
B_N[u] \right)_{N \in \N}$ for only piecewise smooth functions $u$. 
Here, we call a function $u$ \textit{piecewise $C^k$} on $[0,1]$ if there is a finite set of points 
${0 < x_1 < \dots < x_K < 1 }$ such that $\left. u \right|_{[0,1] \setminus \{x_1,\dots,x_K\}} \in C^k $ and if 
the one-sided limits 
\begin{equation}
  u( x_k^+ ) := \lim_{x \to x_k^+} u(x)
  \quad \text{and} \quad 
  u( x_k^- ) := \lim_{x \to x_k^-} u(x)  
\end{equation}
exist for all $k=1,\dots,K$. 
In this case $u$ is still integrable and the first order convergence of the Bernstein polynomials carries over in an 
$L^p$-sense.
\begin{theorem}
  Let $u$ be piecewise $C^2$ and let $\left( B_N[u] \right)_{N \in \N}$ be the sequence of corresponding Bernstein 
polynomials. 
  Then there is a constant $C > 0$ such that 
  \begin{equation}
    \left( \int_0^1 \left| u(x) - B_N[u](x) \right|^p \d x \right)^{\frac{1}{p}} 
      \leq C N^{-1}
  \end{equation}
  holds for all $N \in \N$ and $1 \leq p < \infty$.
\end{theorem}
\begin{proof}
  Let $0 < x_1 < \dots < x_k < 1$ be the points where $u$ is not $C^2$. 
  When further denoting $x_0=0$ and $x_{K+1} = 1$, we have 
  \begin{equation}
    \int_0^1 \left| u(x) - B_N[u](x) \right|^p \d x 
      = \sum_{k=0}^K \int_{x_k}^{x_{k+1}} \left| u(x) - B_N[u](x) \right|^p \d x
  \end{equation}
  and, since $u$ is piecewise $C^2$, there is a generic constant $C > 0$ such that 
  \begin{equation}\label{eq:proof_thw2}
    \left| u(x) - B_N[u](x) \right| \leq C N^{-1} u''(x) 
  \end{equation}
  holds for $x_k < x < x_{k+1}$. 
  It should be noted however that $u''$ might be unbounded on the open interval $(x_k,x_{k+1})$. 
  Therefore, let us choose $\varepsilon > 0$ such that $\varepsilon < \min_{k=0,\dots,K} |x_k - x_{k+1}|$ and 
  let us consider the closed interval $[x_k+\varepsilon,x_{k+1}-\varepsilon]$. 
  On these intervals $u''$ is bounded and \eqref{eq:proof_thw2} becomes 
  \begin{equation}
    \left| u(x) - B_N[u](x) \right| \leq C N^{-1}
  \end{equation}
  for $x_k+\varepsilon \leq x \leq x_{k+1}-\varepsilon$. 
  Hence, we have 
  \begin{equation}\label{eq:proof2_thw2}
  \begin{aligned}
    & \int_{x_k}^{x_{k+1}} \left| u(x) - B_N[u](x) \right|^p \d x \\
      & \quad = \int_{x_k}^{x_{k}+\varepsilon} \left| u(x) - B_N[u](x) \right|^p \d x \\ 
      & \qquad + \int_{x_{k}+\varepsilon}^{x_{k+1}-\varepsilon} \left| u(x) - B_N[u](x) \right|^p \d x 
	+ \int_{x_{k+1}-\varepsilon}^{x_{k+1}} \left| u(x) - B_N[u](x) \right|^p \d x \\ 
      & \quad = \int_{x_k}^{x_{k}+\varepsilon} \left| u(x) - B_N[u](x) \right|^p \d x 
	+ (x_{k+1}-x_k-2\varepsilon) C^p N^{-p} 
	+ \int_{x_{k+1}-\varepsilon}^{x_{k+1}} \left| u(x) - B_N[u](x) \right|^p \d x.
  \end{aligned}
  \end{equation}
  For the two remaining integrals, we remember that $u$ and therefore all $B_N[u]$ are bounded. 
  This yields 
  \begin{equation}
    \left| u(x) - B_N[u](x) \right| 
      \leq \left| u(x) \right| + \left| B_N[u](x) \right| 
      \leq 2 \norm{u}_\infty
  \end{equation}
  for all $x \in [0,1]$ and $N \in \N$. 
  As a consequence, \eqref{eq:proof2_thw2} reduces to 
  \begin{equation}
    \int_{x_k}^{x_{k+1}} \left| u(x) - B_N[u](x) \right|^p \d x 
      \leq (x_{k+1}-x_k-2\varepsilon) C^p N^{-p} + 4 \varepsilon \norm{u}_\infty
  \end{equation}
  and letting $\varepsilon \to 0$ results in 
  \begin{equation}
    \int_{x_k}^{x_{k+1}} \left| u(x) - B_N[u](x) \right|^p \d x 
      \leq (x_{k+1}-x_k) C^p N^{-p}.
  \end{equation}
  Finally, we have 
  \begin{equation}
  \begin{aligned}
    & \int_{x_k}^{x_{k+1}} \left| u(x) - B_N[u](x) \right|^p \d x 
      && \leq (x_{k+1} - x_k) C N^{-p} \\ 
    \implies & \int_{0}^{1} \left| u(x) - B_N[u](x) \right|^p \d x 
      && \leq C N^{-p} \\
    \implies & \left( \int_0^1 \left| u(x) - B_N[u](x) \right|^p \d x \right)^{\frac{1}{p}} 
      && \leq C N^{-1},
  \end{aligned} 
  \end{equation}
  which yields the assertion.
\end{proof}

One might reproach that this is an unacceptable order of convergence. 
We reply to this argument with a few selected counter-arguments.

\begin{remark}\label{rem:approx}
    \begin{itemize}
        \item 
        For numerical solutions of hyperbolic conservation laws, it is almost universally accepted that near shocks, the solution can be first order accurate at most \cite{persson2006sub}. 
        Thus, accuracy of the numerical solution won't decrease noticeably by reconstructing it as a Bernstein polynomial.  
        
        \item In fact, high-order methods often not even provide accuracy of first order but constant (or even decreasing) accuracy. 
        This is due to the Gibbs--Wilbraham phenomenon \cite{hewitt1979gibbs,richards1991gibbs,gottlieb1997gibbs} 
        for (polynomial) higher-order approximations of discontinuous functions. 
        Yet, for instance, Gzyl and Palacios \cite{gzyl2003approximation} have shown the absence of the 
Gibbs--Wilbraham phenomenon for Bernstein polynomials. 
        It is still an open problem which properties of the approximation cause the Gibbs--Wilbraham phenomenon, but it 
is our conjecture that the order of accuracy for smooth functions is the deciding factor.
        
        \item While Bernstein polynomials converge uniformly for every continuous function, for instance, 
        polynomial interpolation in general only converges if $u$ is at least (Dini--)Lipschitz continuous 
\cite{bernstein1912best}. 
        Even worse, Faber \cite{faber1914interpolatorische} showed in 1914 that no polynomial interpolation scheme 
        -- no matter how the points are distributed -- will converge for the whole class of continuous functions. 
        For approximations by orthogonal projection, such as Fourier series, a similar result follows from divergence 
        of the corresponding operator norms \cite{berman1958impossibility,petras1990minimal}. 
        
    \end{itemize}
\end{remark}

We conclude from Remark \ref{rem:approx} that Bernstein polynomials, while appearing not attractive for approximating 
sufficiently smooth functions, provide some advantages for the approximation of just continuous or even discontinuous 
functions.
\section{The Bernstein Procedure} 
\label{sec:Berstein-procedure} 

In this section, we introduce a novel sub-cell shock capturing procedure by using Bernstein polynomials. 
The procedure is described for the reference element $[-1,1]$ 
and is based on replacing polluted high-order approximations by their Bernstein reconstruction.

\subsection{Bernstein Reconstruction}
\label{sub:B-reconstruction}

We start by introducing the (modified) Bernstein reconstruction which is obtained by applying the Bernstein 
operator \eqref{eq:Bernstein-op} to a polynomial approximation. 
Let $u \in \mathbb{P}_N([-1,1])$ be an approximate solution at time $t \geq 0$ with coefficients 
$\vec{\hat{u}} \in \mathbb{R}^{N+1}$ w.\,r.\,t.\ a (nodal or modal) basis $\{ \varphi_n \}_{n=0}^N$ in the reference 
element $\Omega_{ref} = [-1,1]$. 
Then, the original approximation, for instance, obtained by interpolation or (pseudo) $L^2$-projection, is 
modified in the following way: 
\begin{enumerate}
    \item 
    Compute the \emph{Bernstein reconstruction of $u$} by   
    \begin{equation}\label{eq:Bernstein-reconstruction}
        B_N[u](x) 
            = \sum_{n=0}^N u\left( -1+2\frac{n}{N} \right) b_{n,N}(x).
    \end{equation}
    
    \item 
    Write $B_N[u]$ w.\,r.\,t.\ the basis $\{ \varphi_n \}_{n=0}^N$ by a change of bases with transformation 
matrix $\mat{T}$, i.\,e. 
    \begin{equation}\label{eq:transformation}
        \vec{u}^{(B)} = \mat{T} \vec{b}, 
    \end{equation}
    where $\vec{b}$ is a vector containing the Bernstein coefficients 
    $\beta_{n,N} = u\left( -1+2\frac{n}{N} \right)$.
\end{enumerate} 
To put it in a nutshell, the idea is to replace the coefficients $\vec{\hat{u}}$ in a troubled cell by the coefficients $\vec{u}^{(B)}$ of the Bernstein reconstruction. 
Table \ref{tab:cond} lists the condition numbers \cite{trefethen1997numerical} of the transformation matrix $\mat{T}$ 
for the Lagrange and Legendre bases. 
Thereby, the nodal basis of Lagrange polynomials is considered w.\,r.\,t.\ the Gauss--Lobatto points in $[-1,1]$.
For all reasonable polynomial degrees ($N \leq 10$) and both bases we observe the condition number to be fairly small. 
\renewcommand{\arraystretch}{1.5}
\begin{table}[htb]
  \centering
  \begin{adjustbox}{width=0.95\textwidth}
  \begin{tabular}{c cccccccccc}
    \multicolumn{11}{c}{$\mathrm{cond}(\mat{T})$} \\ 
    \toprule 
    $N$ & 1 & 2 & 3 & 4 & 5 & 6 & 7 & 8 & 9 & 10 \\ \hline 
    Lagrange 
      & $1.0 \times 10^{0}$ & $2.3 \times 10^{0}$ 
      & $4.4. \times 10^{0}$ & $8.6 \times 10^{0}$ 
      & $1.7 \times 10^{1}$ & $3.4 \times 10^{1}$ 
      & $6.7 \times 10^{1}$ & $1.3 \times 10^{2}$ 
      & $2.6 \times 10^{2}$ & $5.3 \times 10^{2}$ \\ \hline 
    Legendre & $1.0 \times 10^{0}$ & $1.9 \times 10^{0}$ 
      & $2.9 \times 10^{0}$ & $4.3 \times 10^{0}$ 
      & $5.4 \times 10^{0}$ & $7.7 \times 10^{0}$ 
      & $1.0 \times 10^{1}$ & $1.6 \times 10^{1}$ 
      & $2.4 \times 10^{1}$ & $4.1 \times 10^{1}$ \\ 
    \bottomrule
    \end{tabular}
    \end{adjustbox}
    \caption{Condition numbers of the transformation matrix $\mat{T}$ for the Lagrange (w.\,r.\,t.\ Gauss--Lobatto 
points) and Legendre bases. 
      The spectral norm $||\mat{T}|| = ||\mat{T}||_2 := \sqrt{ \lambda_{\mathrm{max}}((\mat{T}^*) \cdot \mat{T}) }$ has 
been used, where $\lambda_{\mathrm{max}}((\mat{T}^*) \cdot \mat{T})$ is the largest eigenvalue of the 
positive-semidefinite matrix $(\mat{T}^*) \cdot \mat{T}$. 
      This value is sometimes referred to as the largest singular value of the matrix $\mat{T}$.}
    \label{tab:cond}
\end{table}

Concerning the now obtained Bernstein reconstruction, Lemma \ref{lem:boundary_pres} does not only ensure preservation 
of bounds, but also allows us to enforce such bounds. 
Let us introduce the \emph{modified Bernstein reconstruction w.\,r.\,t.\ the lower bound $m$ and the upper bound $M$}
\begin{equation}\label{eq:mod_Bernstein_operator}
    B_N^{(m,M)}[u](x) = \sum_{n=0}^N \beta_n b_{n,N}(x) 
    \quad \text{with} \quad 
    \beta_n = 
    \renewcommand*{\arraystretch}{1.2}
    \left\{ 
    \begin{array}{ccc}
      u\left( -1 + 2\frac{n}{N} \right) & \text{, if } & m \leq u\left( -1 + 2\frac{n}{N} \right) \leq M \\ 
      m & \text{, if } & u\left( -1 + 2\frac{n}{N} \right) < m \\ 
      M & \text{, if } & u\left( -1 + 2\frac{n}{N} \right) > M
    \end{array}
    \right. .
\end{equation}
To the best of the authors' knowledge, the modified Bernstein operator \eqref{eq:mod_Bernstein_operator} has not been 
defined anywhere else yet.  
The most beautiful property of the modified Bernstein operator is preservation -- by default -- of lower and upper 
bounds $m$ and $M$, respectively. 
This follows directly from Lemma \ref{lem:boundary_pres} and is summarised in

\begin{theorem}
    Let $B_N^{(m,M)}$ be the modified Bernstein operator w.\,r.\,t.\ the lower bound $m$ and the upper bound $M$ given 
by \eqref{eq:mod_Bernstein_operator} and let $u$ be some function. 
    Then $B_N^{(m,M)}[u]$ fulfils 
    \begin{equation}
        m \leq B_N^{(m,M)}[u] \leq M.
    \end{equation}
\end{theorem}

We close this subsection by noting that, for instance, positivity (think about density in the Euler equations), 
can be easily ensured by setting $m = \varepsilon$, 
where $\varepsilon > 0$ is a suitable value larger than machine precision.

\subsection{Proposed Procedure} 
\label{sub:procedure} 

Finally, we propose a procedure on how to replace the original polynomial approximation by its (modified) Bernstein 
reconstruction. 
Let $\Omega_i$ be a troubled element with approximate solution $u \in \mathbb{P}_N([-1,1])$ 
and let the coefficients of $u$ w.\,r.\,t.\ a (nodal or modal) basis $\{ \varphi_n \}_{n=0}^N$ be given by 
$\vec{\hat{u}} \in \mathbb{R}^{N+1}$.  
The \emph{Bernstein procedure} consists of two steps: 
\begin{enumerate}
    \item 
    Compute the (modified) Bernstein reconstruction $B_N[u]$ of $u$.
    
    \item 
    Build an 'appropriate' convex combination $u^{(\alpha)}$ of the original approximation $u$ and its Bernstein 
    reconstruction $B_N[u]$, i.\,e. 
    \begin{equation}\label{eq:alpha-B-reconstruction}
        u^{(\alpha)}(x) = \alpha u(x) + (1-\alpha) B_N[u](x) 
    \end{equation} 
    with $\alpha \in [0,1]$.
    
\end{enumerate} 
W.\,r.\,t.\ the original basis $\{ \varphi_n \}_{n=0}^N$, and therefore utilising the transformation 
\eqref{eq:transformation}, the \emph{$\alpha$ Bernstein reconstruction} $u^{(\alpha)} \in \mathbb{P}_N([-1,1])$ can be 
written as  
\begin{equation}
    u^{(\alpha)}(x) = \sum_{n=0}^N \left[ \alpha \hat{u}_n + (1-\alpha) u^{(B)}_n \right] \varphi_n(x). 
\end{equation}
Hence, its coefficients are given by 
\begin{equation}
    \vec{u}^{(\alpha)} = \alpha \vec{\hat{u}} + (1-\alpha) \vec{u}^{(B)}.  
\end{equation}
This procedure can be incorporated easily into an already existing solver. 
Note that we have $u^{(1)} = u$ and $u^{(0)} = u^{(B)}$ in the extreme cases of $\alpha = 1$ and $\alpha = 0$. 
Thus, the $\alpha$ Bernstein reconstruction $u^{(\alpha)}$ (linearly) varies between the original 
(and potentially oscillating or boundary violating) approximation $u$ and the more robust (modified) Bernstein 
reconstruction $B_N[u]$.
This is illustrated in Figure \ref{fig:illustrate_B-procedure} for the signum function $u(x) = \text{sign}(x)$, 
different parameters $\alpha$, and $N=1,5,9$.

\begin{figure}[!htb]
  \centering
  \begin{subfigure}[b]{0.33\textwidth}
    \includegraphics[width=\textwidth]{%
      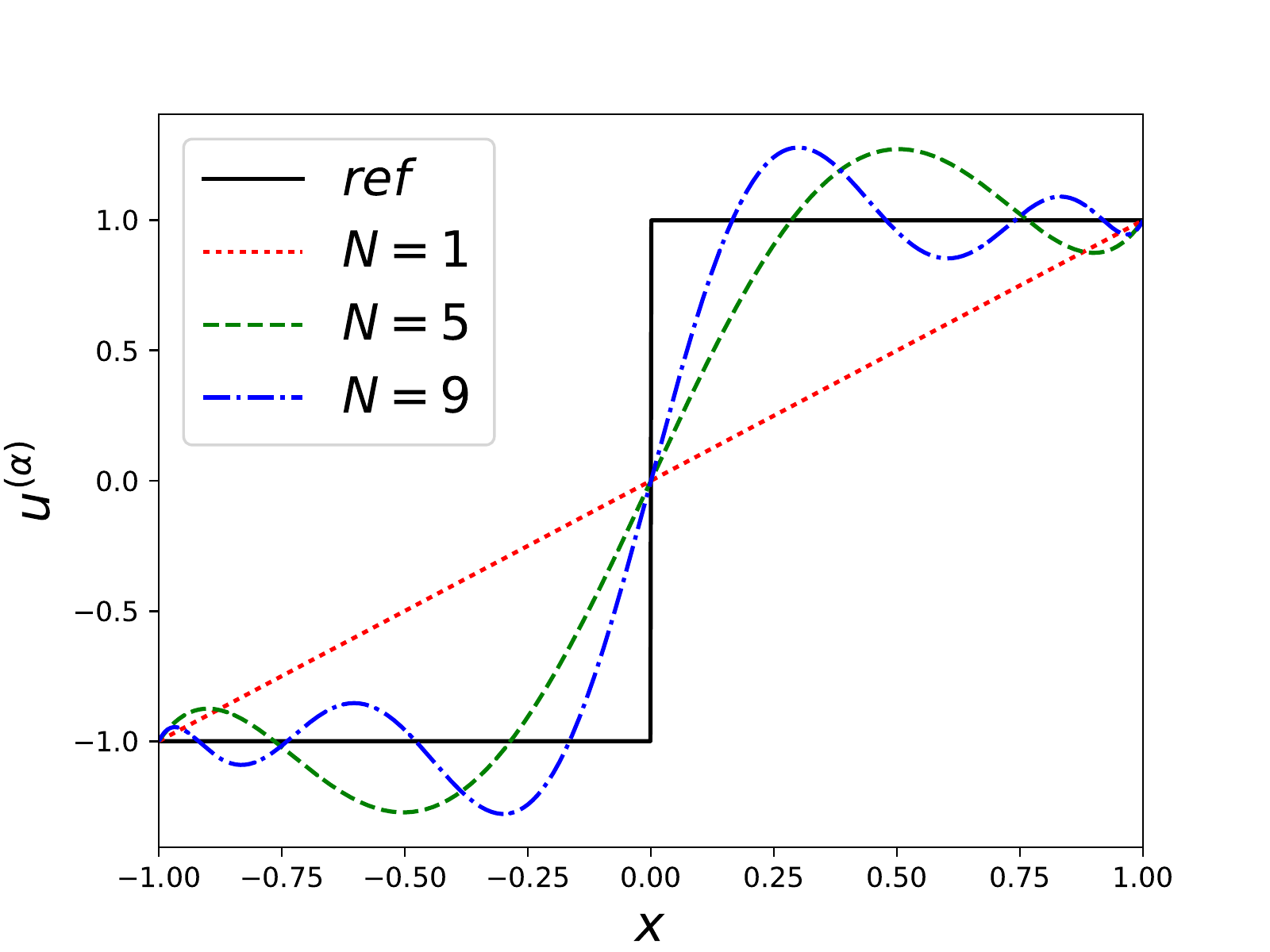}
    \caption{$\alpha=1$: $u^{(1)} = u$.}
    \label{fig:alpha1} 
  \end{subfigure}%
  ~ 
  \begin{subfigure}[b]{0.33\textwidth}
    \includegraphics[width=\textwidth]{%
      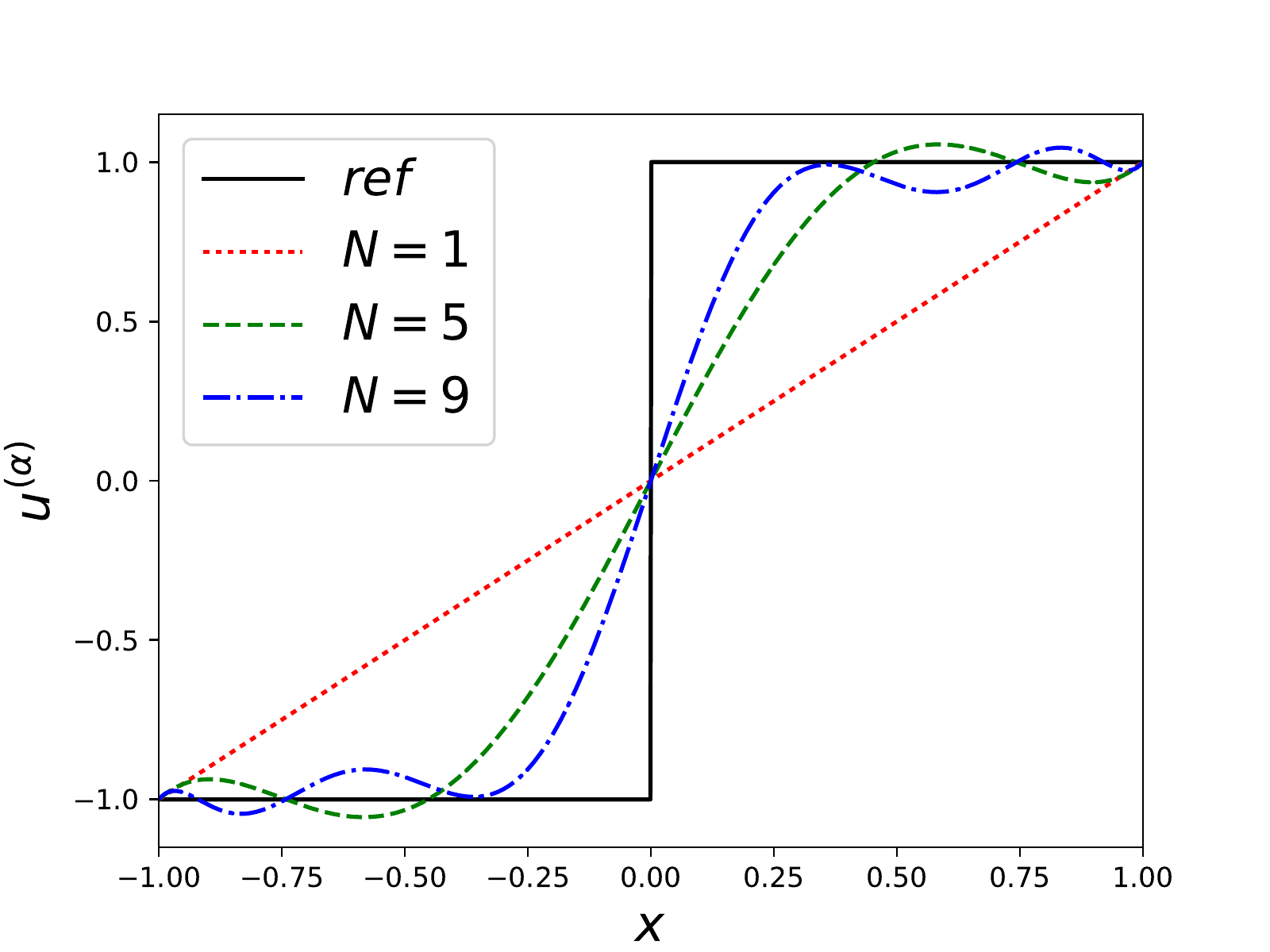} 
    \caption{$\alpha=\frac{1}{2}$: $u^{(0.5)} = \frac{1}{2}(u + B_N[u])$.}
    \label{fig:alpha05}
  \end{subfigure}%
  ~ 
  \begin{subfigure}[b]{0.33\textwidth}
    \includegraphics[width=\textwidth]{%
      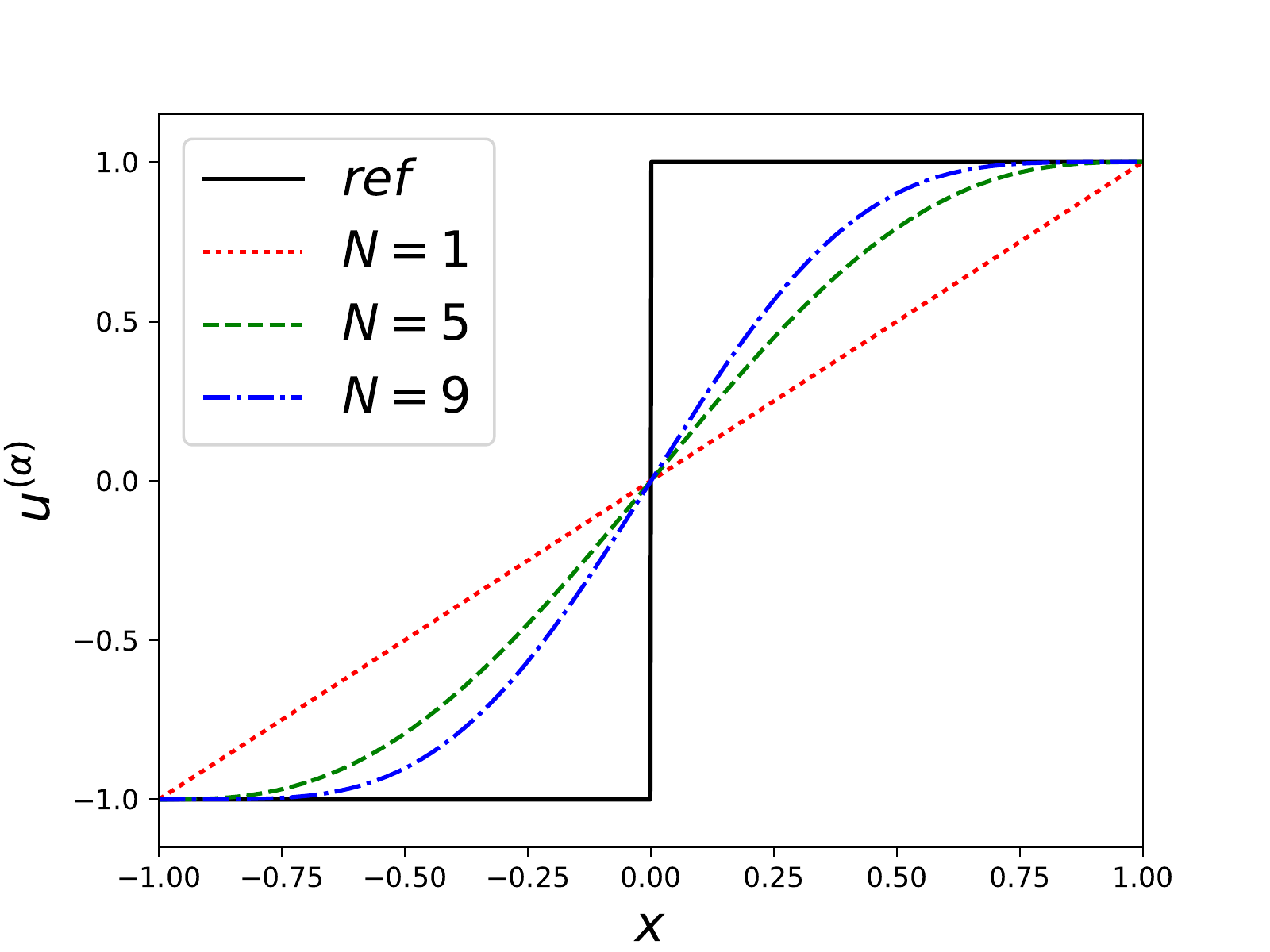} 
    \caption{$\alpha = 0$: $u^{(0)} = B_N[u]$.}
    \label{fig:alpha0}
  \end{subfigure}%
  \caption{Illustration of the Bernstein reconstruction $u^{(\alpha)}(x) = \alpha u(x) + (1-\alpha) B_N[u](x)$ 
for $\alpha=1$ (original interpolation $u$), $\alpha=0.5$, and $\alpha = 0$ ('full' Bernstein reconstruction 
$B_N[u]$).}
  \label{fig:illustrate_B-procedure}
\end{figure} 

Obviously, the order of the approximation is reduced to one in elements where $\alpha < 1$. 
Yet, in elements with $\alpha = 1$, which define the large majority of the domain, high accuracy is retained. 
Subsequently, we demonstrate how the parameter $\alpha$ can be adapted to the regularity of the underlying solution.

\subsection{\texorpdfstring{Selection of Parameter $\alpha$}{Selection of Parameter Alpha}} 
\label{sub:alpha} 

The parameter $\alpha$ in \eqref{eq:alpha-B-reconstruction} is adapted to adequately capture different discontinuities and regions of smoothness in the solution. 
The value of $\alpha$ adjusts in space and time to accurately capture strong variations in the solution. 
Hence, the proposed Bernstein procedure is able to calibrate the polynomial approximation to the regularity of the 
solution. 
It should be stressed that modification of the mesh topology, the number of degrees of freedom, node positions, or the 
type of SE method is utterly unnecessary. 
This makes the Bernstein procedure an efficient and easy to implement shock capturing method. 
As described above, the extreme values $\alpha=1$ and $\alpha=0$ yield the original approximation $u$ and the Bernstein 
reconstruction $B_N[u]$, respectively. 
For intermediate parameter values $\alpha \in (0,1)$, the Bernstein reconstruction renders the convex combination  
$u^{(\alpha)} = \alpha u + (1-\alpha)u^{(B)}$ more and more robust. 
Thus, $\alpha$ allows us to adapt the amount of stabilisation introduced by the Bernstein reconstruction. 
Here, we chose $\alpha$ being a function of a discontinuity sensor as proposed in \cite{persson2006sub,huerta2012simple} and briefly revisited in the next subsection.

\subsection{Discontinuity Sensor} 
\label{sub:sensor} 

In this work, we use a discontinuity sensor proposed in \cite{glaubitz2019high} to detect troubled elements and to steer 
the parameter $\alpha$. 
This sensor is an element-based function leading to a single scalar measure of the solutions' smoothness. 
It is a nonlinear functional 
\begin{equation}
    S: \Oref \to \R, \quad s \mapsto S(s), 
\end{equation} 
which depends on a \emph{sensing variable} $s$. 
For systems, such as the Euler equations, the density or Mach number could be utilised for the sensing variable. 
Since we only consider scalar conservation laws in this work, $s$ is simply chosen as the conserved quantity $u$. 
The sensor was first proposed in \cite{glaubitz2019high} and is based on comparing polynomial annihilation (PA) 
operators \cite{archibald2005polynomial} of increasing orders. 
A \emph{PA operator of order $m$}, 
\begin{equation}
    L_m[s](x) = \frac{1}{q_m(x)} \sum_{\xi_j \in S_x} c_j(x) s(\xi_j), 
\end{equation} 
is designed as a high-order approximation of the jump function 
\begin{equation}
    [s](x) = s\left(x^+\right) - s\left(x^- \right)
\end{equation} 
of the sensing variable $s$. 
Here, $S_x = \{ \xi_0(x), \dots, \xi_m(x) \} \subset [-1,1]$ denotes a set of $m+1$ local grid points around $x$ and the so called \emph{annihilation coefficients} $c_j:[-1,1] \to \R$ are given by 
\begin{equation}
    \sum_{\xi_j \in S_x} c_j(x) p_l(\xi_j) = p_l^{(m)}, 
    \quad j=0,\dots,m, 
\end{equation}
where $\{ p_l \}_{l=0}^m$ is a basis of $\mathbb{P}_m$. 
An explicit formula for the annihilation coefficients is provided in \cite{archibald2005polynomial} by 
\begin{equation}
    c_j(x) = \frac{m!}{\omega_j(S_x)} 
    \quad \text{with} \quad 
    \omega_j(S_x) = \prod_{\begin{smallmatrix} \xi_i \in S_x \\ i \neq j \end{smallmatrix}} ( \xi_j - \xi_i ). 
\end{equation}
The \emph{normalisation factor} $q_m$ is chosen as 
\begin{equation}
    q_m(x) = \sum_{\xi_j \in S_x^+} c_j(x) 
    \quad \text{with} \quad 
    S_x^+ = \{ \xi_j \in S_x \ | \ \xi_j \geq x \}
\end{equation} 
and ensures convergence to the right jump strengths. 
It has been proved in \cite{archibald2005polynomial} that 
\begin{equation}
    L_m[s](x) = 
    \left\{ 
    \begin{array}{ccl}
      [s](x) + \mathcal{O}\left(h(\xi)\right) & , & \text{if } \xi_{j-1} \leq x, \xi \leq \xi_j \\ 
      \mathcal{O}\left( h^{\min(m,k)}(\xi) \right) & , & \text{if } s \in C^k(I_x)
    \end{array}
    \right. ,
\end{equation} 
holds, where $h(\xi) = \max \{ |\xi_j - \xi_{j-1}| \ | \ \xi_{j-1},\xi_j \in S_x \}$ and $I_x$ is the smallest closed interval such that $S_x \subset I_x$. 
We note that PA might be further enhanced, for instance, by applying a minmod limiter as described in 
\cite{archibald2005polynomial}.

Next, let the \emph{sensor value of order $m$} be given by 
\begin{equation}\label{eq:sensor-value}
    S_m = \max_{k=0,\dots,p-1} \left| L_m[s] ( x_{k+1/2} ) \right|, 
\end{equation}
i.\,e.\ by the greatest absolute value of the PA operator of order $m$ at the mid points of the collocation points. 
If $s$ has at least $m$ continuous derivatives, the PA operator $L_m$ provides convergence to $0$ of order $m+1$, 
and we therefore expect the sensor value \eqref{eq:sensor-value} to decrease for an increasing order $m$. 
In this case, a parameter value $\alpha=1$ ($\alpha > 0$) is chosen, which means that the Bernstein procedure is not (fully) activated. 
In this work, we only compare the sensor values of order $m=1$ and $m=3$ since it was endorsed in \cite{archibald2005polynomial} to use the same number of local grid points $\xi_j$ on both sides of a point $x$. 
Of course, a variety of modification is possible for the PA sensor. 
Now, the parameter value $\alpha = 0$ is chosen only if 
\begin{equation}\label{eq:PA-sensor}
    S(s) := \frac{S_3}{S_1} \geq 1
\end{equation} 
holds, i.\,e.\ if the sensor value does not decrease when going over from order $m=1$ to $m=3$. 
Finally, we decide for the parameter $\alpha$ to linearly vary between $\alpha=1$ and $\alpha=0$. 
This is realised by the parameter function 
\begin{equation}\label{eq:param-fun}
    \alpha(S) =
    \left\{ 
    \begin{array}{ccl}
      1 & , & \text{if } S \leq \kappa \\ 
      \frac{1-S}{1-\kappa} & , & \text{if } \kappa < S < 1 \\ 
      0 & , & \text{if } 1 \leq S
    \end{array}
    \right. ,
\end{equation} 
where $\kappa \in (0,1)$ is a problem dependent \emph{ramp parameter}. 
Figure \ref{fig:param-fun} illustrates the parameter function w.\,r.\,t.\ the values of the PA sensor $S$ given by 
\eqref{eq:PA-sensor}. 

\begin{figure}[!htb]
  \centering
    \includegraphics[width=0.45\textwidth]{%
        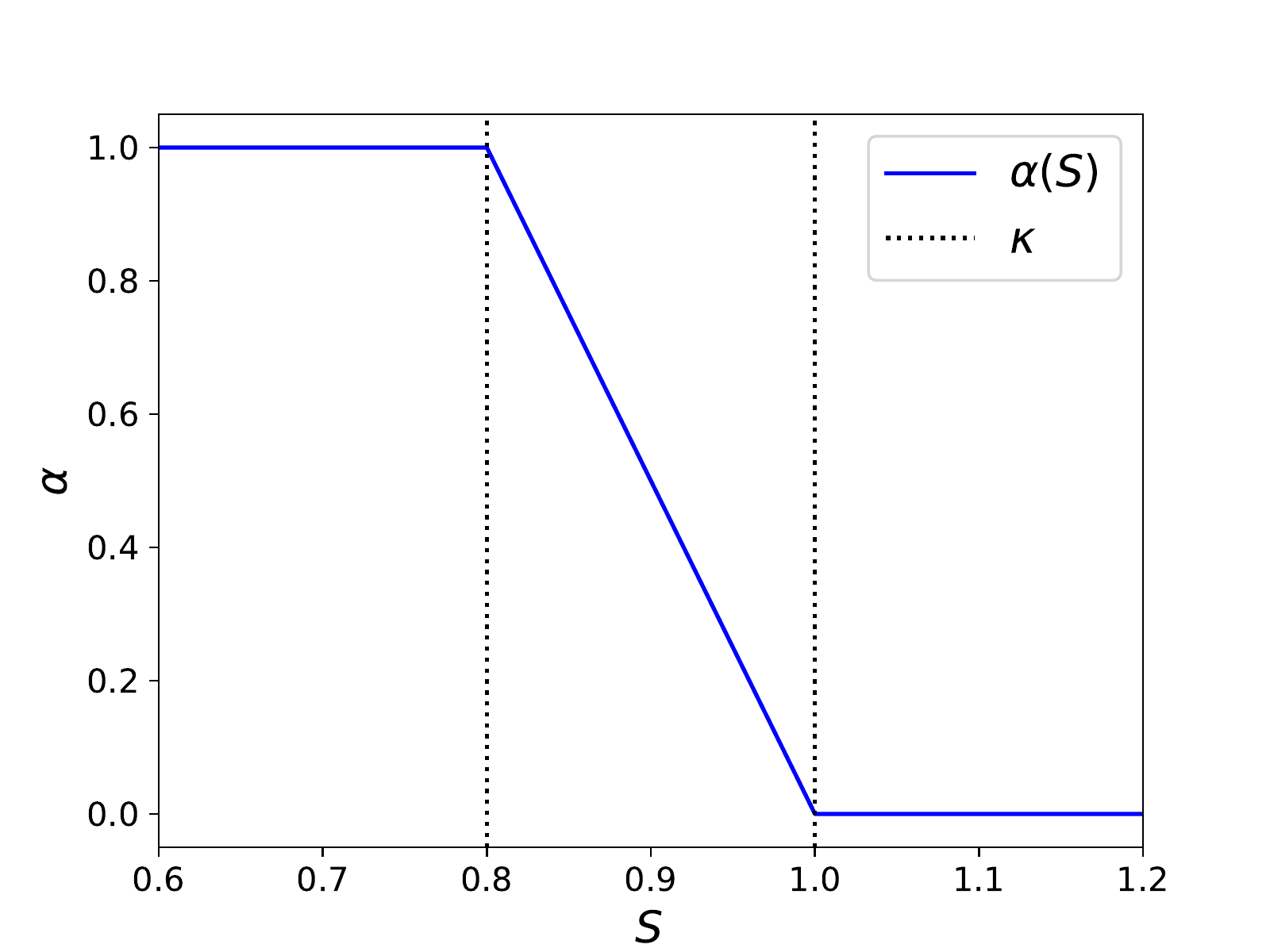}
  \caption{Parameter function $\alpha(S)$ as defined in \eqref{eq:param-fun}.}
   \label{fig:param-fun}
\end{figure}

For the later numerical tests we investigated different other parameter functions as well, 
of which some have been discussed in \cite{huerta2012simple}. 
Yet, we obtained the best results with \eqref{eq:param-fun}.
It should be noted that the above revisited PA sensor is only recommended for high orders $N \geq 4$. 
In our numerical test, we have observed some miss-identifications for $N=3$.
For $N \leq 3$ other discontinuity sensors could be used instead. 
We mention the modal-decay based sensor of Persson and Peraire 
\cite{persson2006sub,huerta2012simple} and its refinements \cite{barter2010shock,klockner2011viscous} as well as the 
KXRCF sensor \cite{krivodonova2004shock,qiu2005comparison} of Krivodonova et al., which is build up on a strong 
superconvergence phenomena of the DG method at outflow boundaries. 
Future work will address a detailed comparison of different shock sensors.

\begin{remark}\label{rem:param_tuning}
  The above sensor is simple to implement, but has its price, which is the introduction of the problem dependent and 
tunable parameter $\kappa \in (0,1)$ in \eqref{eq:param-fun}. 
  In practice, we follow a strategy proposed by Guermond, Pasquetti, and Popov for their \textit{entropy viscosity 
method for nonlinear conservation laws} \cite{guermond2011entropy}. 
  For a fixed polynomial degree $N$, the parameter $\kappa$ is tuned by testing the method on a coarse grid. 
  For all problems presented later in \S \ref{sec:tests}, the tuning has been done quickly on a coarse mesh of $I=10$ 
elements. 
  Further, we have observed the Bernstein procedure to be robust w.\,r.\,t.\ the parameter $\kappa$. 
  This will be demonstrated in \S \ref{subsub:linear_param} for the linear advection equation.
\end{remark}

\section{Entropy, Total Variation, and Monotone (Shock) Profiles}
\label{sec:stability}

In this section, we investigate some analytical properties of the proposed Bernstein procedure, such as its 
effect on entropy, total variation and monotone (shock) profiles of the numerical solution. 
For sake of simplicity, all subsequent investigations are carried out on the reference element $\Oref = [0,1]$. 

\subsection{Entropy Stability}
\label{sub:entropy}

Let $U \in C^1$ be a convex entropy function. 
We show that the proposed Bernstein reconstruction only yields a change in the total amount of 
entropy which is consistent with the total amount of the entropy of the original approximation $u$. 
This means that for increasing $N$ the total amount of entropy of the Bernstein reconstruction $B_N[u]$ converges to 
the total amount of entropy of the function $u$. 
Yet, even though the change of entropy is consistent, it does not always yield a decrease of the entropy. 
This is demonstrated by 
\begin{example}\label{es:entropy}
    Let $u(x) = x^2$ with Bernstein reconstruction $B_N[u](x) = x^2 + \frac{1}{N}x(1-x)$. 
    For the usual $L^2$-entropy $U(u)=u^2$, we have 
    \begin{equation}
        \begin{aligned}
            \int_0^1 U\left( u(x)\right) \d x & = \frac{1}{5}, \\
            \int_0^1 U\left( B_N[u](x) \right) \d x 
                & = \int_0^1 x^4 \d x + \int_0^1 \frac{2}{N}x^3(1-x) + \frac{1}{N^2}x^2(1-x)^2 \d x \\ 
                & = \int_0^1 U(u(x)) \d x + \frac{3N+1}{30N^2}.
        \end{aligned}
    \end{equation}
    Thus, the total amount of entropy is increased by applying the Bernstein procedure. 
\end{example}
%
In fact, Example \ref{es:entropy} is no exceptional case. 
We note that for every continuous and convex function $u$, the sequence of Bernstein reconstructions will 
converge to $u$ from above, i.\,e. 
\begin{equation}
    B_{N}[u](x) \geq B_{N+1}[u](x) \geq u(x) 
\end{equation} 
holds for all $N\geq 1$; see for instance \cite[Theorem 7.1.8 and 7.1.9]{phillips2003interpolation}. 
Hence, for $u \geq 0$, the $L^2$-entropy will be increased by the Bernstein procedure. 
Yet, we can further prove that the change of the total amount of entropy by the Bernstein reconstruction is consistent 
and vanishes for increasing $N$.
\begin{theorem}\label{thm:entropy}
  Let $U \in C^1$ be a convex entropy function and let $u$ be piecewise $C^2$. 
  Then, 
  \begin{equation}
    \lim_{N \to \infty} \int_0^1 U\left( B_N[u](x) \right) \d x 
      =  \int_0^1 U(u(x)) \d x
  \end{equation}
  holds. 
\end{theorem}
\begin{proof}
  Since $U$ is $C^1$ and $u$ is piecewise $C^2$, equation \eqref{eq:Bernstein-approx} yields 
  \begin{equation}
    \lim_{N \to \infty} U(B_N[u](x)) = U(u(x)) 
    \quad \text{ for all } x \in [0,1] \setminus \{ x_1, \dots, x_k \}, 
  \end{equation}
  where $0<x_1< \dots < x_k < 1$ are the points where $u''$ does not exist. 
  Hence, $U \circ B_N[u]$ converges almost everywhere to $U \circ u$. 
  Further, $u$ and all $B_N[u]$ are uniformly bounded, let us say by $m$ and $M$, i.\,e. 
  \begin{equation}
    m \leq u(x), B_N[u](x) \leq M 
  \end{equation}
  for all $x \in [0,1]$. 
  Since $U$ is continuous, $U$ is also bounded on $[m,M]$ and there is a $v^* \in [m,M]$ such that 
  \begin{equation}
    \left| U(v) \right| \leq \left| U(v^*) \right| =: C \quad \forall v \in [m,M]. 
  \end{equation}
  Thus, we have 
  \begin{equation}
    \left| U(B_N[u](x)) \right| \leq C \quad \forall x \in [0,1], \ N \in \N,
  \end{equation}
  and the sequence $\left( U \circ B_N[u] \right)_{N \in \N}$ is uniformly bounded. 
  Finally, Lebesgue's dominated convergence theorem \cite[Chapter 1.3]{evans2018measure} yields 
  \begin{equation}
    \lim_{N \to \infty} \int_0^1 U\left( B_N[u](x) \right) \d x 
      =  \int_0^1 U(u(x)) \d x 
  \end{equation}
  and therefore the assertion. 
\end{proof}
Note that in the proposed Bernstein procedure, the Bernstein reconstruction is always computed for $u \in 
\mathbb{P}_N$. 
Yet, Theorem \ref{thm:entropy} holds for general $u$ which are piecewise $C^2$.

\subsection{Total Variation} 
\label{sub:TV}

A fundamental property of the (exact) solution of a scalar hyperbolic conservation law \eqref{eq:cl}, 
assuming the initial data function $u_0(x)=u(0,x)$ has bounded variation, is that 
\cite{lax1973hyperbolic,toro2013riemann} 
\begin{enumerate}
    \item 
    no additional spatial local extrema occur. 
    
    \item 
    the values of local minima do not decrease and the values of local maxima do not increase. 
\end{enumerate}
As a consequence, the \emph{total variation (TV)} of the solution, 
\begin{equation}
    TV(u(t,\cdot)) = \sup_{J \in \N, \ 0=x_0 < \dots < x_{J} = 1} \sum_{j=0}^{J-1} \left| u(t,x_{j+1}) - 
u(t,x_j) \right|, 
\end{equation}
is a non-increasing function in time, i.\,e. 
\begin{equation}
    TV(u(t_2,\cdot)) \leq TV(u(t_1,\cdot)) 
\end{equation} 
for $t_2 \geq t_1$. 
In the presence of (shock) discontinuities, in fact, the TV typically decreases \cite{harten1984class}. 
We are thus interested in designing shock capturing methods which mimic this behaviour of being 
\emph{TV diminishing (TVD)}. 
The proposed Bernstein procedure is now shown to fulfil the TVD property in the sense that the Bernstein 
reconstruction $B_N[u]$ of a function $u$ has a reduced (or the same\footnote{In all numerical tests, we actually 
observed the TV to decrease}) TV, i.\,e.  
\begin{equation}
    TV(B[u]) \leq TV(u)
\end{equation} 
holds for the Bernstein reconstruction \eqref{eq:Bernstein-reconstruction}. 
\begin{theorem}\label{thm:TVD}
    Let $B_N[u] \in \mathbb{P}_N$ be the Bernstein reconstruction of a function $u:[0,1] \to \R$. 
    Then, the TV of $B_N[u]$ is less or equal to the TV of $u$, and the Bernstein procedure is TVD. 
\end{theorem} 
\begin{proof}
    Since $B_N[u] \in \mathbb{P}_N$ and by consulting \eqref{eq:derivative}, we have 
    \begin{align}
        TV(B_N[u]) 
            & = \int_0^1 \left| B_N'(x) \right| \d x \\ 
            & \leq N \sum_{n=0}^{N-1} \left| u\left( \frac{n-1}{N} \right) - u\left( \frac{n}{N} \right) \right| 
               \binom{N-1}{n} \int_0^1 x^n (1-x)^{N-n-1} \d x,  
    \end{align}
    where the integrals are given by 
    \begin{equation}
        \int_0^1 x^n (1-x)^{N-n-1} \d x = N^{-1} \binom{N-1}{n}^{-1}. 
    \end{equation} 
    Thus, inequality 
    \begin{align}
        TV(B_N[u]) 
            \leq \sum_{n=0}^{N-1} \left| u\left( \frac{n-1}{N} \right) - u\left( \frac{n}{N} \right) \right| 
            \leq TV(u) 
    \end{align} 
    follows, and therefore the assertion. 
\end{proof}

\subsection{Monotone (Shock) Profiles} 
\label{sub:monotone}

Let $u:[0,1] \to \R$ be a piecewise smooth function with single discontinuity at $x_1 \in (0,1)$, 
representing a shock profile in a troubled element. 
It is desirable for the polynomial approximation of such a function to not introduce new (artificial) local 
extrema. 
Yet, typical polynomial approximations, such as interpolation and (pseudo) projections, 
are doing so by the Gibbs--Wilbraham phenomenon \cite{hewitt1979gibbs}. 
The Bernstein reconstruction, however, has been proved to not feature such spurious oscillations. 
This can, for instance, be noted from 
\begin{theorem}[Theorem 1.9.1 in Lorentz \cite{lorentz2012bernstein}]\label{thm:Gibbs}
    Suppose that $u$ is bounded in $[0,1]$ and let $L^+,L^-$ respectively denote the right and left upper limits and 
$l^+,l^-$ the right and left lower limits of $u$ at a point $x$.
    Then 
    \begin{equation}
        \frac{1}{2} \left( l^+ + l^- \right) 
            \leq \liminf_{N \to \infty} B_N[u](x) 
            \leq \limsup_{N \to \infty} B_N[u](x) 
            \leq \frac{1}{2} \left( L^+ + L^- \right).
    \end{equation}
\end{theorem} 
Note that Theorem \ref{thm:Gibbs} is fairly general. 
The absence of the Gibbs--Wilbraham phenomenon (without taking convergence into account) could have been noted by Lemma 
\ref{lem:boundary_pres} already. 
It should be stressed that we can not just rule out spurious (Gibbs--Wilbraham) oscillations for the 
Bernstein reconstruction of discontinuous (shock) profiles, but we are further able to ensure the preservation of 
monotonicity. 
Let $u$ be a monotonic increasing (and possibly discontinuous) function on $[0,1]$. 
Then, the following lemma ensures that the Bernstein reconstruction $B_N[u]$ is monotonic increasing as well. 
\begin{lemma} 
    Let $u:[0,1] \to \R$ be monotonic increasing and let $B_N[u]$ denote the Bernstein reconstruction of $u$. 
    Then, $B_N[u]$ is monotonic increasing as well. 
\end{lemma} 
\begin{proof}
    Note that for monotonic increasing $u$, we have 
    \begin{equation}
        \Delta u \left( \frac{n}{N} \right) 
            = u \left( \frac{n+1}{N} \right) - u \left( \frac{n}{N} \right) 
            \geq 0.
    \end{equation} 
    Thus, by consulting \eqref{eq:derivative}, inequality
    \begin{equation}
        B_N'[u](x) \geq 0 
    \end{equation} 
    follows and as a result also the assertion. 
\end{proof} 
Note that the same result holds for monotonic decreasing (shock) profiles.
\section{Numerical Tests}
\label{sec:tests}

In this section, we test the Bernstein procedure incorporated into a nodal collocation-type discontinuous Galerkin 
finite element method (DGFEM) as described in \cite{hesthaven2007nodal}. 
For time integration, we have used the explicit TVD/SSP-RK method of third order using three stages 
(SSPRK(3,3)) given in \cite{gottlieb1998total} by Gottlieb and 
Shu: 
Let $u^n$ be the solution at time $t^n$, then the solution $u^{n+1}$ at time $t^{n+1}$ is obtained 
by 
\begin{equation}\label{eq:SSPRK33}
\begin{aligned}
  u^{(1)} & = u^n + \Delta t L\left( u^n \right), \\ 
  u^{(2)} & = \frac{3}{4} u^n + \frac{1}{4} u^{(1)} + \frac{1}{4} \Delta t L\left( u^{(1)} \right), 
\\ 
  u^{n+1} & = \frac{1}{3} u^n + \frac{2}{3} u^{(2)} + \frac{2}{3} \Delta t L\left( u^{(2)} \right),  
\end{aligned}
\end{equation}
where $L(u)$ is a discretisation of the spatial operator of the DGFEM.
For the time step size, we have used  
\begin{equation}
  \Delta t = C \cdot \frac{|\Omega|}{I (2N+1)^2 \max{|f'(u)|}}
\end{equation}
with $C = 0.1$ and where $\max{|f'(u)|}$ is calculated for all $u$ between $\min_{x \in \Omega} u_0(x)$ and 
$\max_{x \in \Omega} u_0(x)$ in all our numerical tests. 
Following \cite{cockburn1991runge} in parts, four different problems are investigated for which the exact 
solutions can be calculated.
Note that we assume periodic boundary conditions (BCs) in all numerical tests. 
This restriction is utterly unnecessary for the proposed Bernstein procedure and is only made in order to 
compactly provide reference solutions, which refer to the exact entropy solutions. 
Further, for every problem the local Lax--Friedrichs (Rusanov) flux 
\begin{equation}
    \fnum(u_-,u_+) = \frac{f(u_+) + f(u_-)}{2} - \frac{\lambda_{\text{max}}}{2} \left( u_+ - u_- \right)
\end{equation} 
has been applied, where 
$\lambda_{\text{max}} = \max_{u \in [u_-,u_+]}\left| f'(u) \right|$ 
is a locally determined viscosity coefficient based on maximum characteristics speed \cite[Chapter 
12.5]{leveque2002finite}.

\subsection{Linear Advection Equation}
\label{sub:linear}

Let us consider the linear advection equation 
\begin{equation}\label{eq:problem1}
  \partial_t u + \partial_x u = 0
\end{equation} 
on $\Omega=[0,1]$ with periodic BCs and a discontinuous initial condition (IC)  
\begin{equation}
  u_0(x) =
  \left\{ 
    \begin{array}{ccl}
      1 & , & 0.4 \leq x \leq 0.8 \\ 
      0 & , & \text{otherwise} \\
    \end{array}
    \right. .
\end{equation}
By the method of characteristics, the solution at time $t\geq0$ is given by 
\begin{equation}
  u(t,x) = u_0(x-t).
\end{equation}
The linear advection equation is the simplest PDE that can feature discontinuous solutions. 
Thus, the (shock capturing) method can be observed in a well-understood setting, isolated from nonlinear effects.
Yet, the linear advection equation provides a fairly challenging example. 
Similar to contact discontinuities in Euler's equations, discontinuities are not self-steeping, i.\,e.\ once such a discontinuity is smeared by the method, it can not be recovered to its original sharp shape 
\cite{kuzmin2006flux}.

\subsubsection{Parameter Study}
\label{subsub:linear_param}

We start by investigating the ramp parameter $\kappa \in (0,1)$, which goes into the PA sensor \eqref{eq:param-fun} and 
steers the Bernstein procedure.   
%
%
Figure \ref{fig:linear_param} demonstrates the effect of the ramp parameter $\kappa \in (0,1)$ on the results produced 
by the Bernstein procedure for a linear advection equation with discontinuous IC. 
The IC has thereby been evolved over time until $t=1$. 
We observe the Bernstein procedure to be fairly robust w.\,r.\,t.\ the ramp parameter $\kappa$. 
On coarse meshes, such as $I=10$ and $I=20$, slight differences can be observed between different parameter values. 
Yet, these differences become less significant when the mesh is refined. 
As mentioned in Remark \ref{rem:param_tuning}, the tuning of the ramp parameter has been done quickly on a coarse mesh 
of $I=10$ elements for all problems presented.

\subsubsection{Comparison with Usual Filtering in DG Methods}
\label{subsub:linear_comp}

Next, we investigate the approximation properties of a DGFEM enhanced with the Bernstein procedure for the linear 
advection equation with discontinuous IC. 
Figure \ref{fig:linear_comp_T=1} illustrates the results at time $t=1$. 
%
%
Further, we compare our results with the DGFEM without any filtering and with a usual filtering technique, where the DG 
solution $u \in \mathbb{P}_N(\Omega_i)$ is replaced by its mean, 
\begin{equation}
  \overline{u} = \int_{\Omega_i} u(x) \intd x,
\end{equation}
in a troubled element $\Omega_i$. 
Troubled elements are detected by a critical value $S(u) \geq 1$, see \eqref{eq:PA-sensor}. 
Note that the usual filtering is therefore expected to be applied in less elements than the Bernstein procedure, since 
the Bernstein procedure is already activated for $S(u) > \kappa$, see \eqref{eq:param-fun}. 
Yet, we still observe the usual filtering to smear the numerical solution around discontinuities considerably. 
Over time, this smearing yields the numerical solution to nearly become constant. 
This can be observed in Figure \ref{fig:linear_comp_T=10}, where the results are further evolved in time until 
$t=10$.
%
%
At the same time, the results of the Bernstein procedure remain in their relatively sharp shape, even near 
discontinuities. 
Yet, no oscillations are observed, in contrast to the results produced by the DGFEM without any filtering. 
It should be stressed that this test case is especially challenging for the usual filtering by mean values, since the 
initial discontinuities have traveled trough the domain several times until $t=10$ is reached. 
The following tests might provide a fairer comparison. 
An extensive smearing of the numerical solution by the usual filtering by mean values, especially compared to the 
Bernstein procedure, is always observed to some extent, however.  
\begin{remark}
In Figure \ref{fig:linear_T=1_comp_N3_I80} (and Figure \ref{fig:buckley-leverett_T=025_I=40_N=3}) we note 
stronger oscillations for the DGFEM with mean value filtering than without any filtering. 
Typically, this is not excepted. 
A reason for this behaviour might be that the mean value filtering, when activated by the above sensor, is neither 
ensured to be TVD nor to preserve the relation between boundary values at element interfaces (while $u_- < u_+$ holds 
for the original approximation, mean value filtering in one or both elements connected by the interface might result 
in a reverse relation $\overline{u}_- > \overline{u}_+$). 
Such behaviour could be prevented by local projection limiters which are steered by modified minmod function; see
\cite{cockburn1991runge,cockburn1989tvb}.
\end{remark}

\subsection{Invicid Burgers' Equation}
\label{sub:Burgers} 

Let us now consider the nonlinear invicid Burgers' equation 
\begin{equation}\label{eq:problem2}
  u_t + \left( \frac{u^2}{2} \right)_x = 0 
\end{equation}
on $\Omega=[0,1]$ with smooth IC 
\begin{equation}
  u_0(x) = 1+ \frac{1}{4 \pi} \sin( 2 \pi x )
\end{equation}
and periodic BCs. 
For this problem a shock develops in the solution when the wave breaks at time 
\begin{equation}
  t_b = - \frac{1}{\min_{0 \leq x \leq 1} u_0'(x)} = 2.
\end{equation}
In the subsequent numerical tests, we consider the solution at times $t=2$ and $t=3$.
The reference solutions have been computed using characteristic tracing, solving the implicit equation $u(t,x) = 
u_0(x - tu)$ in smooth regions. 
The jump location, separating these regions, can be determined by the Rankine--Hugoniot condition.
%
%
Figure \ref{fig:burgers_comp_T=2} illustrates the results at time $t=2$, at which the shock waves starts to arise. 
As a consequence, only slight differences are observed at this state. 
In particular, we can observe that the Bernstein procedure does not smear the solution in smooth regions, even when 
steep gradients arise. 
A slight smearing can be observed for the usual filtering by mean values around the location of the arising shock at 
$x=0.5$. 
%
%
Figure \ref{fig:burgers_comp_T=3} illustrates the results at time $t=3$, for which the shock has fully developed and 
has already traveled through the whole domain once. 
While we observe oscillations for the DGFEM without filtering and smearing for usual filtering by mean values, the 
Bernstein procedure still provides sharp profiles for the discontinuous solution.

\subsection{A Concave Flux Function}
\label{sub:non-convex}

Next, we investigate the conservation law 
\begin{equation}\label{eq:problem3}
  u_t + \left( u(1-u) \right)_x = 0 
\end{equation}
with a concave flux function $f(u) = u(1-u)$ on $\Omega=[0,2]$, periodic BCs, and a discontinuous IC 
\begin{equation}
  u_0(x) =
  \left\{ 
    \begin{array}{ccl}
      1 & , & 0.5 \leq x \leq 1.5 \\ 
      0 & , & \text{otherwise} \\
    \end{array}
    \right. .
\end{equation}
For this problem a rarefaction wave develops in the solution. 
Figure \ref{fig:linear+burgers_comp_T=05} illustrates the results at time $t=0.5$.
%
%
Here, the DGFEM without any filtering fails to capture the rarefaction wave around $x=1$, yielding a 
wrong (physically unreasonable) weak solution. 
The usual filtering by mean values captures the rarefaction wave, but again smears the solution.
The Bernstein procedure is also  able to capture the rarefaction wave, while providing notably sharper profiles. 
Yet, we can observe some remaining slight oscillation for the Bernstein procedure in some cases. 
Remember that the Bernstein procedure is ensured to preserve bounds. 
Thus, the oscillations probably have been caused by the Bernstein procedure getting activated after 
oscillations have already been developed in the numerical solution by the DGFEM. 
Future work will investigate the possibility of using other shock sensors to steer the proposed Bernstein 
procedure in greater detail. 
These sensors might be activated once the TV increases over time or once the numerical solution 
starts to violate certain bounds.

\subsection{The Buckley--Leverett Equation}
\label{sub:Buckley-Leverett}

Finally, we consider the Buckley--Leverett equation 
\begin{equation}\label{eq:problem4}
  u_t + \left( \frac{u^2}{u^2 + (1-u)^2} \right)_x = 0 
\end{equation}
with a non-convex flux function $f(u) = \frac{u^2}{u^2 + (1-u)^2}$ on $\Omega=[0,2]$, 
periodic BCs, and a discontinuous IC 
\begin{equation}
  u_0(x) =
  \left\{ 
    \begin{array}{ccl}
      1 & , & 0.5 \leq x \leq 1.5 \\ 
      0 & , & \text{otherwise} \\
    \end{array}
    \right. .
\end{equation}
The Buckley--Leverett equation is often used to describe an immiscible displacement process, such as the displacement 
of oil by water \cite[Chapter 4.2]{randall1992numerical}. 
Due to its nonconvex flux function, the Riemann solution involves a so-called \textit{compound wave} (sometimes also 
referred to as a \textit{composite curve}), which contains a shock discontinuity and a rarefaction wave at the same 
time. 
For a nonlinear system of equations this can arise in any nonlinear field that fails to be genuinely nonlinear; 
see \cite[Chapter 16.1]{toro2013riemann} and \cite{wendroff1972riemann,wendroff1972riemann2,liu1975riemann}. 
The reference solution has been computed using characteristic tracing again.
Figure \ref{fig:buckley-leverett_comp_T=025} illustrates the results at time $t=0.5$.
%
%
For this problem, the DGFEM without filtering provides fairly poor results, polluted by heavy oscillations, in all 
tests. 
In some cases for $N=3$ even the usual filtering by mean values shows some oscillations. 
This problem seems to be related to some miss-identifications of troubled elements by the PA sensor for $N=3$. 
For finer meshes the problem vanishes in case of the usual filtering by mean values. 
Yet, we still observe the Bernstein procedure to also perform relatively poor for $N=3$. 
For $N > 3$, the Bernstein procedure is again observed to provide notably better results than usual filtering by mean 
values. 
We conclude that the Bernstein procedure steered by the PA sensor \eqref{eq:param-fun} can only be recommended for 
higher orders $N \geq 4$. 
Yet, similar observations for the usual filtering by mean values indicate that the relatively poor behaviour for $N=3$ 
might be caused by the PA sensor, rather than by the Bernstein procedure itself. 

\section{Concluding Thoughts}
\label{sec:conclusion}

In this work, we have proposed a novel shock capturing procedure for SE approximations for scalar hyperbolic 
conservation laws. 
The procedure is easy to implement and neither increases the computational complexity of the method nor adds 
additional time step or CFL restrictions, as other common shock capturing methods do. 
The procedure essentially consists of going over from the original (oscillatory) polynomial approximation to its 
Bernstein reconstruction in troubled elements. 
Thereby, the Bernstein reconstruction of a function is obtained by applying the (modified) Bernstein operator 
\eqref{eq:mod_Bernstein_operator}. 
This operator has been proved to reduce the TV of the underlying function as well as to preserve monotone (shock) 
profiles. 
Both properties distinguish the shock capturing procedure as especially suitable for scalar conservation laws, 
which come along with a TVD property for their physically reasonable solutions. 
We further note that the modified procedure is not just able to preserve but even to enforce certain 
bounds for the numerical solution, such as positivity. 
Finally, the proposed procedure can be calibrated to the regularity of the underlying solution by using a discontinuity 
sensor which, in this work, is based on comparing the values of PA operators of increasing orders. 
Numerical tests demonstrate that the procedure is able to significantly enhance SE approximations in the presence of 
shocks. 

Future work will focus on the extension of the proposed Bernstein procedure to unstructured meshes in multiple 
dimensions and systems of conservation laws. 
Further, the investigation of other shock sensors would be of great interest. 
These sensors might be activated once the TV increases over time or once the numerical solution 
starts to violate certain bounds.  
\begin{figure}[!htb]
  \centering
  \begin{subfigure}[b]{0.33\textwidth}
    \includegraphics[width=\textwidth]{%
      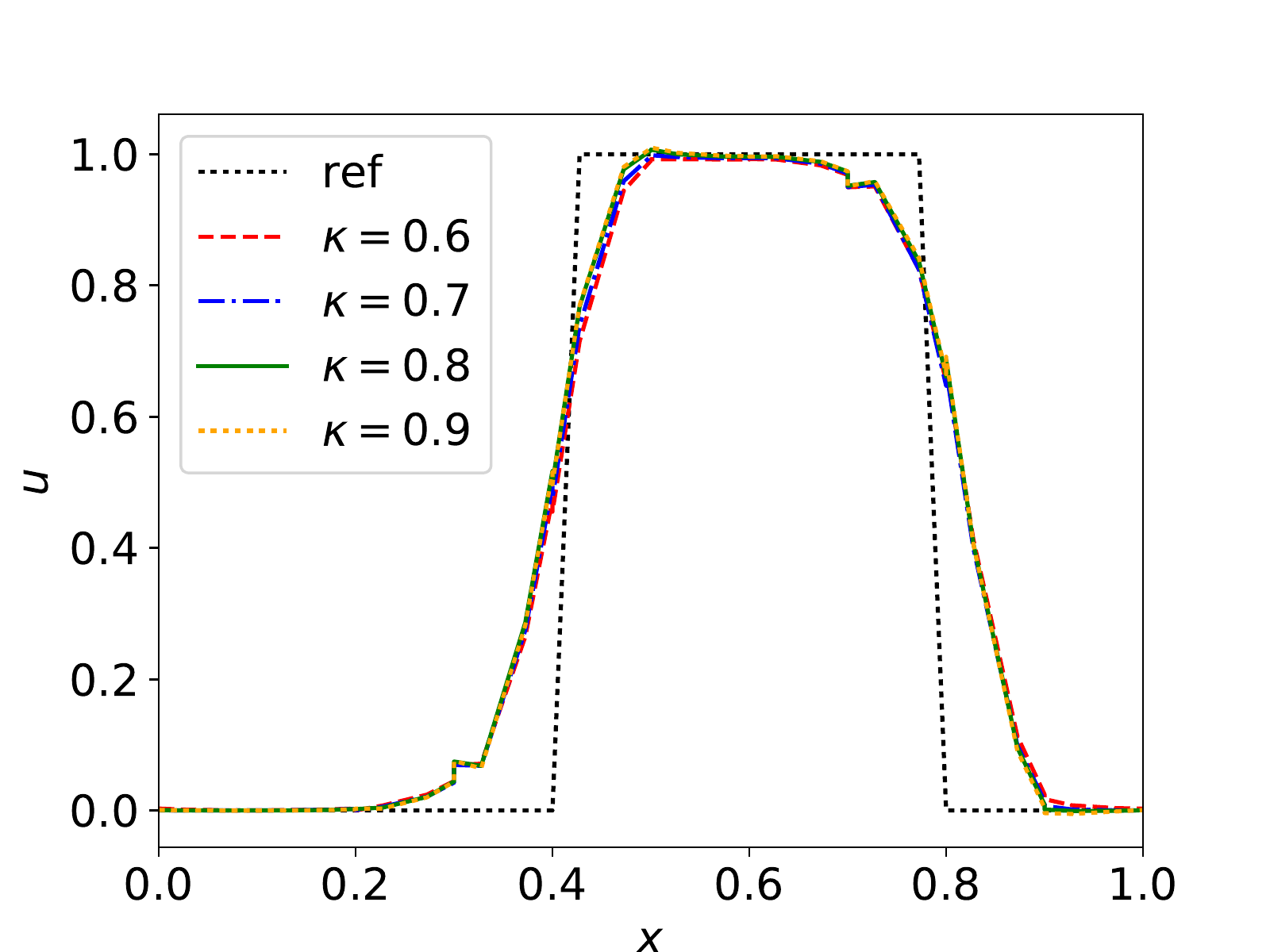}
    \caption{$N=3$ \& $I=10$.}
    \label{fig:linear_param_N3_I10} 
  \end{subfigure}%
  ~ 
  \begin{subfigure}[b]{0.33\textwidth}
    \includegraphics[width=\textwidth]{%
      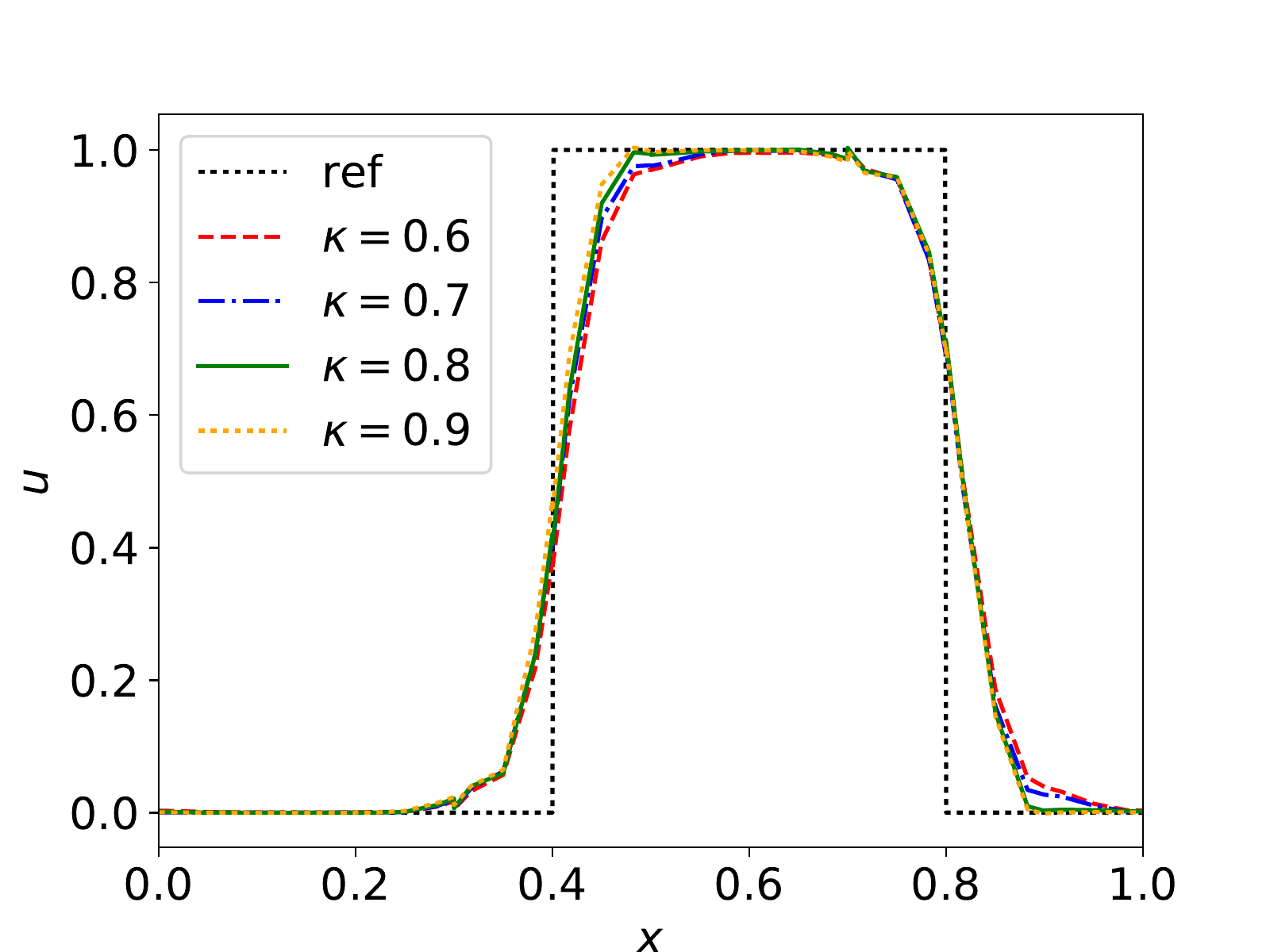}
    \caption{$N=4$ \& $I=10$.}
    \label{fig:linear_param_N4_I10} 
  \end{subfigure}%
  ~ 
  \begin{subfigure}[b]{0.33\textwidth}
    \includegraphics[width=\textwidth]{%
      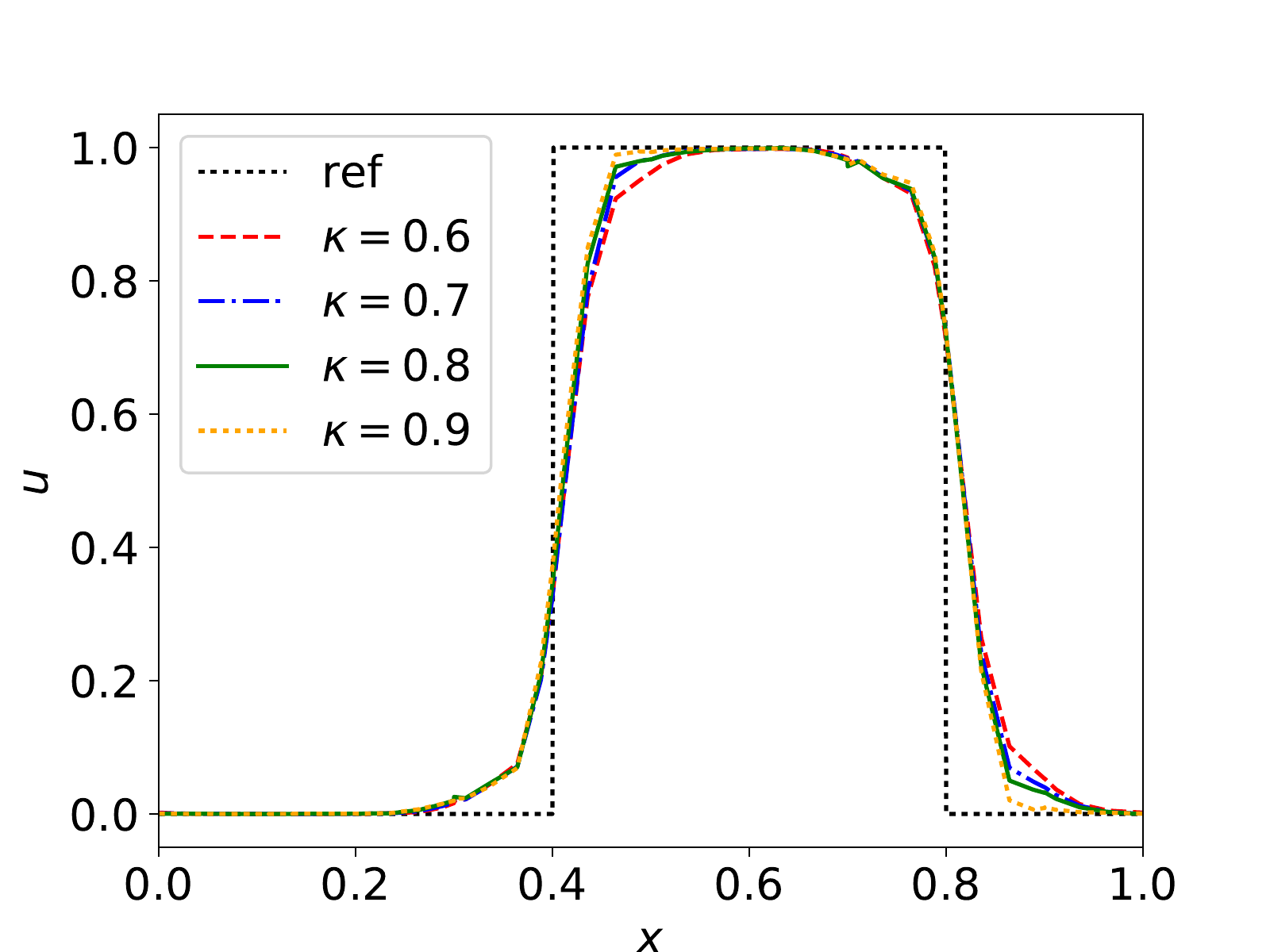}
    \caption{$N=5$ \& $I=10$.}
    \label{fig:linear_param_N5_I10} 
  \end{subfigure}%
  \\
  \begin{subfigure}[b]{0.33\textwidth}
    \includegraphics[width=\textwidth]{%
      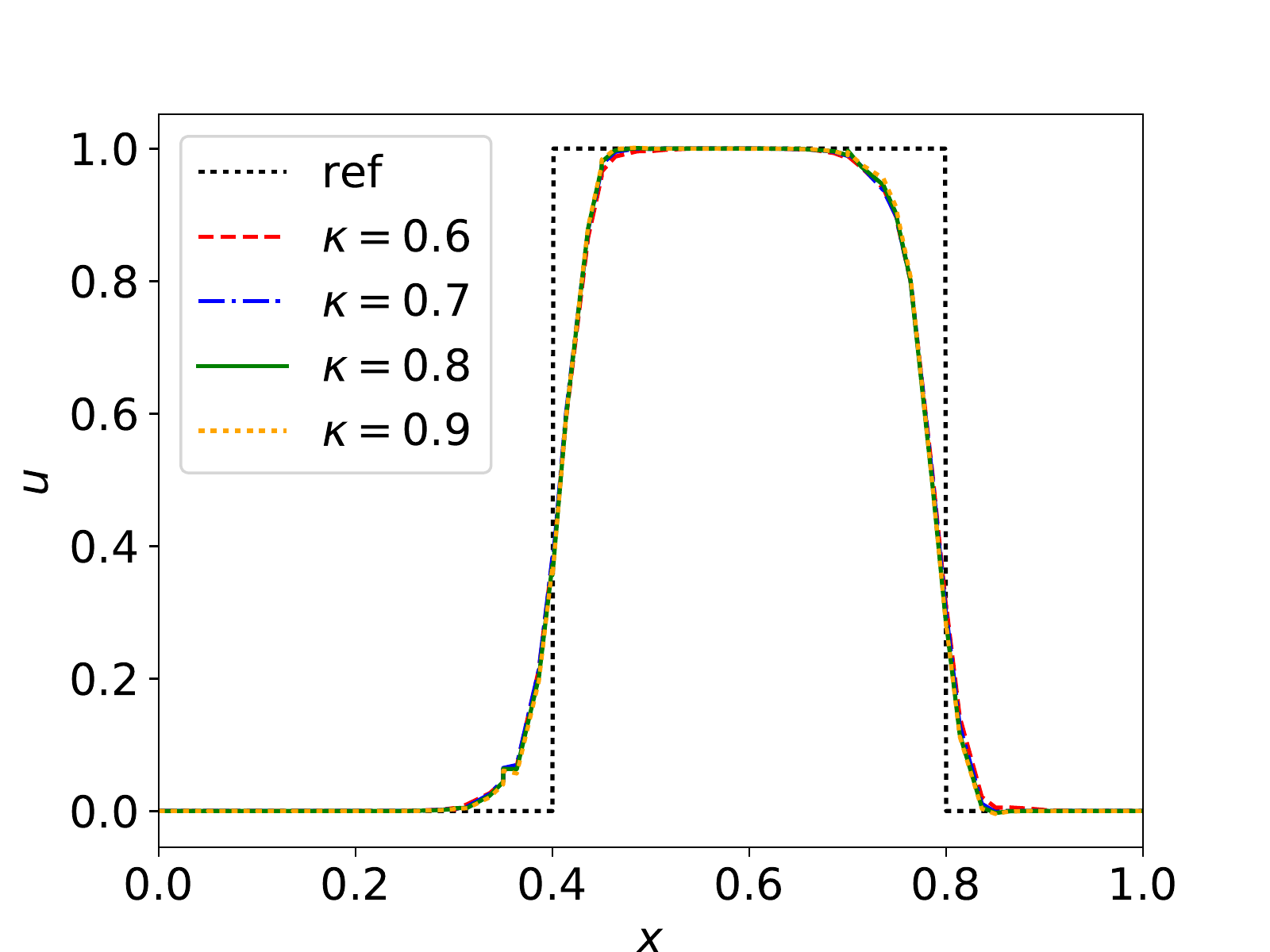}
    \caption{$N=3$ \& $I=20$.}
    \label{fig:linear_param_N3_I20} 
  \end{subfigure}%
  ~ 
  \begin{subfigure}[b]{0.33\textwidth}
    \includegraphics[width=\textwidth]{%
      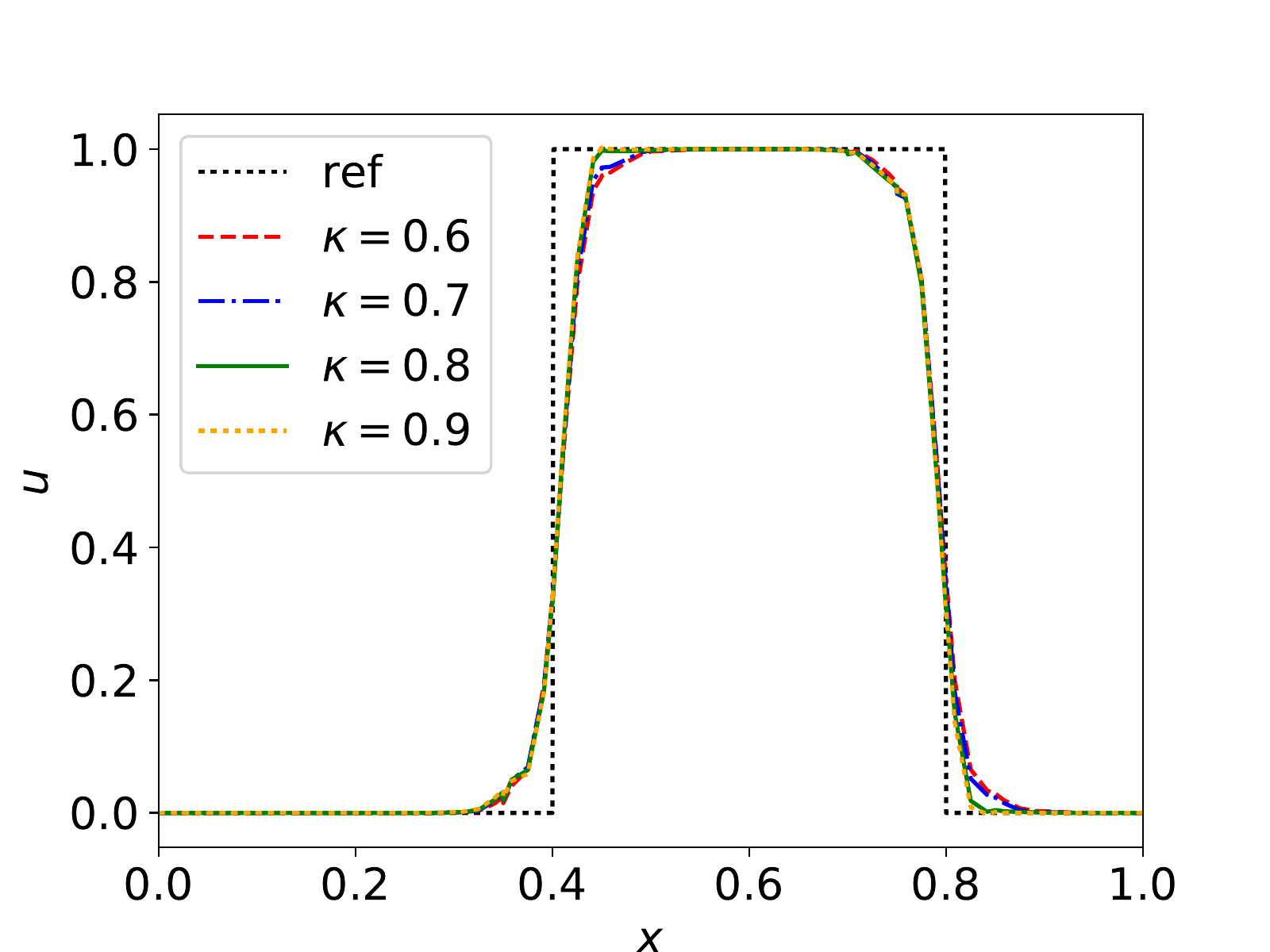}
    \caption{$N=4$ \& $I=20$.}
    \label{fig:linear_param_N4_I20} 
  \end{subfigure}%
  ~ 
  \begin{subfigure}[b]{0.33\textwidth}
    \includegraphics[width=\textwidth]{%
      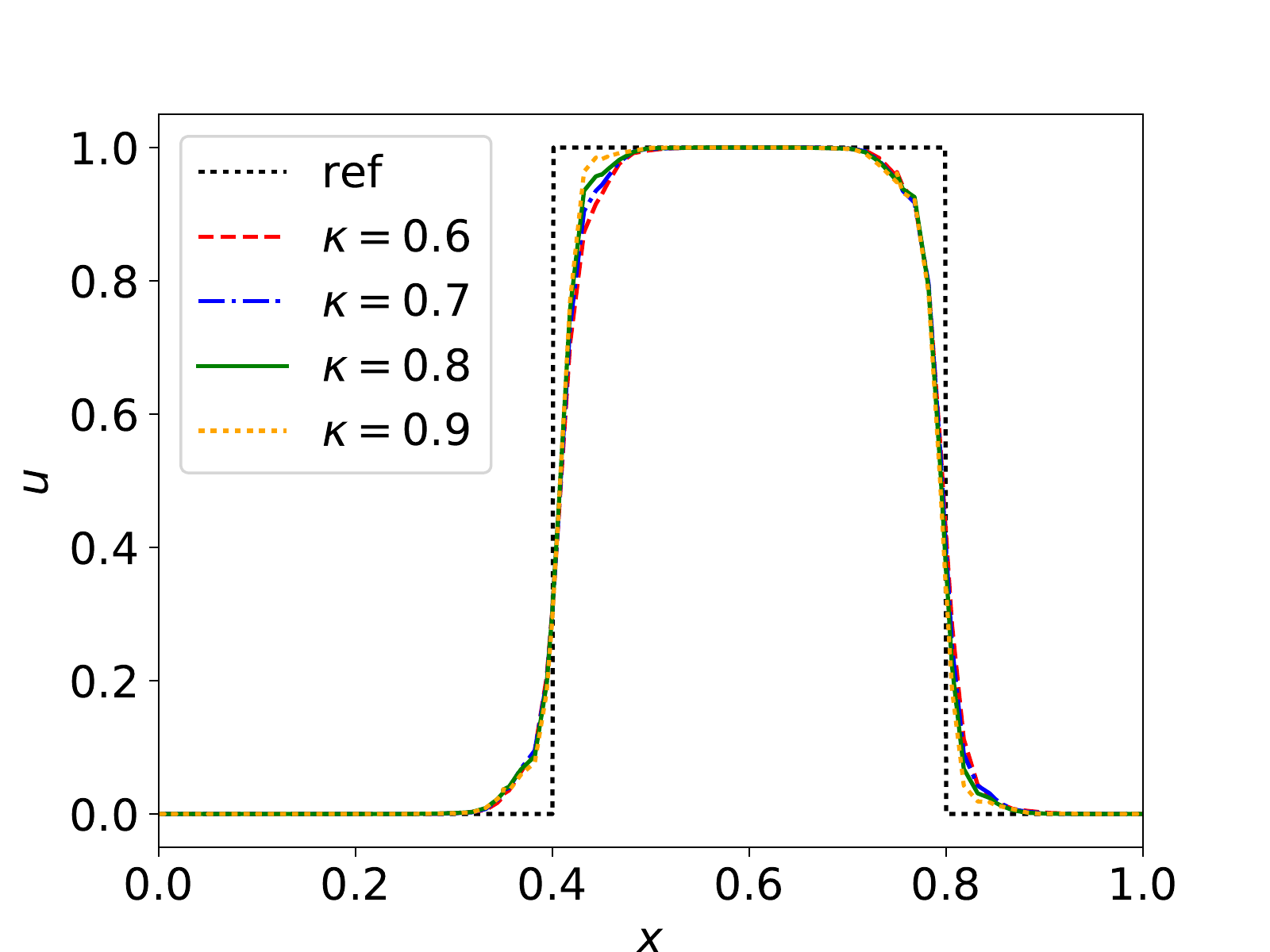}
    \caption{$N=5$ \& $I=20$.}
    \label{fig:linear_param_N5_I20} 
  \end{subfigure}%
  \\
  \begin{subfigure}[b]{0.33\textwidth}
    \includegraphics[width=\textwidth]{%
      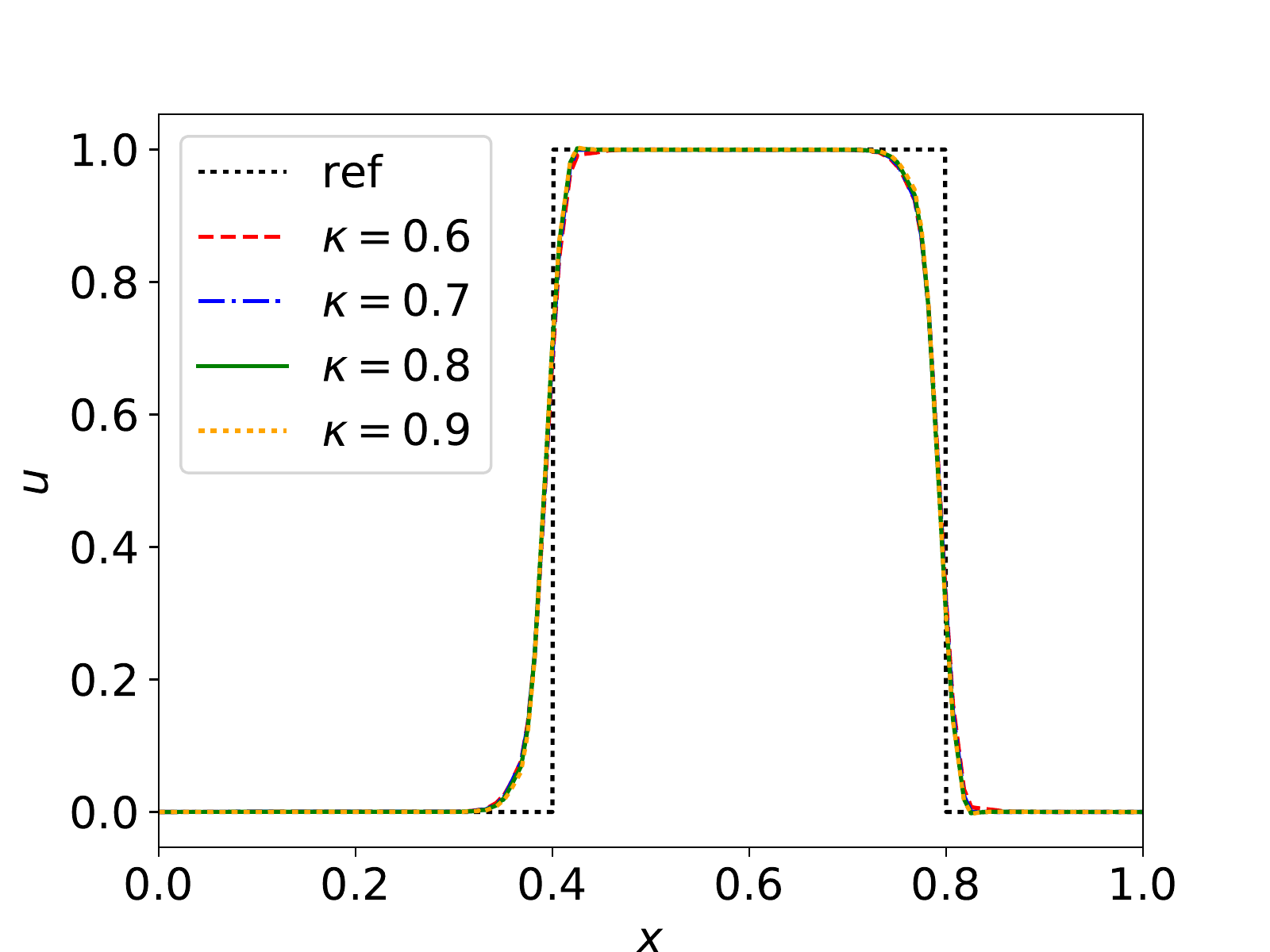}
    \caption{$N=3$ \& $I=40$.}
    \label{fig:linear_param_N3_I40} 
  \end{subfigure}%
  ~ 
  \begin{subfigure}[b]{0.33\textwidth}
    \includegraphics[width=\textwidth]{%
      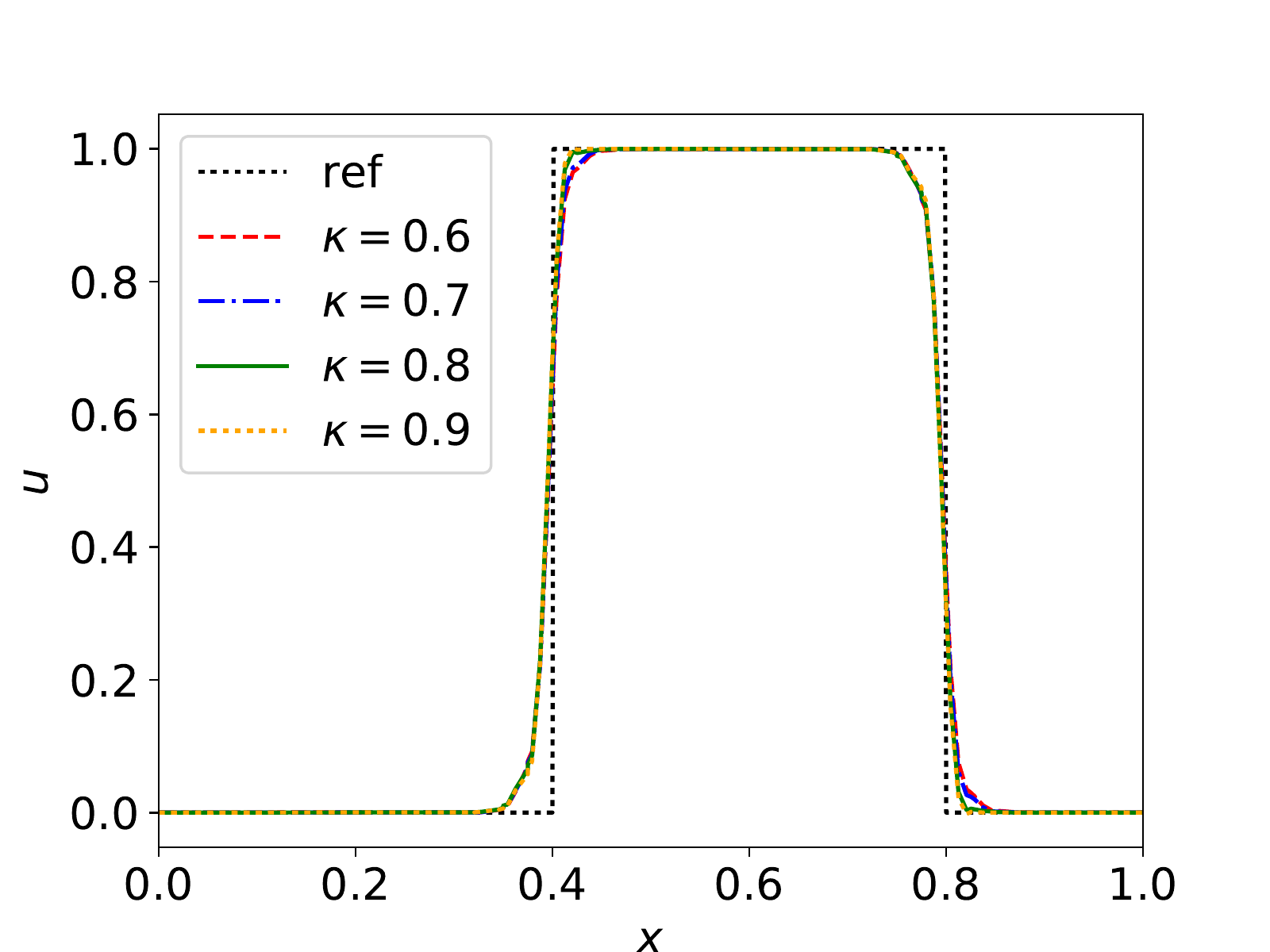}
    \caption{$N=4$ \& $I=40$.}
    \label{fig:linear_param_N4_I40} 
  \end{subfigure}%
  ~ 
  \begin{subfigure}[b]{0.33\textwidth}
    \includegraphics[width=\textwidth]{%
      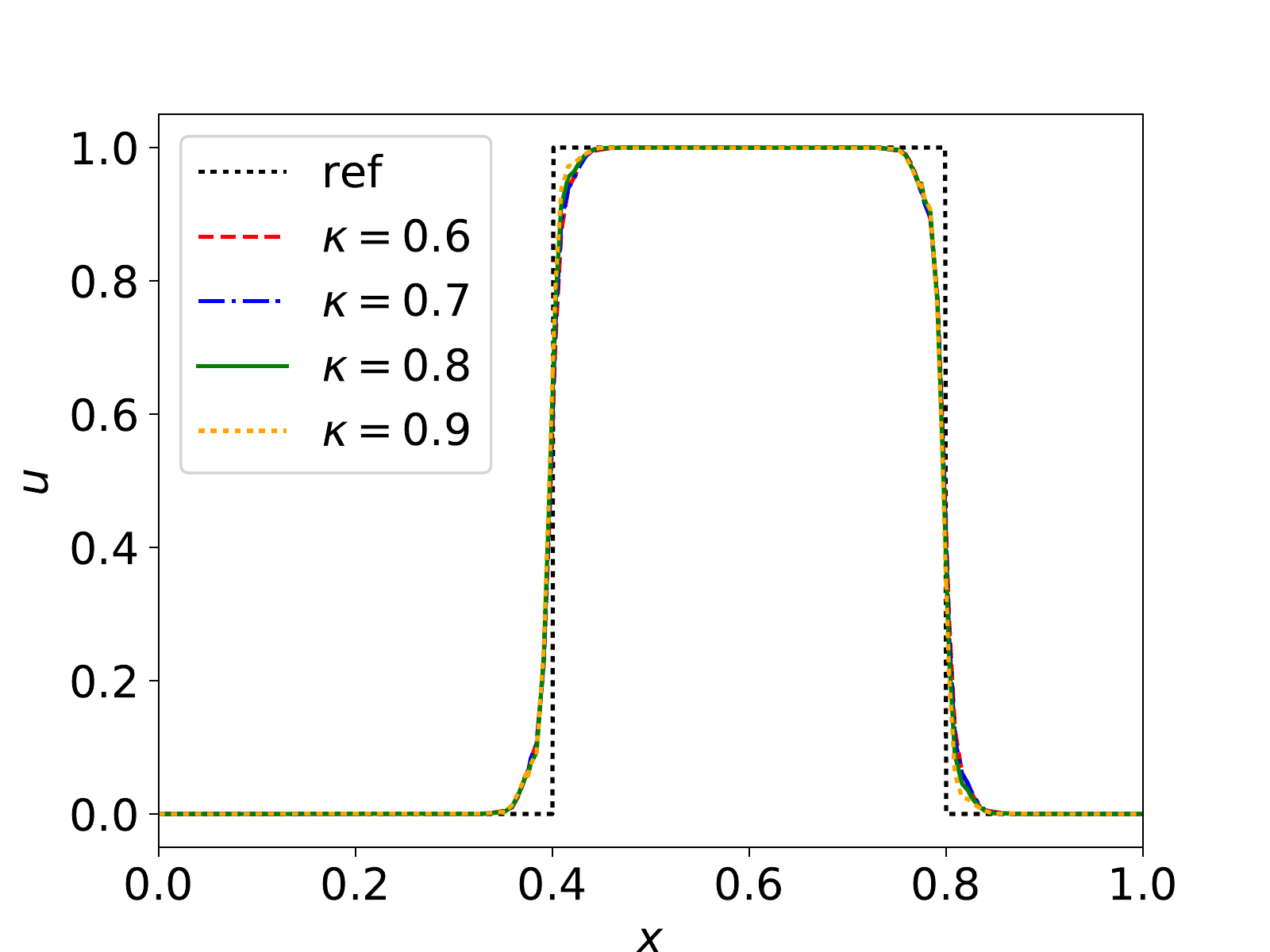}
    \caption{$N=5$ \& $I=40$.}
    \label{fig:linear_param_N5_I40} 
  \end{subfigure}%
  \\
  \begin{subfigure}[b]{0.33\textwidth}
    \includegraphics[width=\textwidth]{%
      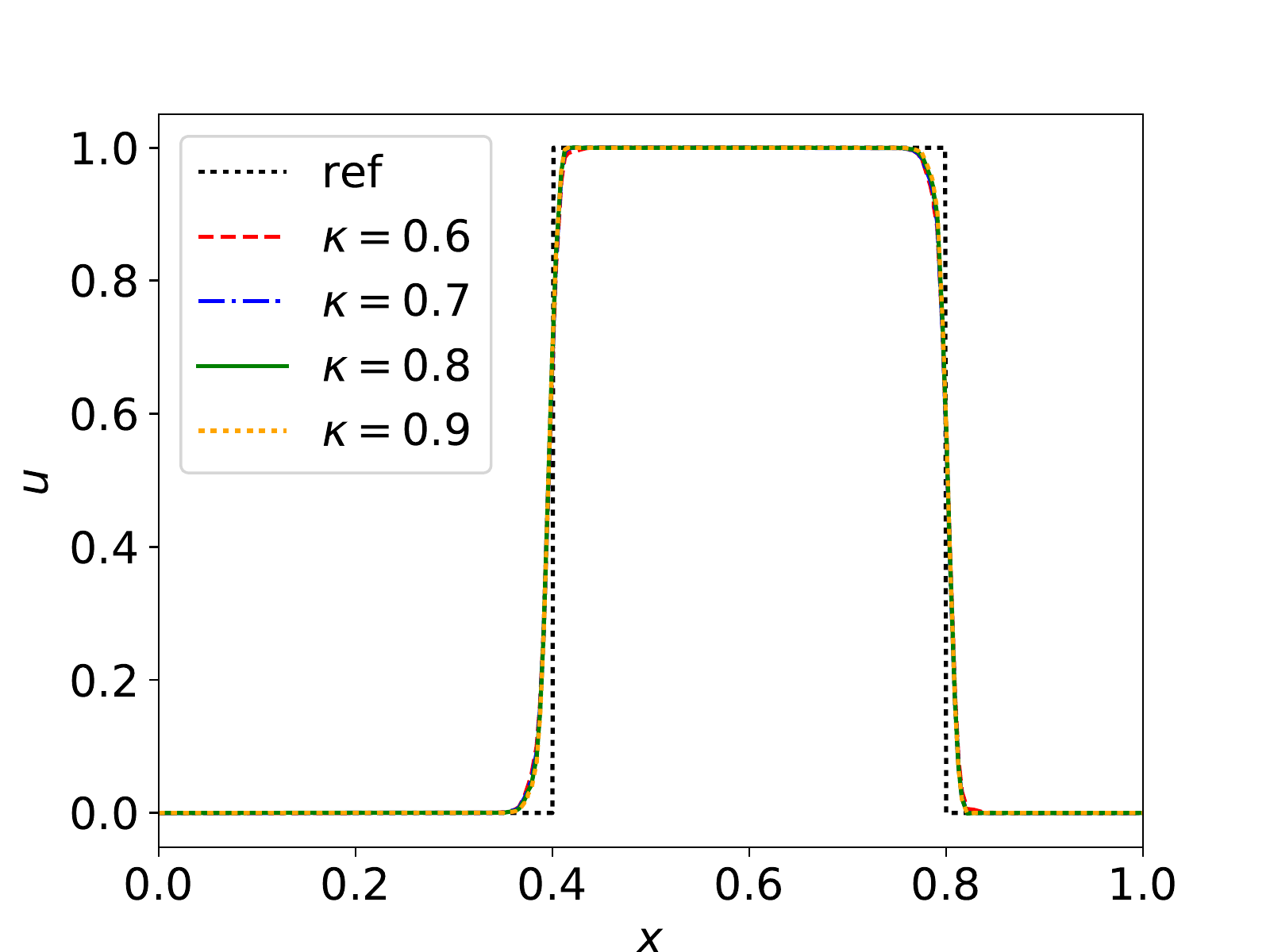}
    \caption{$N=3$ \& $I=80$.}
    \label{fig:linear_param_N3_I80} 
  \end{subfigure}%
  ~ 
  \begin{subfigure}[b]{0.33\textwidth}
    \includegraphics[width=\textwidth]{%
      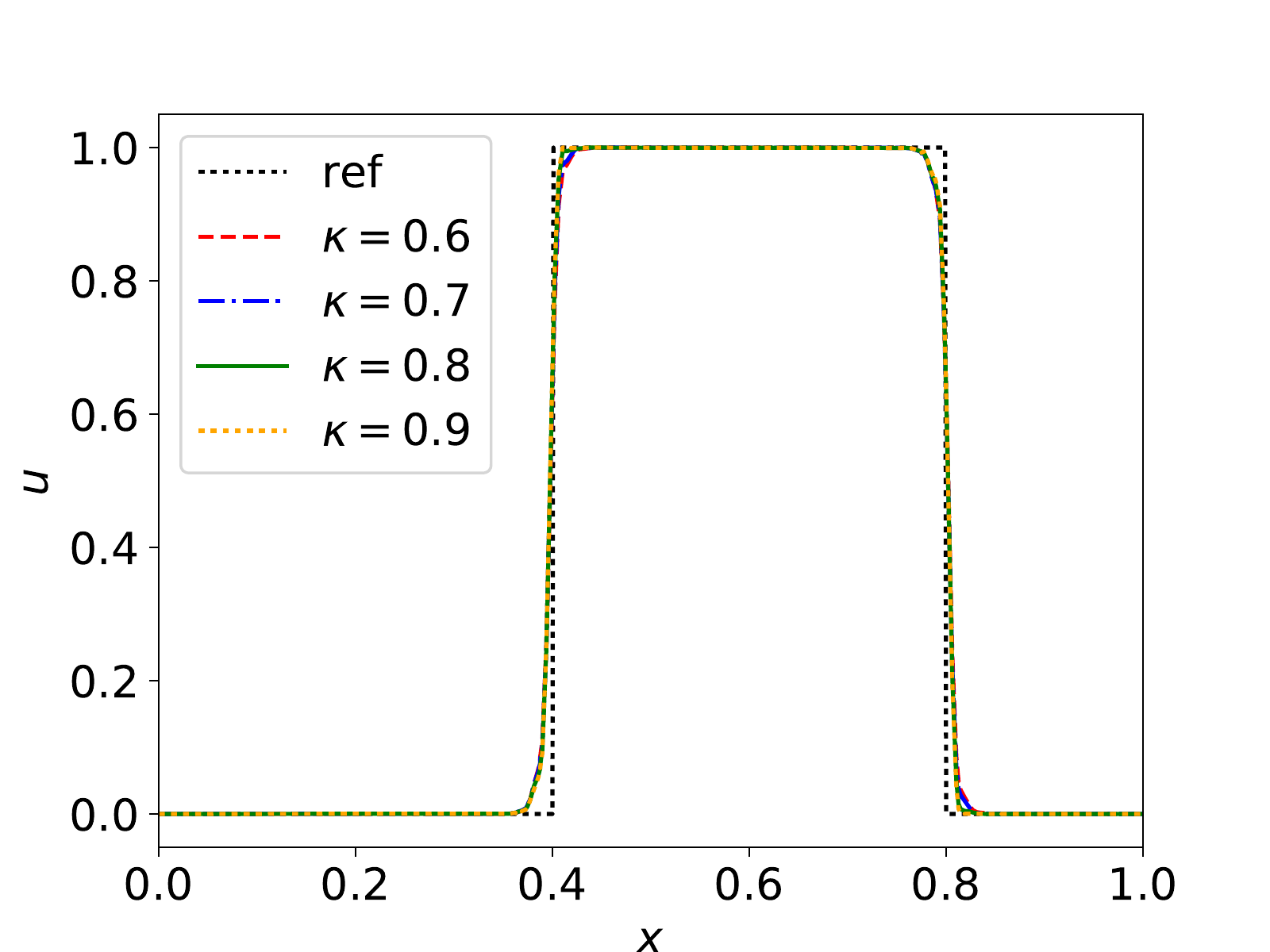}
    \caption{$N=4$ \& $I=80$.}
    \label{fig:linear_param_N4_I80} 
  \end{subfigure}%
  ~ 
  \begin{subfigure}[b]{0.33\textwidth}
    \includegraphics[width=\textwidth]{%
      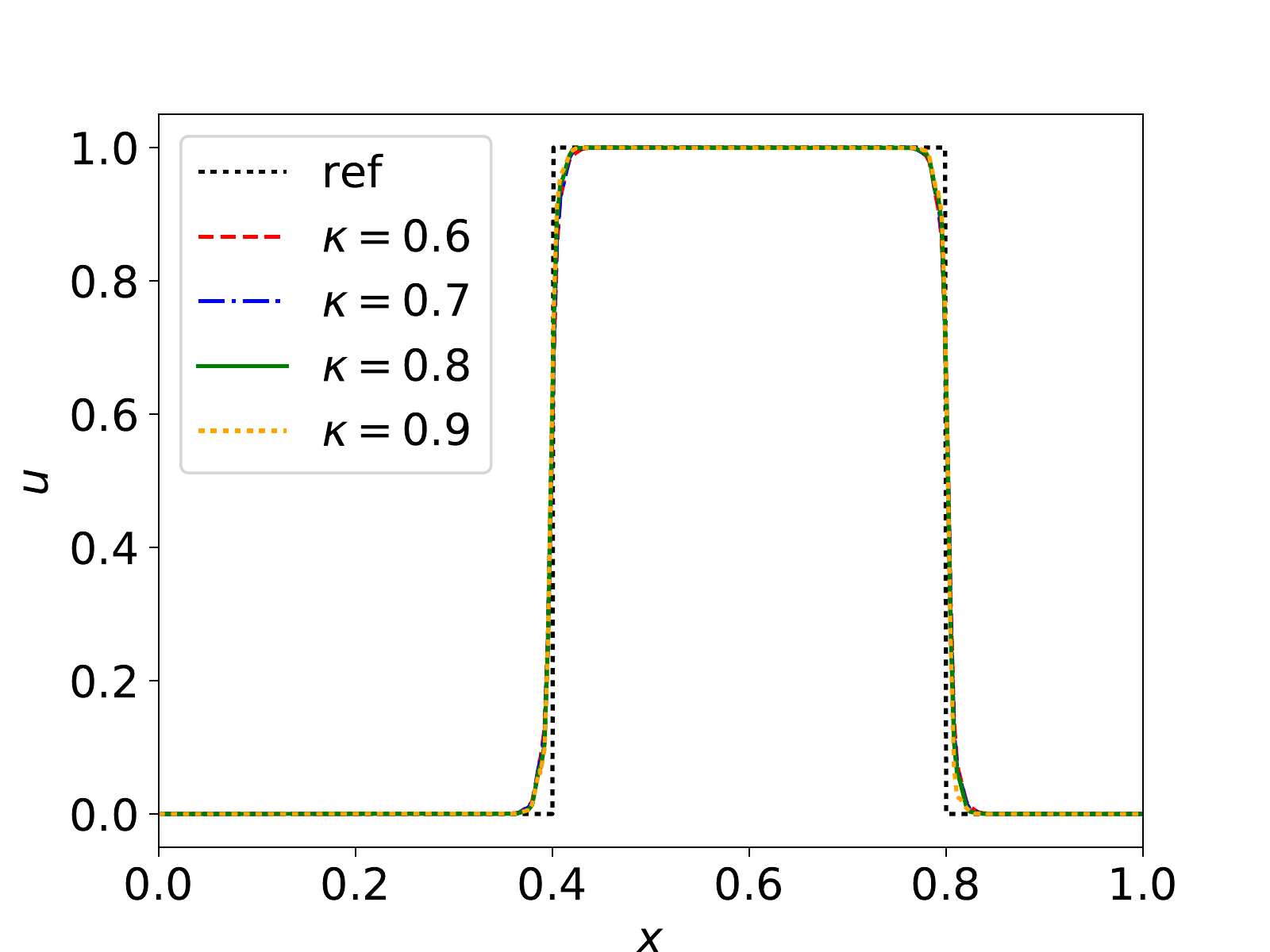}
    \caption{$N=5$ \& $I=80$.}
    \label{fig:linear_param_N5_I80} 
  \end{subfigure}%
  \caption{Parameter study for $\kappa$ for the linear advection equation \eqref{eq:problem1} at time $t=1$.}
  \label{fig:linear_param}
\end{figure} 


\begin{figure}[!htb]
  \centering
  \begin{subfigure}[b]{0.33\textwidth}
    \includegraphics[width=\textwidth]{%
      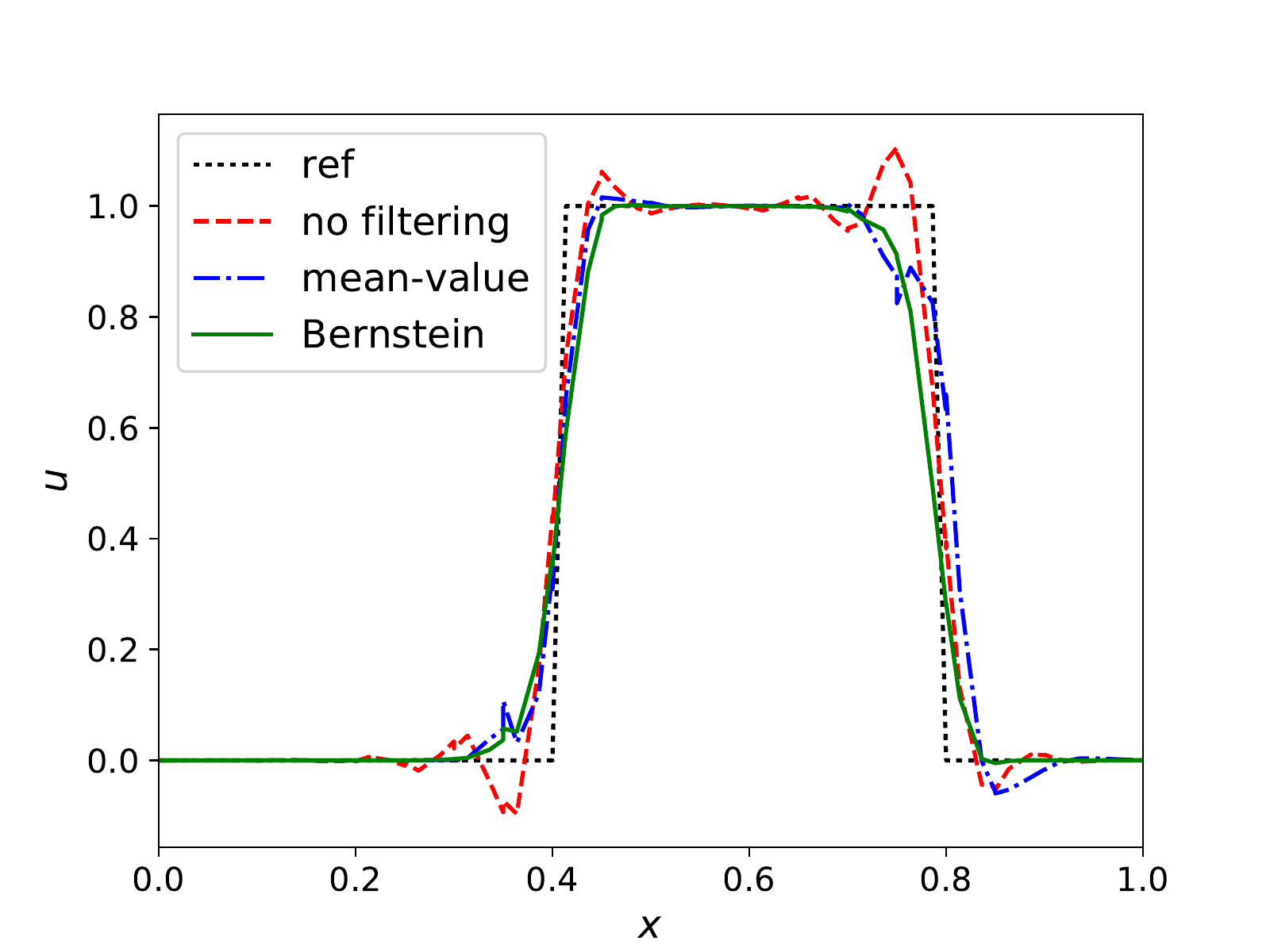}
    \caption{$N=3$ \& $I=20$.}
    \label{fig:linear_T=1_comp_N3_I20} 
  \end{subfigure}%
  ~ 
  \begin{subfigure}[b]{0.33\textwidth}
    \includegraphics[width=\textwidth]{%
      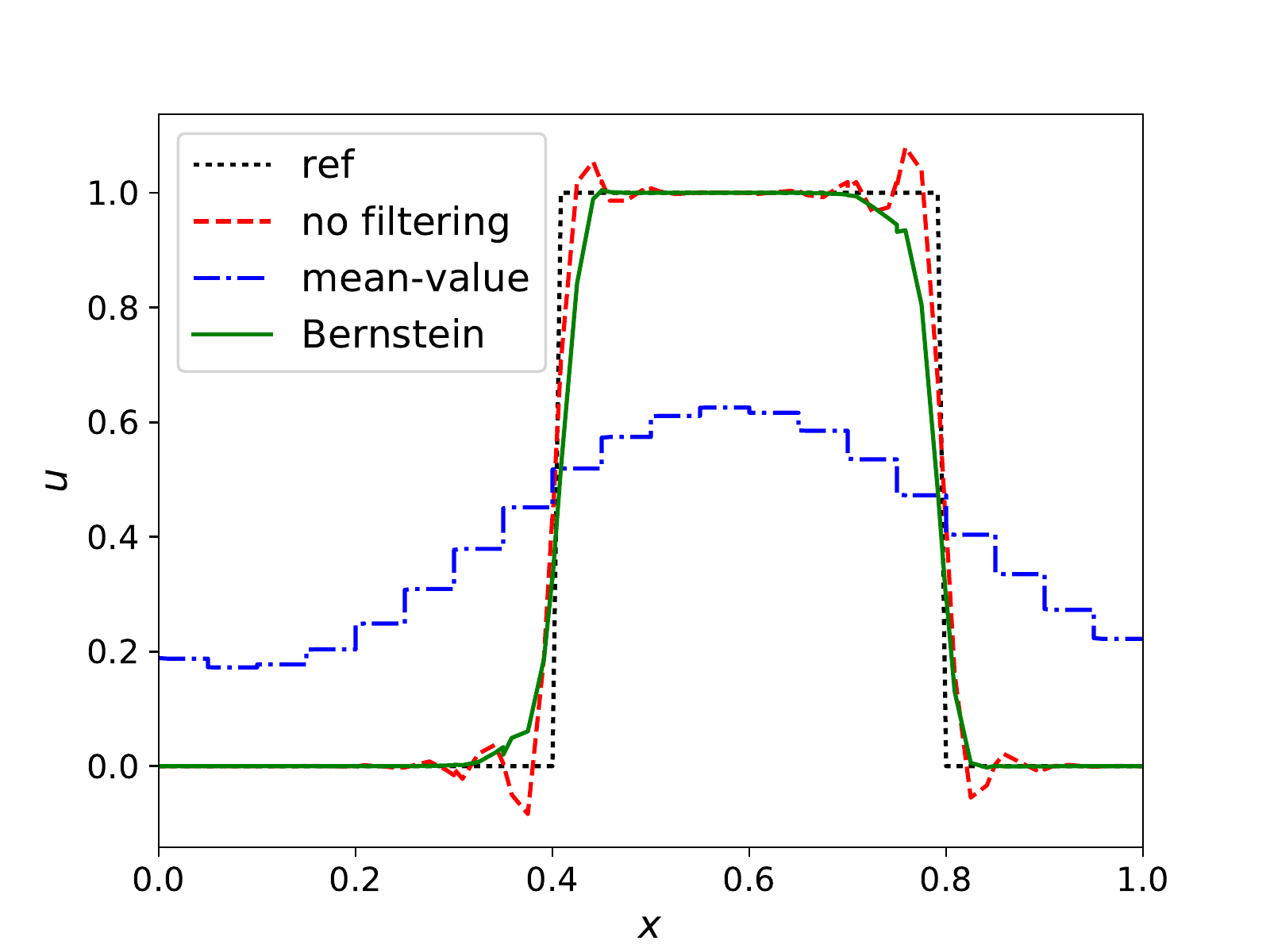}
    \caption{$N=4$ \& $I=20$.}
    \label{fig:linear_T=1_comp_N4_I20} 
  \end{subfigure}%
  ~ 
  \begin{subfigure}[b]{0.33\textwidth}
    \includegraphics[width=\textwidth]{%
      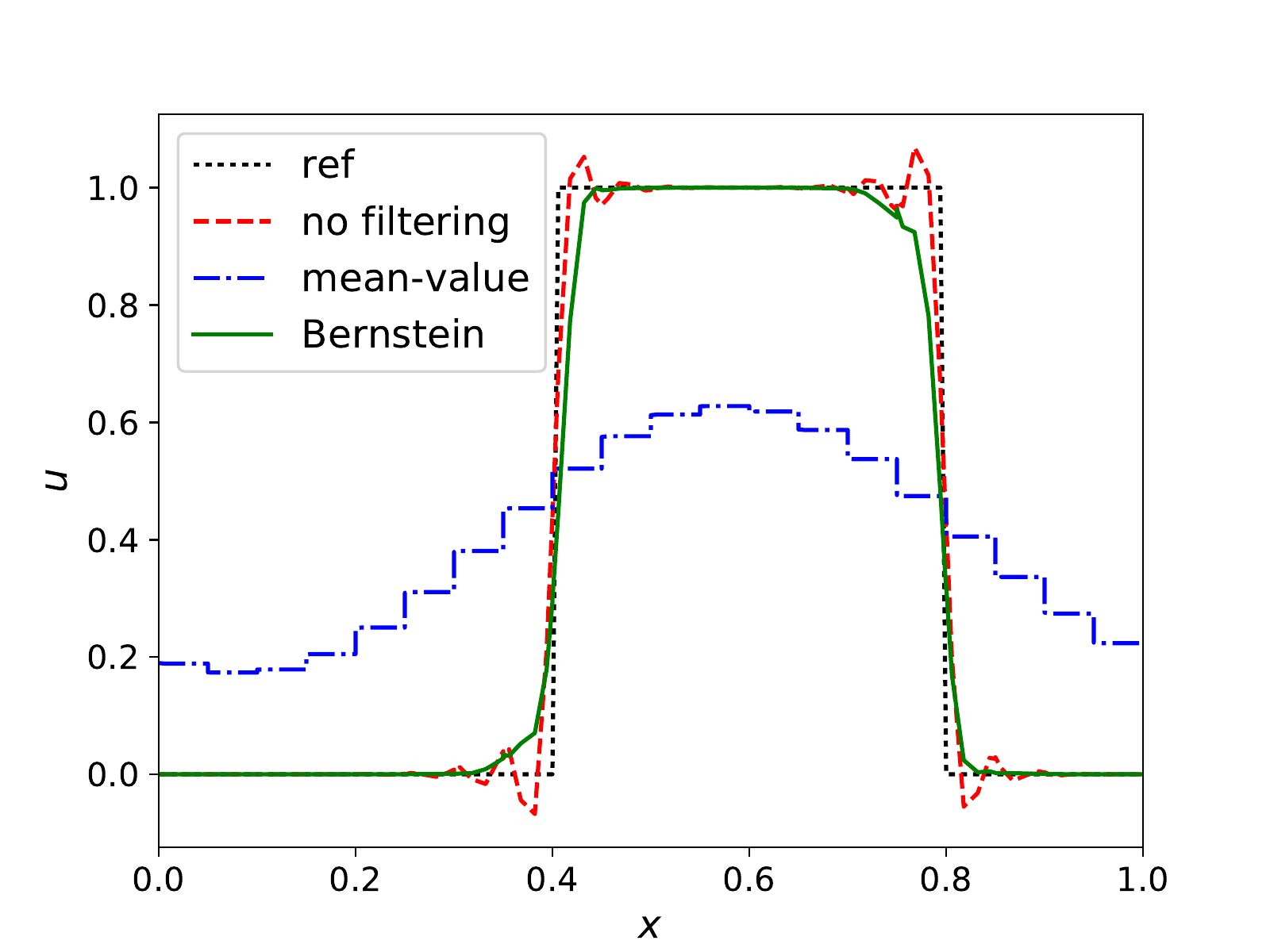}
    \caption{$N=5$ \& $I=20$.}
    \label{fig:linear_T=1_comp_N5_I20} 
  \end{subfigure}%
  \\
  \begin{subfigure}[b]{0.33\textwidth}
    \includegraphics[width=\textwidth]{%
      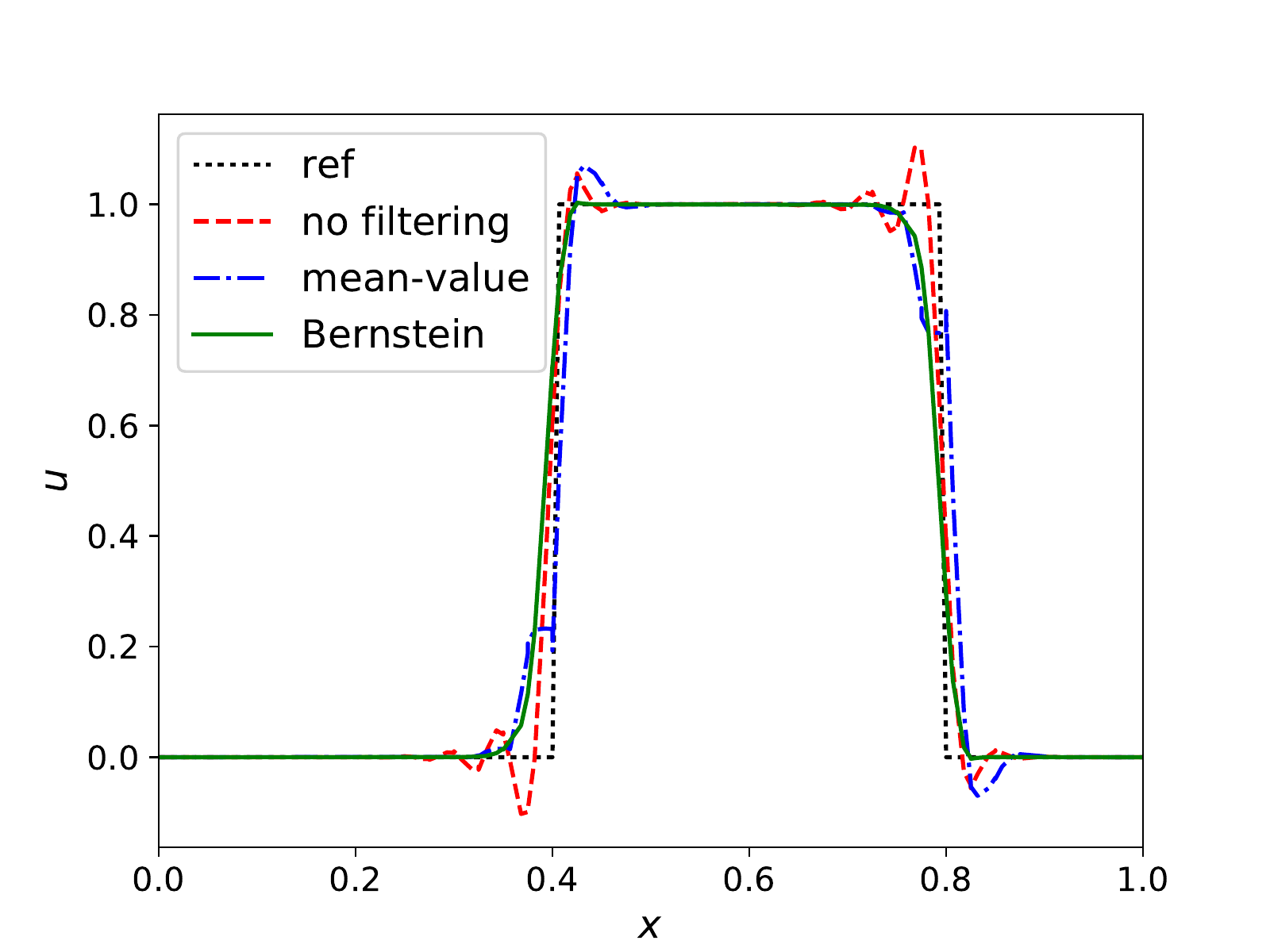}
    \caption{$N=3$ \& $I=40$.}
    \label{fig:linear_T=1_comp_N3_I40} 
  \end{subfigure}%
  ~ 
  \begin{subfigure}[b]{0.33\textwidth}
    \includegraphics[width=\textwidth]{%
      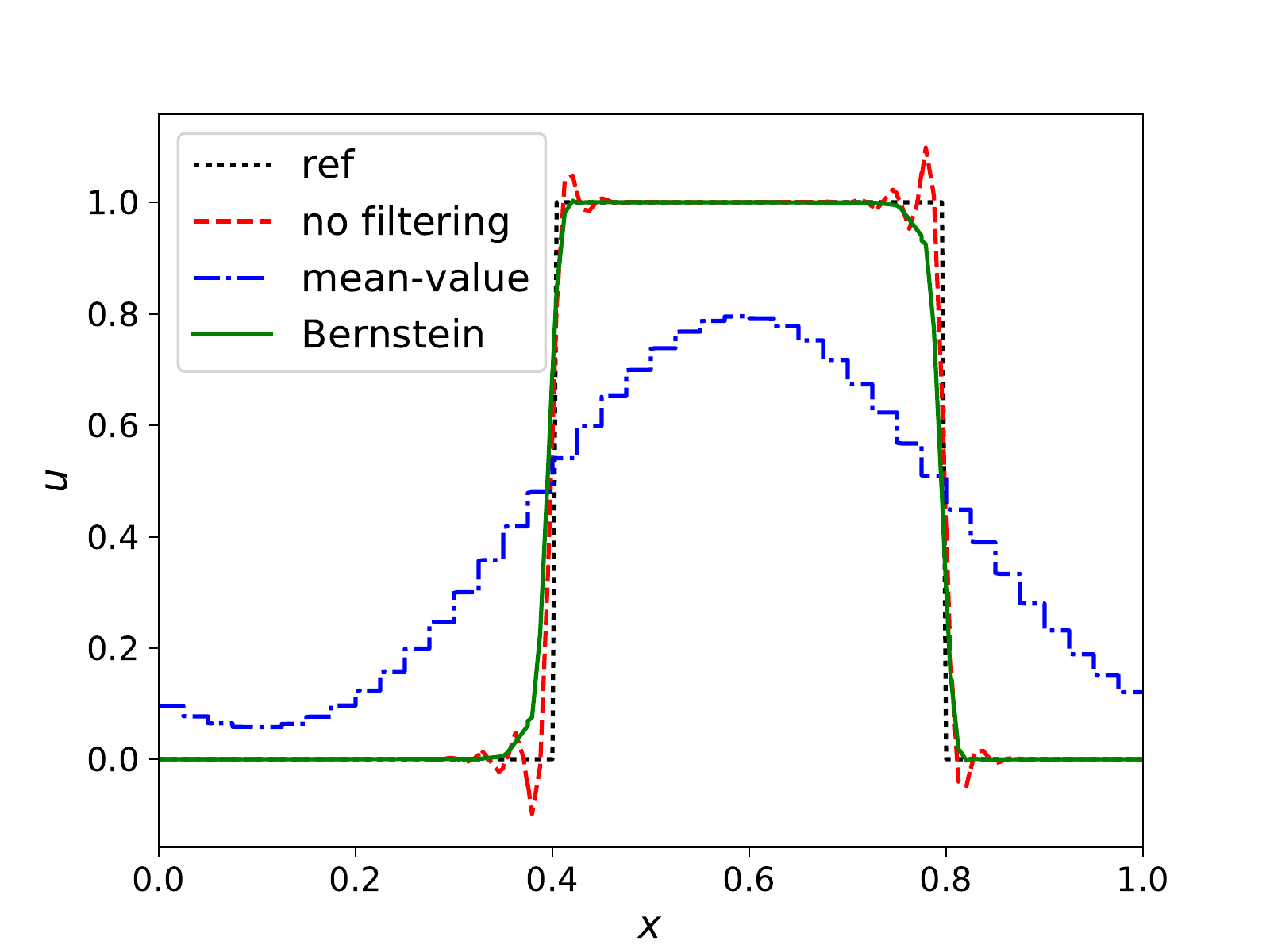}
    \caption{$N=4$ \& $I=40$.}
    \label{fig:linear_T=1_comp_N4_I40} 
  \end{subfigure}%
  ~ 
  \begin{subfigure}[b]{0.33\textwidth}
    \includegraphics[width=\textwidth]{%
      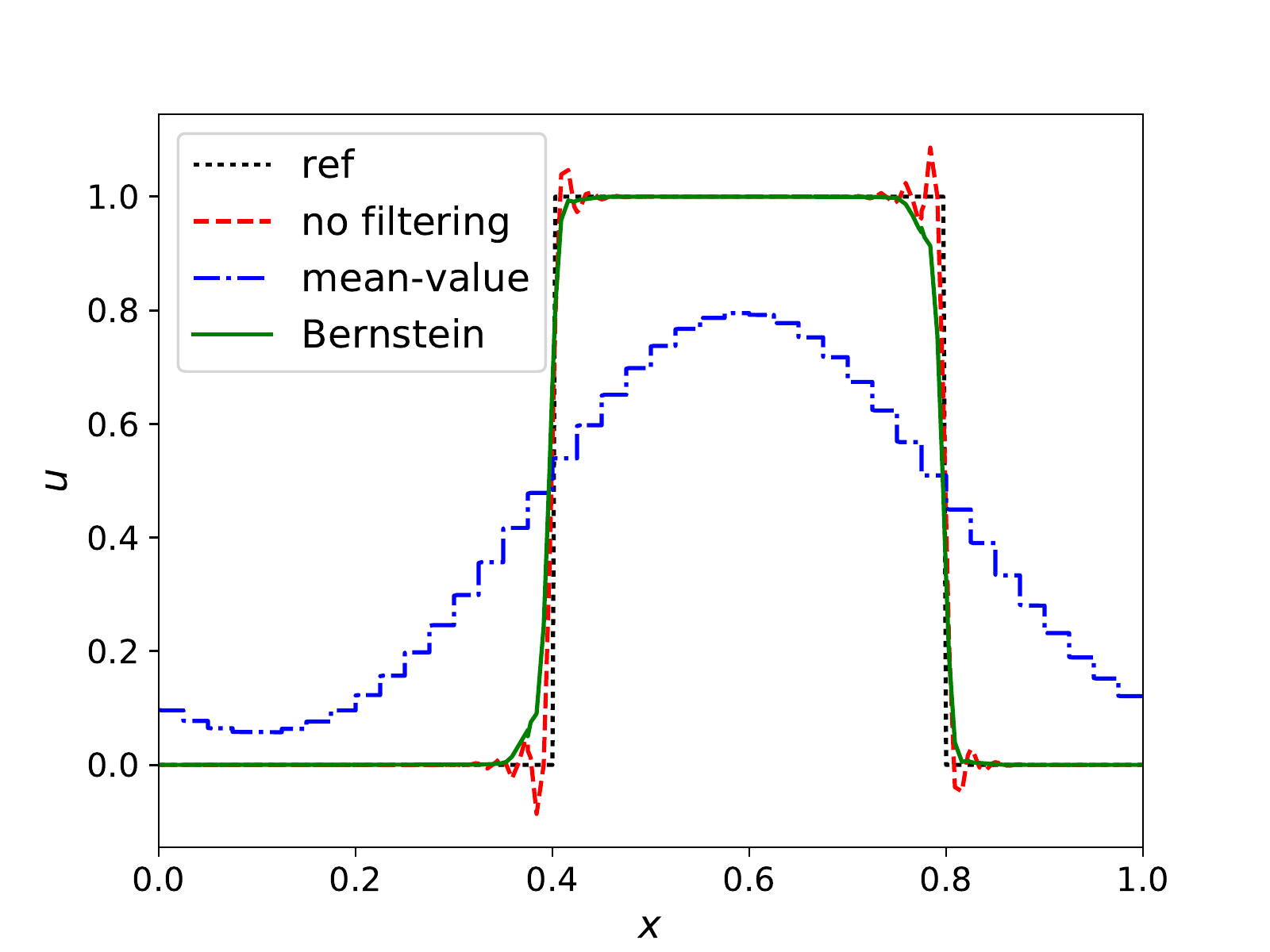}
    \caption{$N=5$ \& $I=40$.}
    \label{fig:linear_T=1_comp_N5_I40} 
  \end{subfigure}%
  \\ 
  \begin{subfigure}[b]{0.33\textwidth}
    \includegraphics[width=\textwidth]{%
      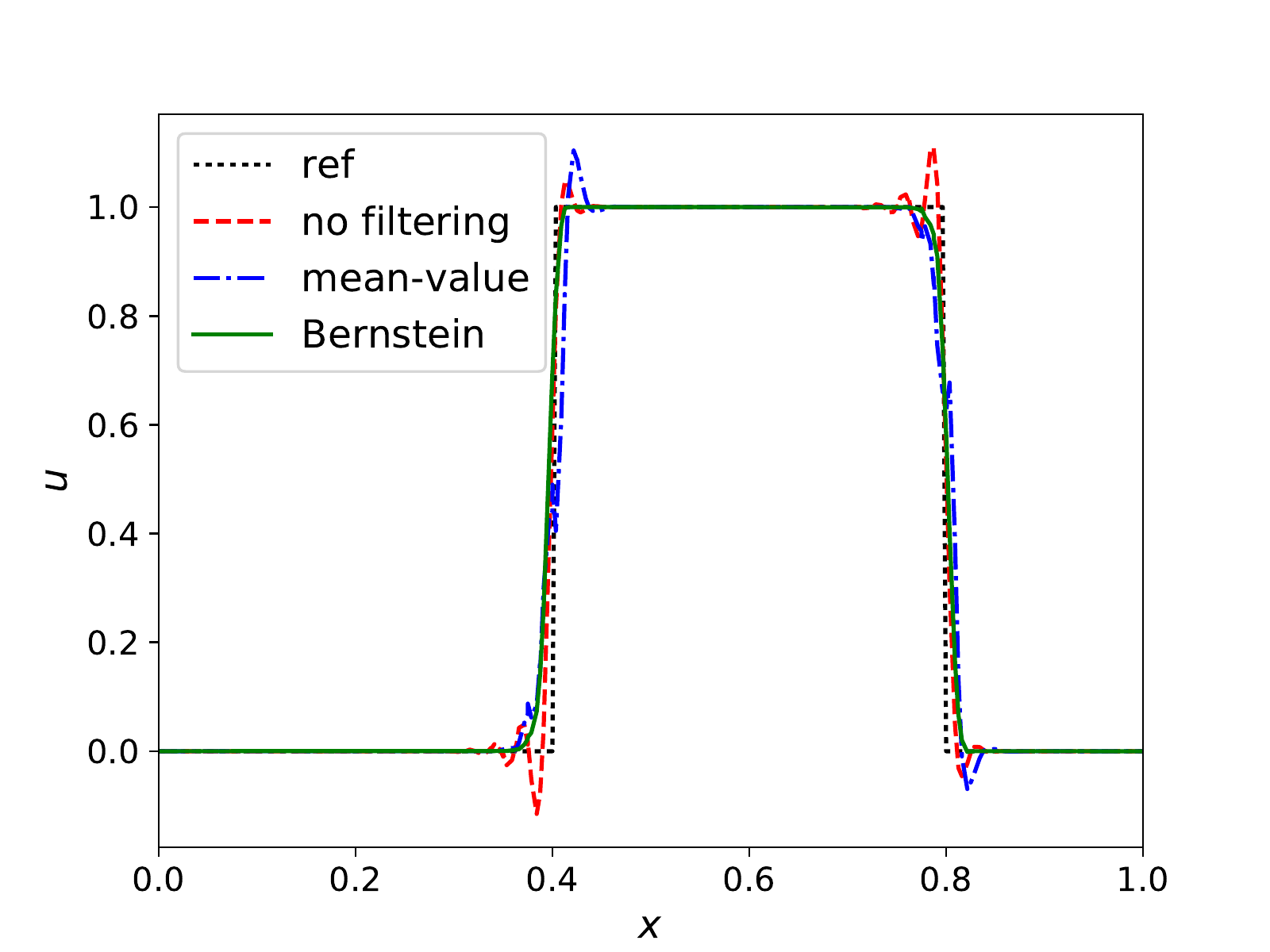}
    \caption{$N=3$ \& $I=80$.}
    \label{fig:linear_T=1_comp_N3_I80} 
  \end{subfigure}%
  ~ 
  \begin{subfigure}[b]{0.33\textwidth}
    \includegraphics[width=\textwidth]{%
      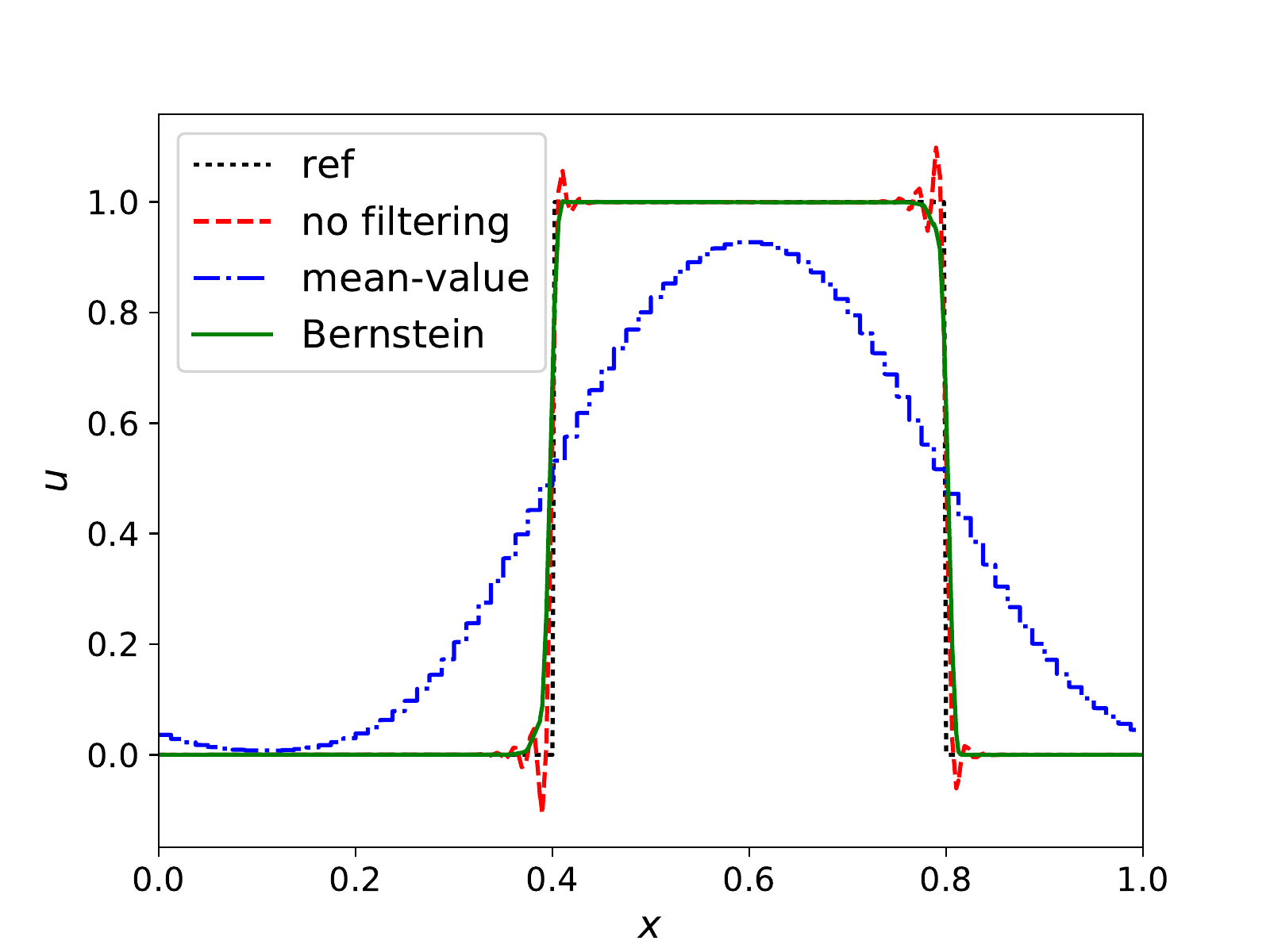}
    \caption{$N=4$ \& $I=80$.}
    \label{fig:linear_T=1_comp_N4_I80} 
  \end{subfigure}%
  ~ 
  \begin{subfigure}[b]{0.33\textwidth}
    \includegraphics[width=\textwidth]{%
      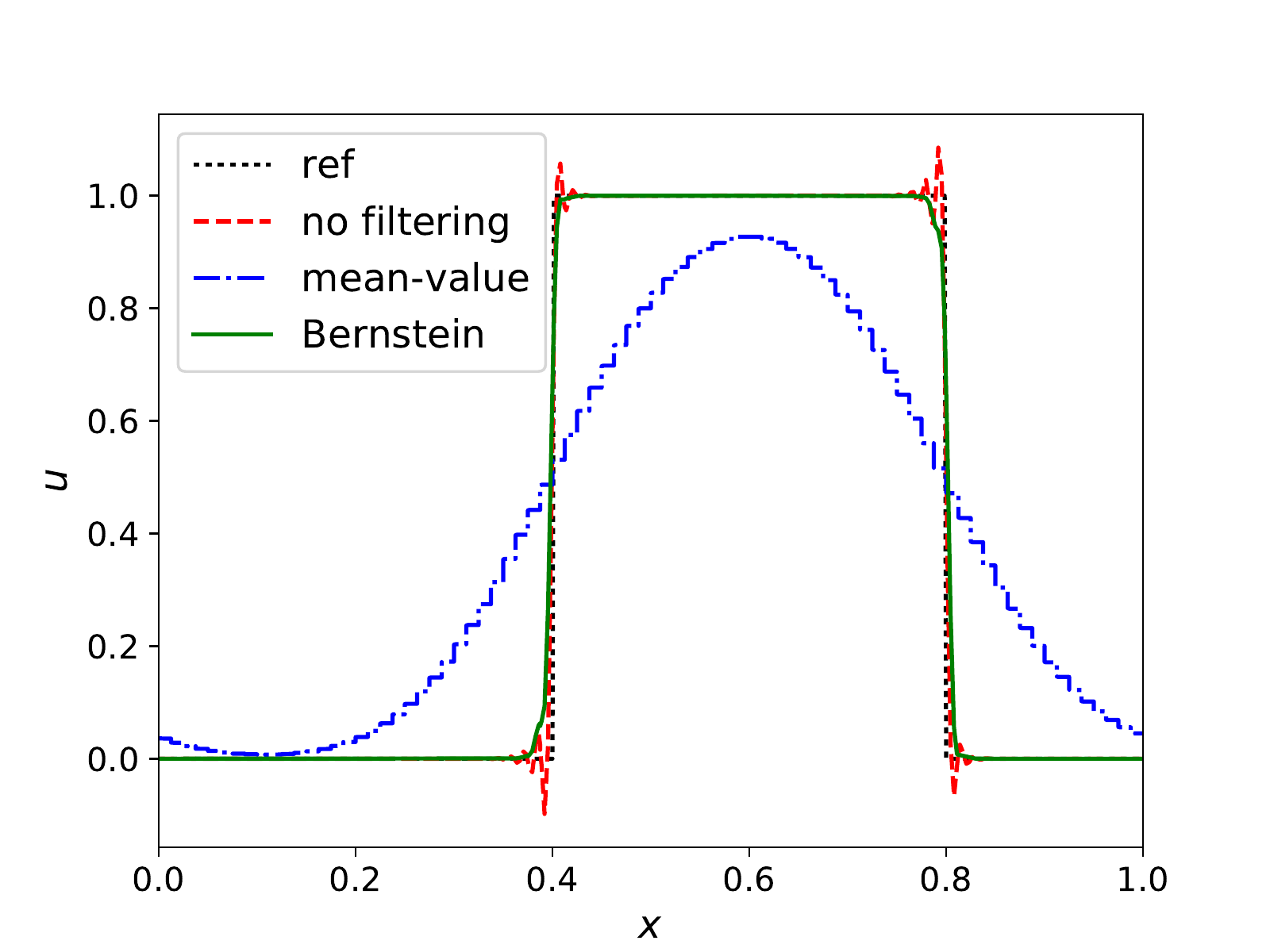}
    \caption{$N=5$ \& $I=80$.}
    \label{fig:linear_T=1_comp_N5_I80} 
  \end{subfigure}%
  \caption{Comparison of the Bernstein procedure 
    in a DG method with a usual filtering technique and no filtering for the linear advection equation 
\eqref{eq:problem1} at time $t=1$.}
  \label{fig:linear_comp_T=1}
\end{figure}


\begin{figure}[!htb]
  \centering
  \begin{subfigure}[b]{0.33\textwidth}
    \includegraphics[width=\textwidth]{%
      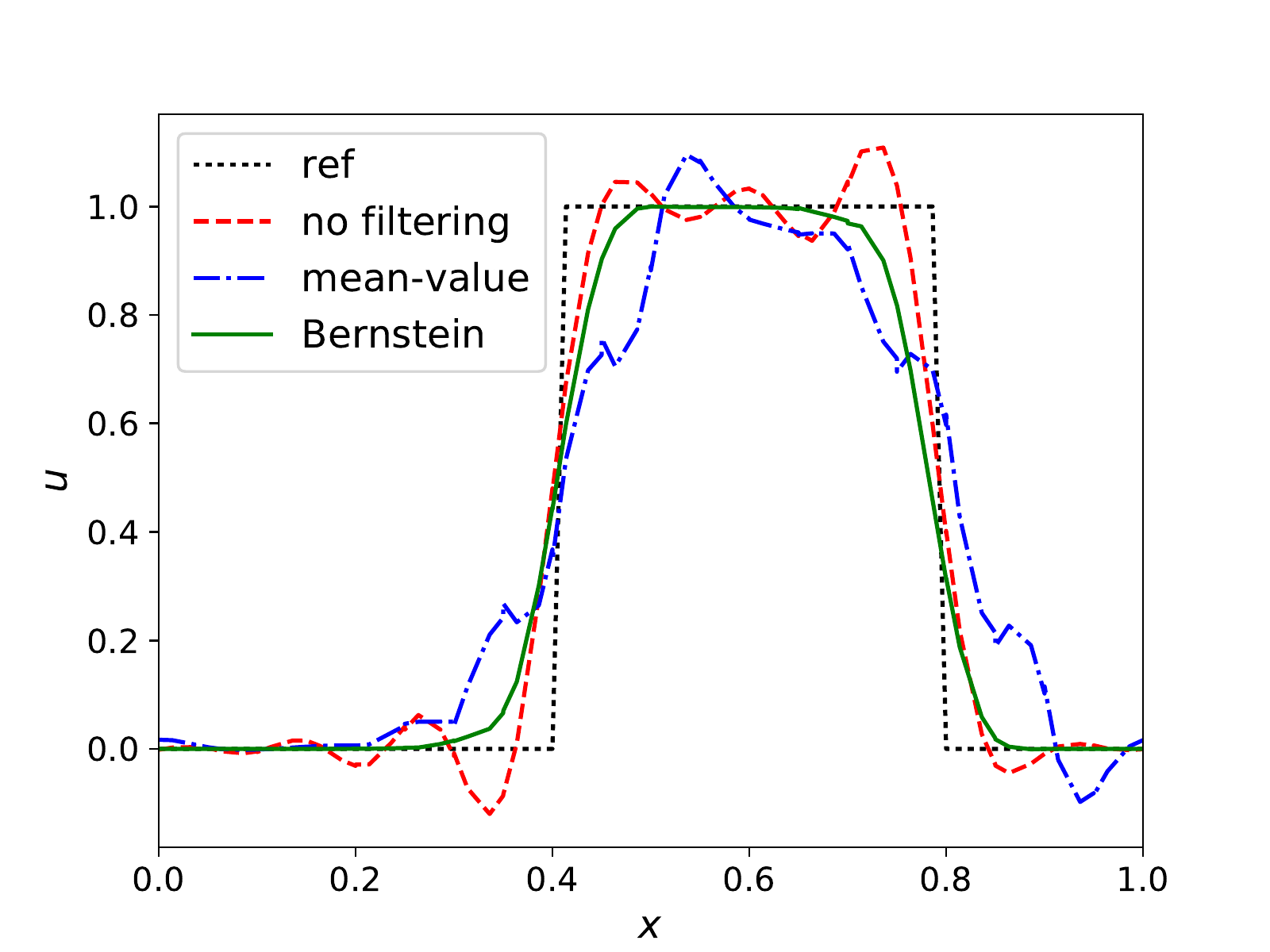}
    \caption{$N=3$ \& $I=20$.}
    \label{fig:linear_T=10_comp_N3_I20} 
  \end{subfigure}%
  ~ 
  \begin{subfigure}[b]{0.33\textwidth}
    \includegraphics[width=\textwidth]{%
      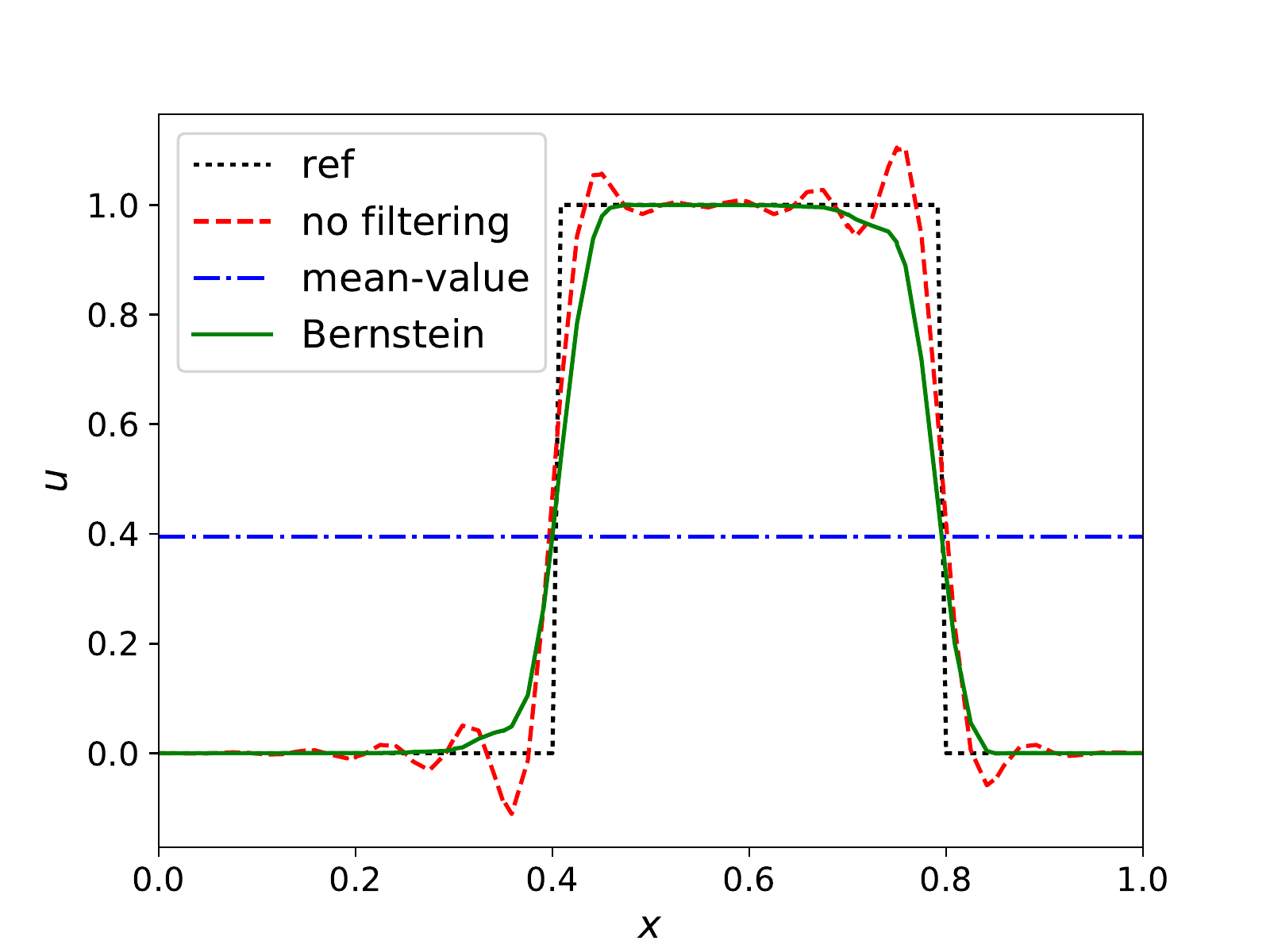}
    \caption{$N=4$ \& $I=20$.}
    \label{fig:linear_T=10_comp_N4_I20} 
  \end{subfigure}%
  ~ 
  \begin{subfigure}[b]{0.33\textwidth}
    \includegraphics[width=\textwidth]{%
      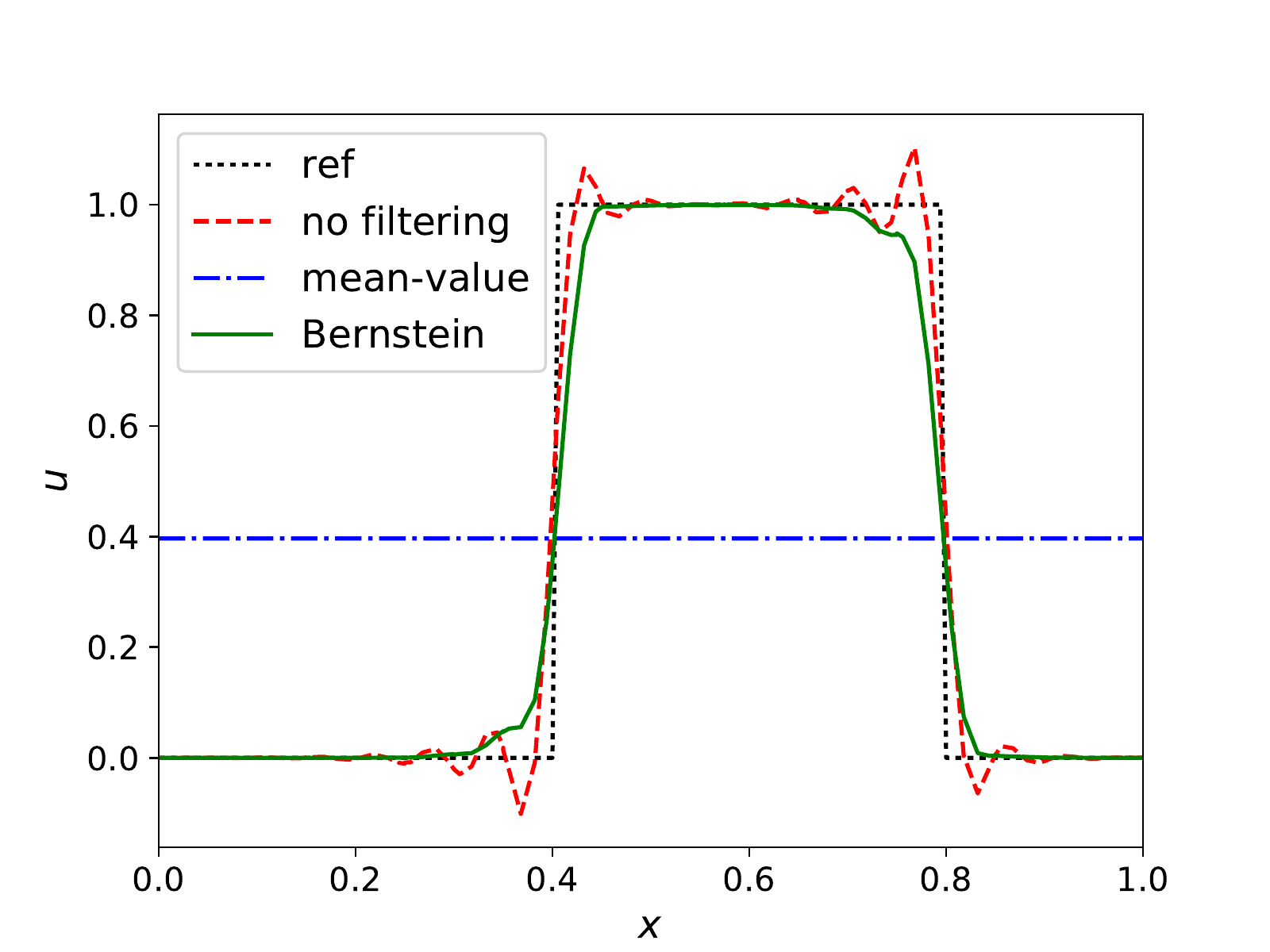}
    \caption{$N=5$ \& $I=20$.}
    \label{fig:linear_T=10_comp_N5_I20} 
  \end{subfigure}%
  \\
  \begin{subfigure}[b]{0.33\textwidth}
    \includegraphics[width=\textwidth]{%
      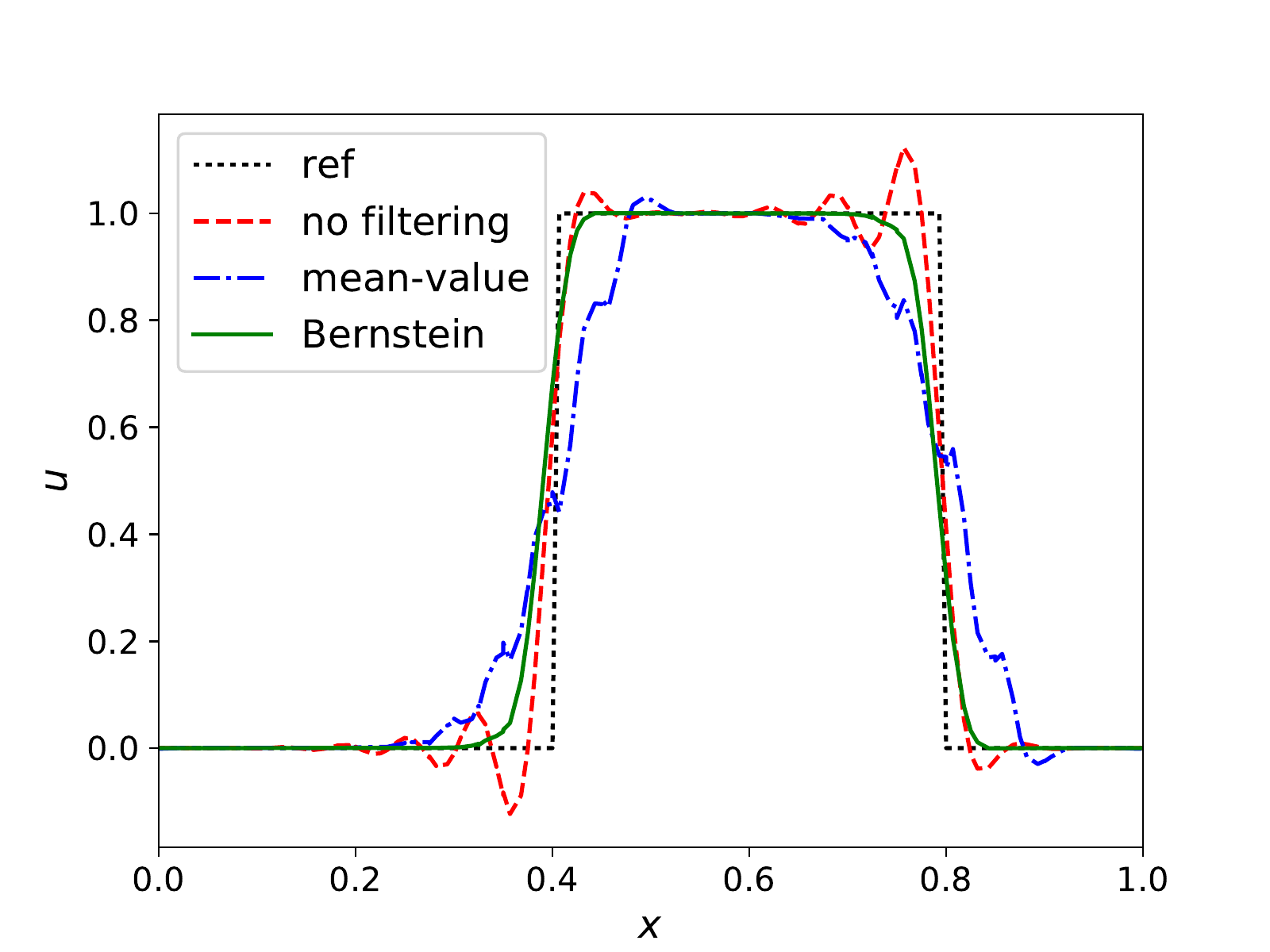}
    \caption{$N=3$ \& $I=40$.}
    \label{fig:linear_T=10_comp_N3_I40} 
  \end{subfigure}%
  ~ 
  \begin{subfigure}[b]{0.33\textwidth}
    \includegraphics[width=\textwidth]{%
      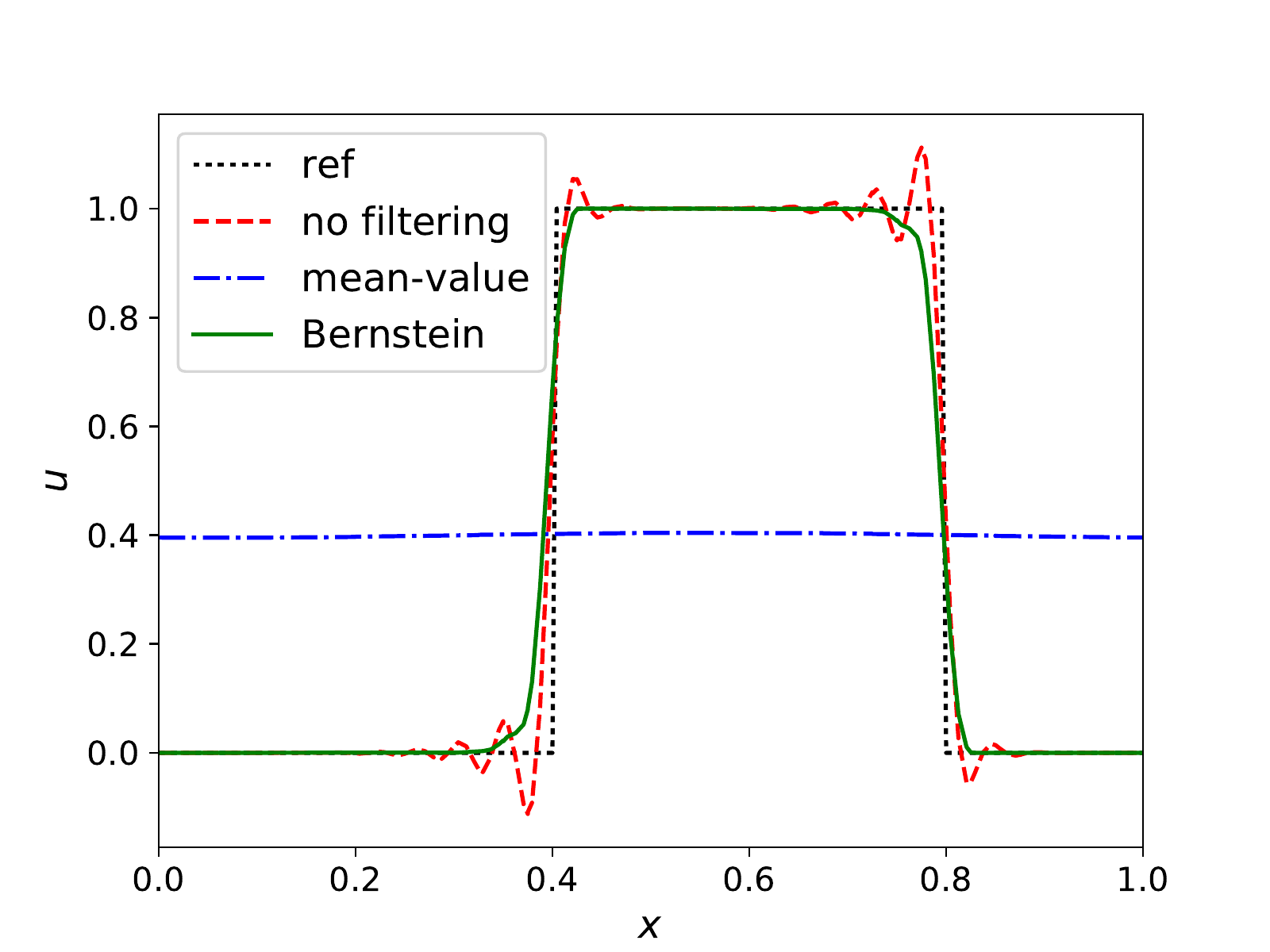}
    \caption{$N=4$ \& $I=40$.}
    \label{fig:linear_T=10_comp_N4_I40} 
  \end{subfigure}%
  ~ 
  \begin{subfigure}[b]{0.33\textwidth}
    \includegraphics[width=\textwidth]{%
      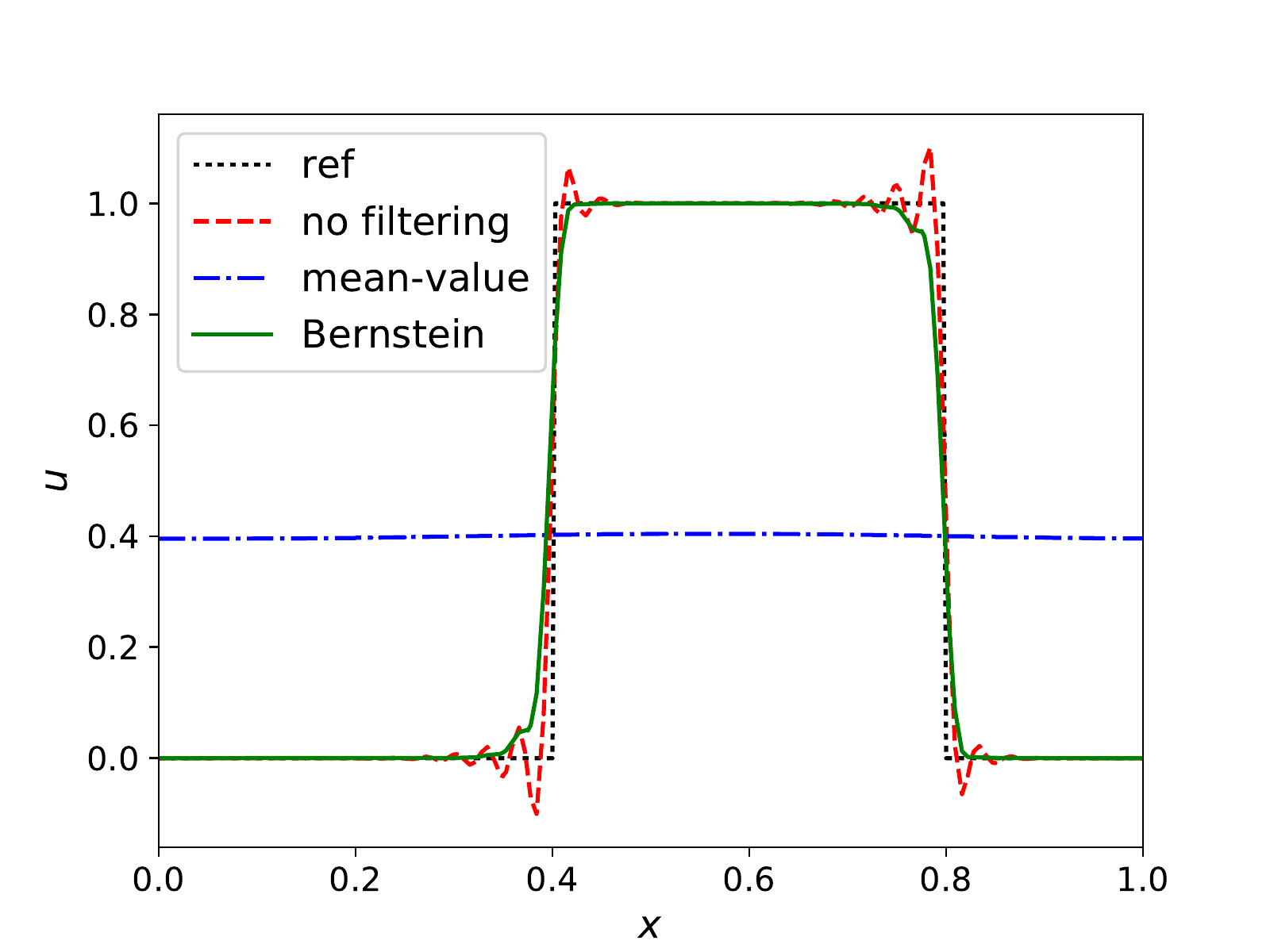}
    \caption{$N=5$ \& $I=40$.}
    \label{fig:linear_T=10_comp_N5_I40} 
  \end{subfigure}%
  \\ 
  \begin{subfigure}[b]{0.33\textwidth}
    \includegraphics[width=\textwidth]{%
      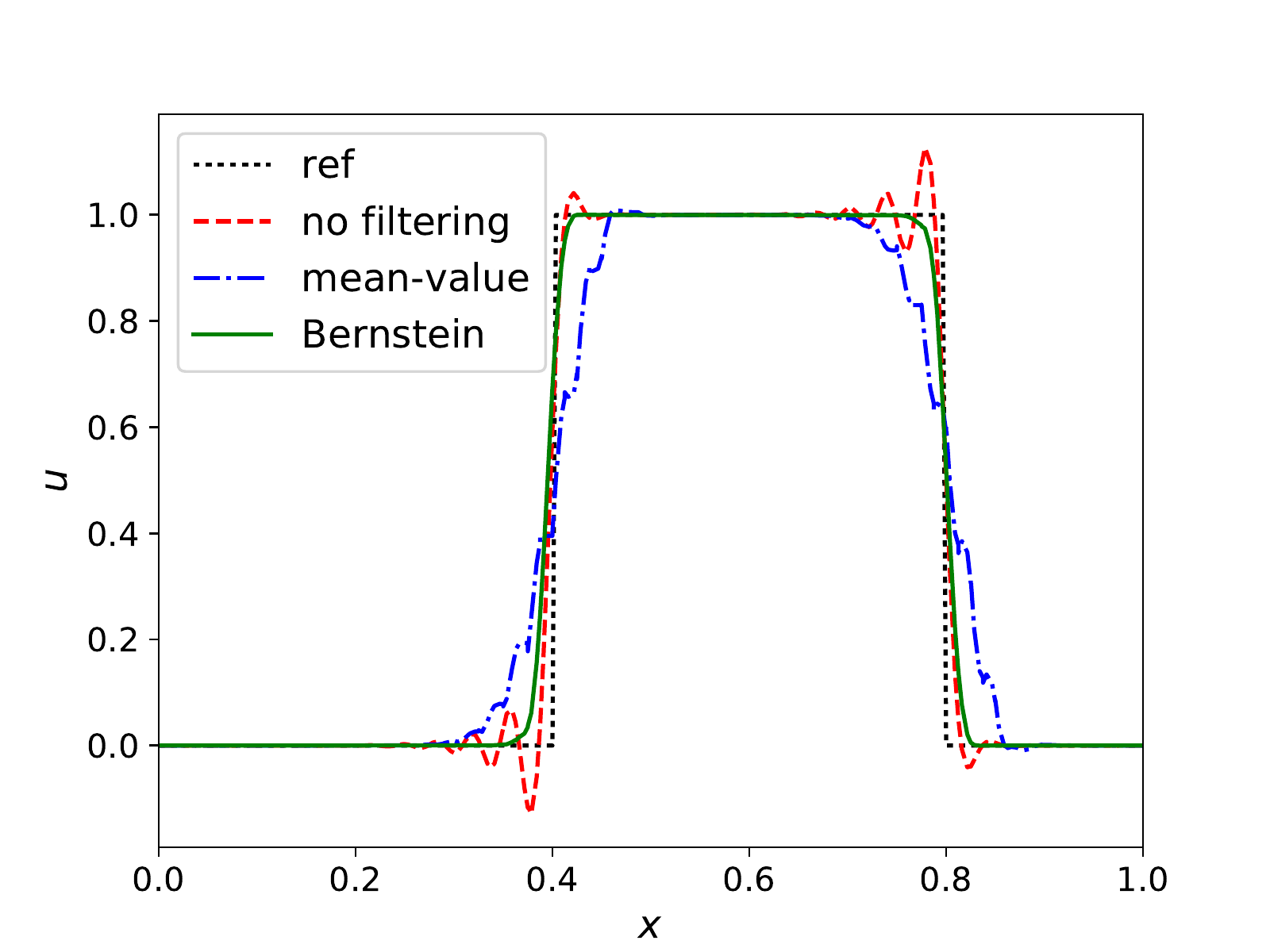}
    \caption{$N=3$ \& $I=80$.}
    \label{fig:linear_T=10_comp_N3_I80} 
  \end{subfigure}%
  ~ 
  \begin{subfigure}[b]{0.33\textwidth}
    \includegraphics[width=\textwidth]{%
      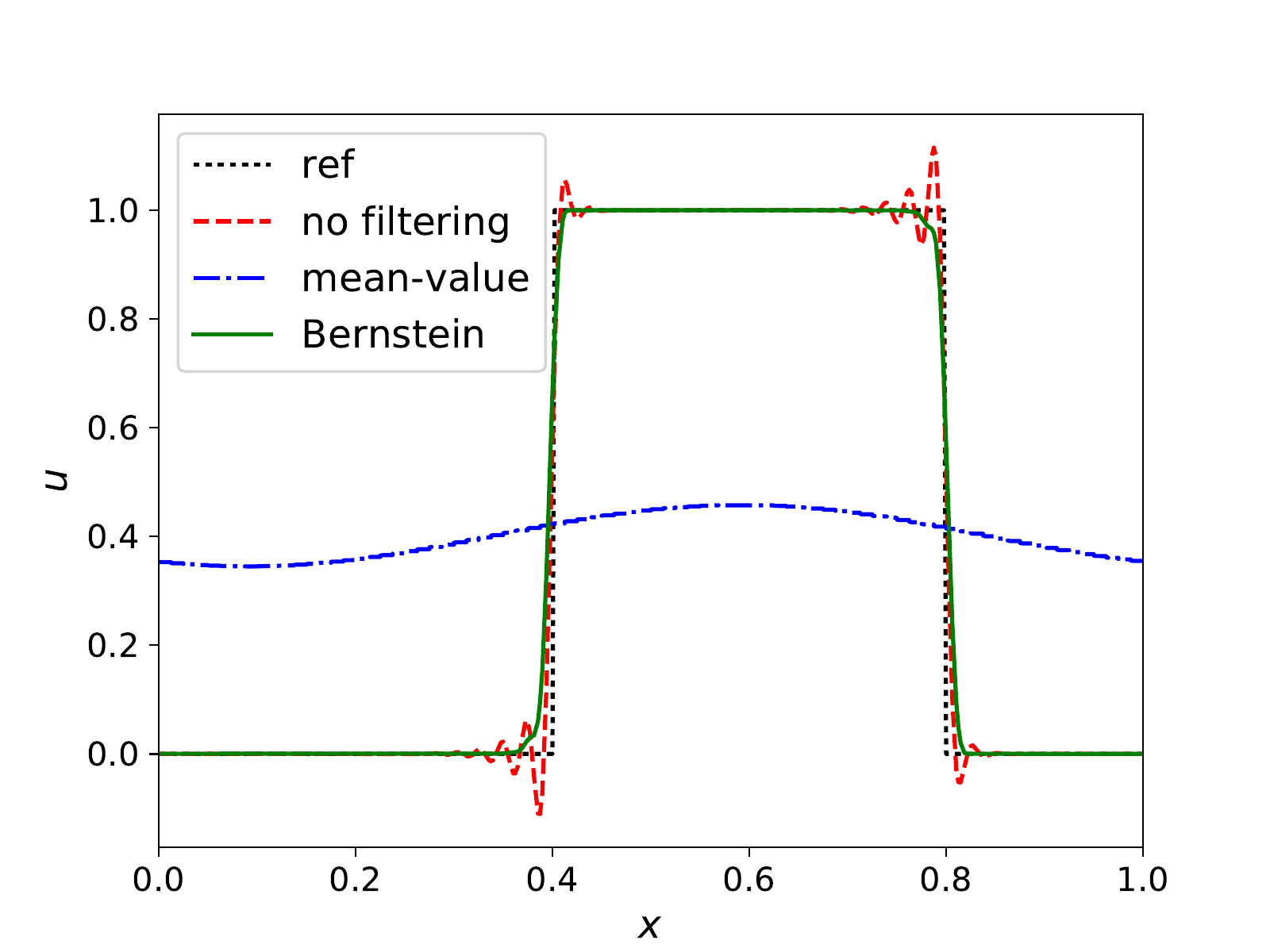}
    \caption{$N=4$ \& $I=80$.}
    \label{fig:linear_T=10_comp_N4_I80} 
  \end{subfigure}%
  ~ 
  \begin{subfigure}[b]{0.33\textwidth}
    \includegraphics[width=\textwidth]{%
      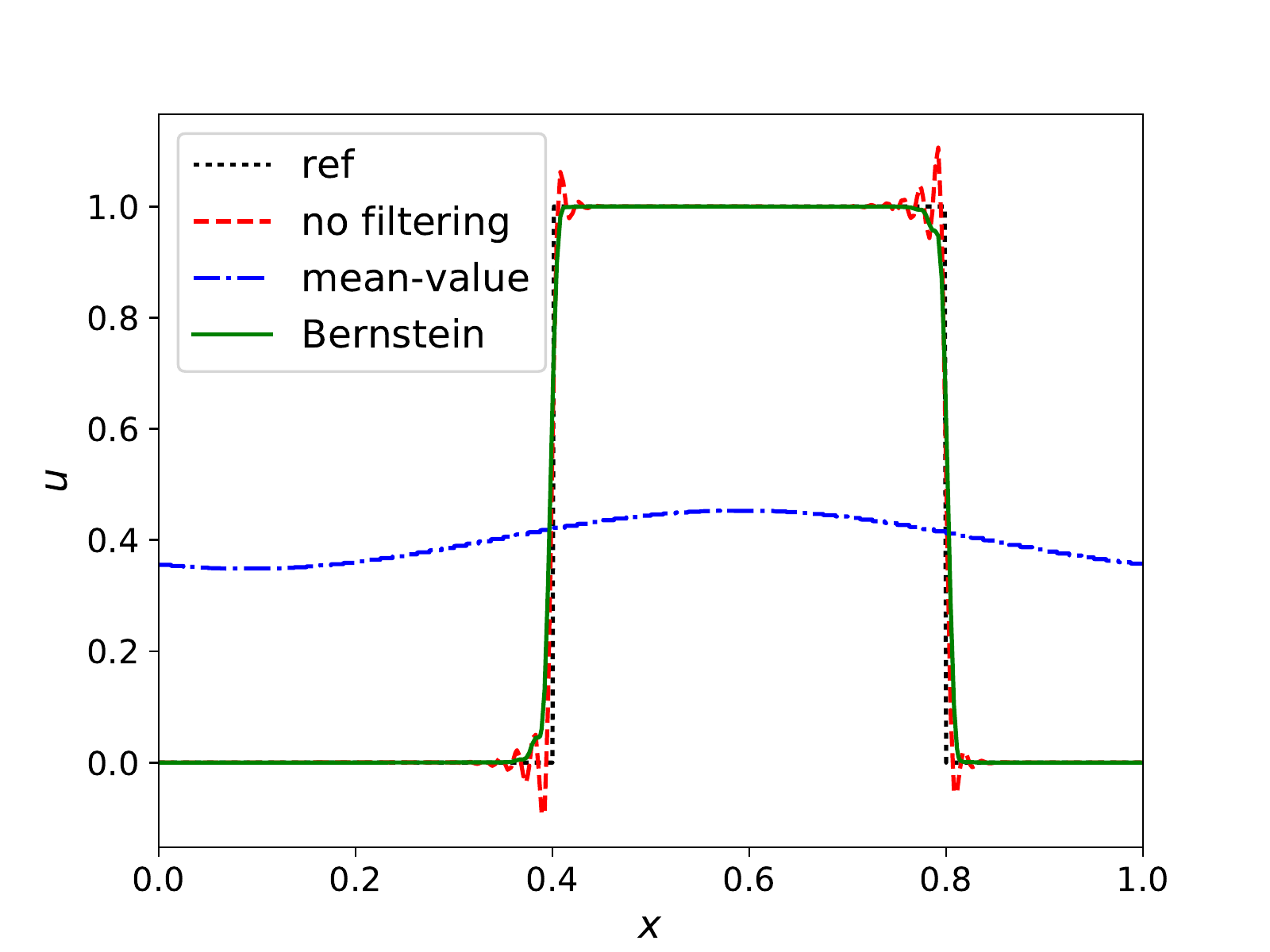}
    \caption{$N=5$ \& $I=80$.}
    \label{fig:linear_T=10_comp_N5_I80} 
  \end{subfigure}%
  \caption{Comparison of the Bernstein procedure 
    in a DG method with a usual filtering technique and no filtering for the linear advection equation 
\eqref{eq:problem1} at time $t=10$.}
  \label{fig:linear_comp_T=10}
\end{figure}


\begin{figure}[!htb]
  \centering
  \begin{subfigure}[b]{0.33\textwidth}
    \includegraphics[width=\textwidth]{%
      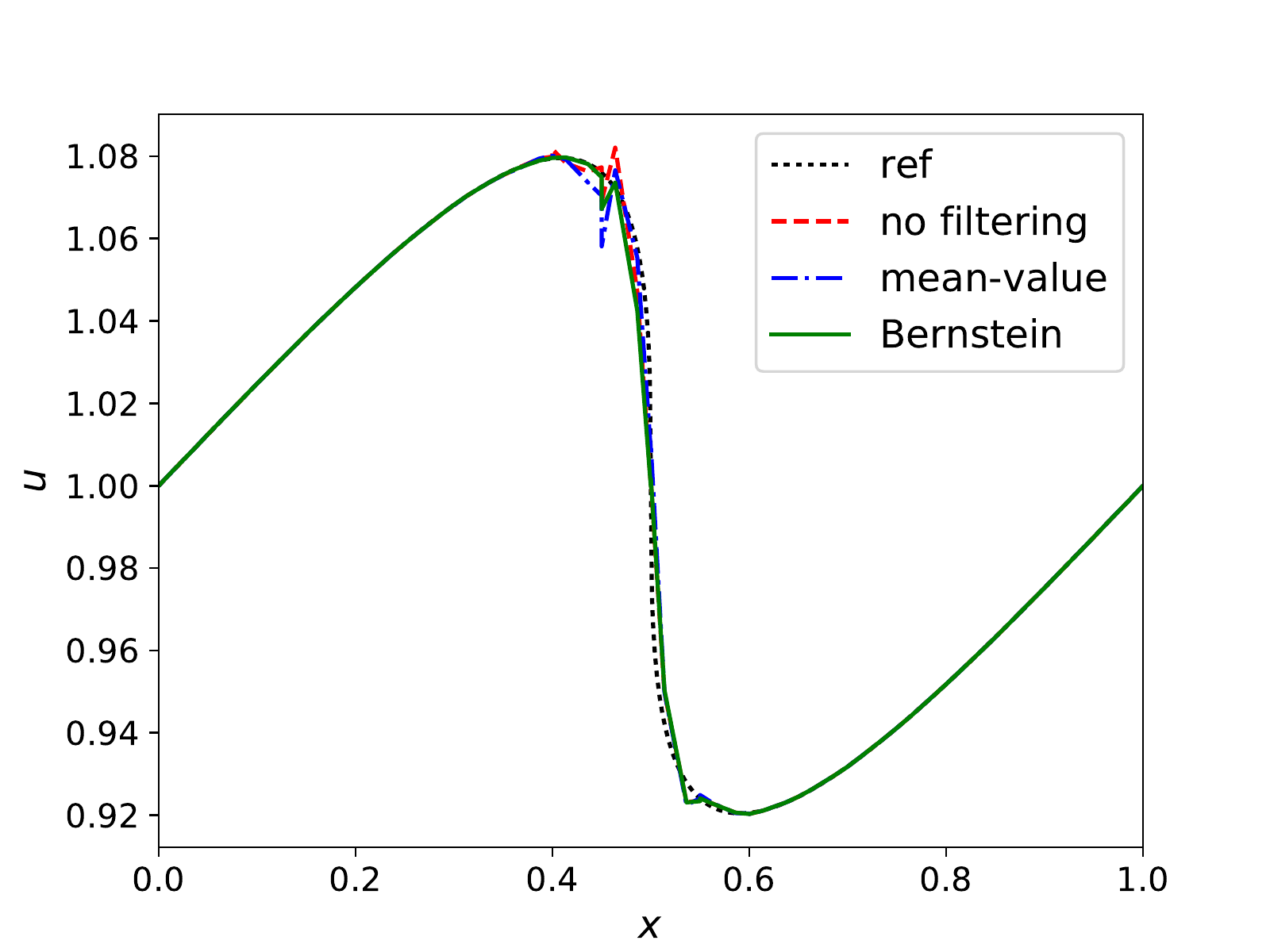}
    \caption{$N=3$ \& $I=20$.}
    \label{fig:burgers_T=2_comp_N3_I20} 
  \end{subfigure}%
  ~ 
  \begin{subfigure}[b]{0.33\textwidth}
    \includegraphics[width=\textwidth]{%
      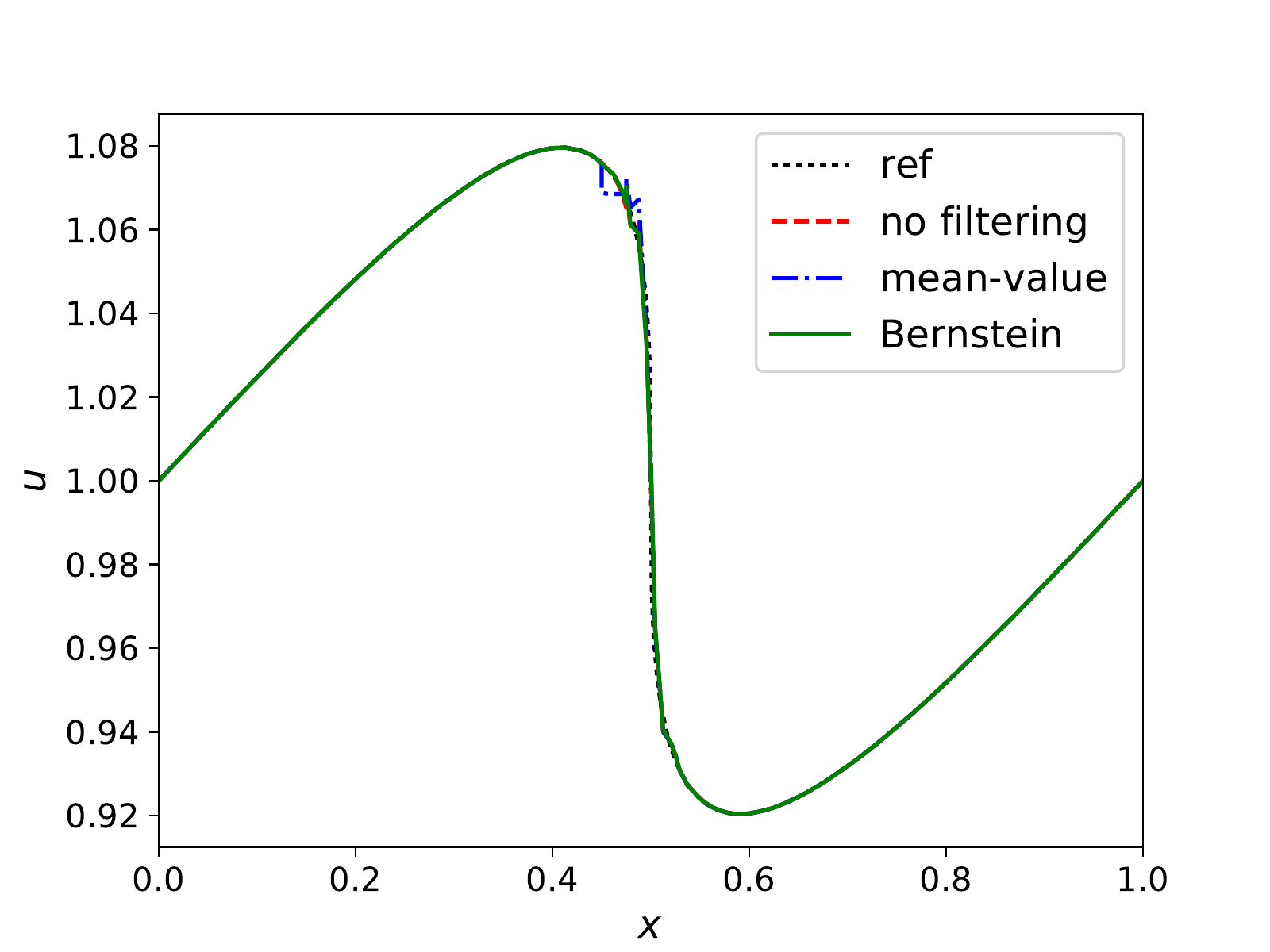}
    \caption{$N=4$ \& $I=20$.}
    \label{fig:burgers_T=2_comp_N4_I20} 
  \end{subfigure}%
  ~ 
  \begin{subfigure}[b]{0.33\textwidth}
    \includegraphics[width=\textwidth]{%
      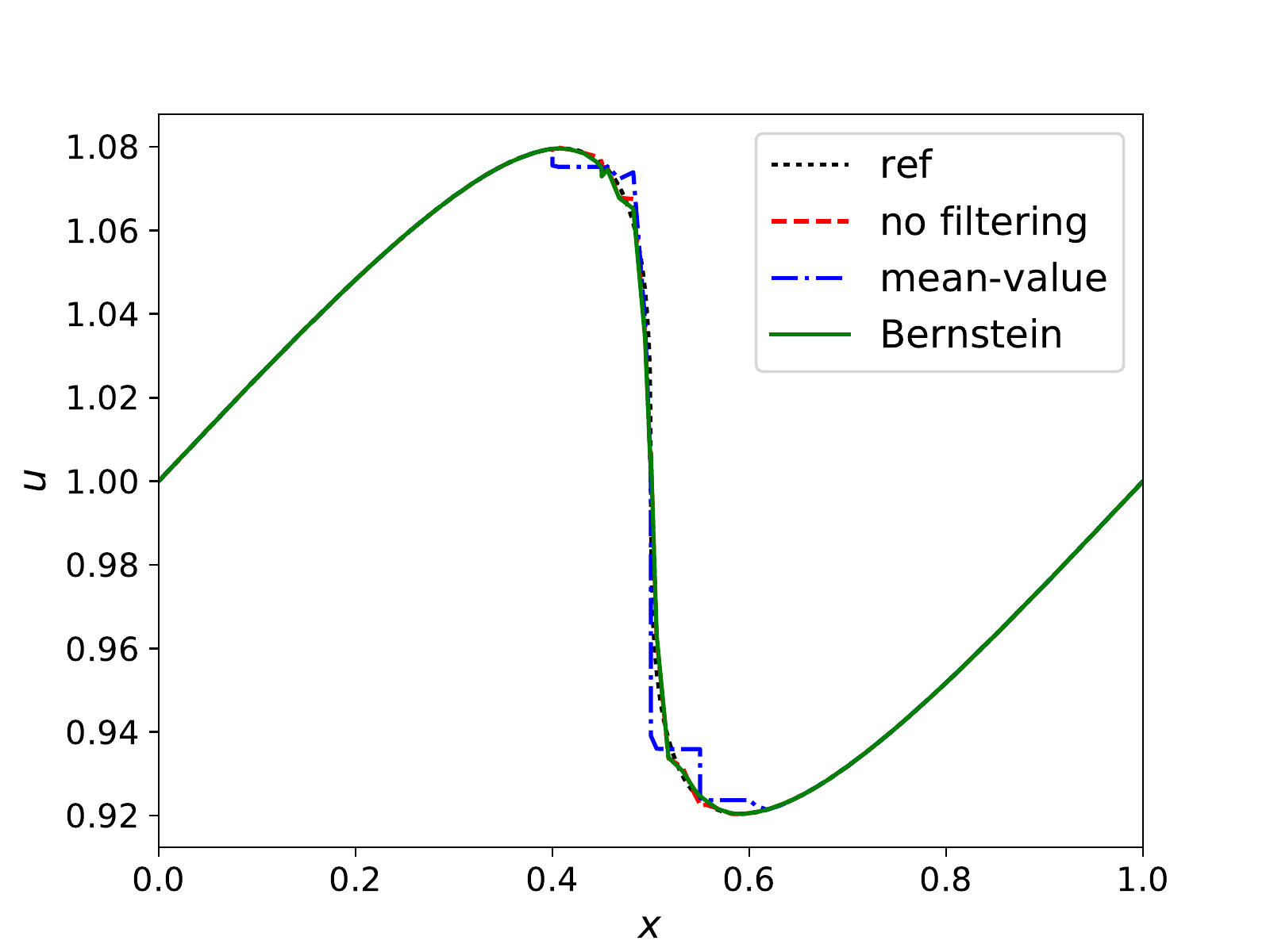}
    \caption{$N=5$ \& $I=20$.}
    \label{fig:burgers_T=2_comp_N5_I20} 
  \end{subfigure}%
  \\ 
  \begin{subfigure}[b]{0.33\textwidth}
    \includegraphics[width=\textwidth]{%
      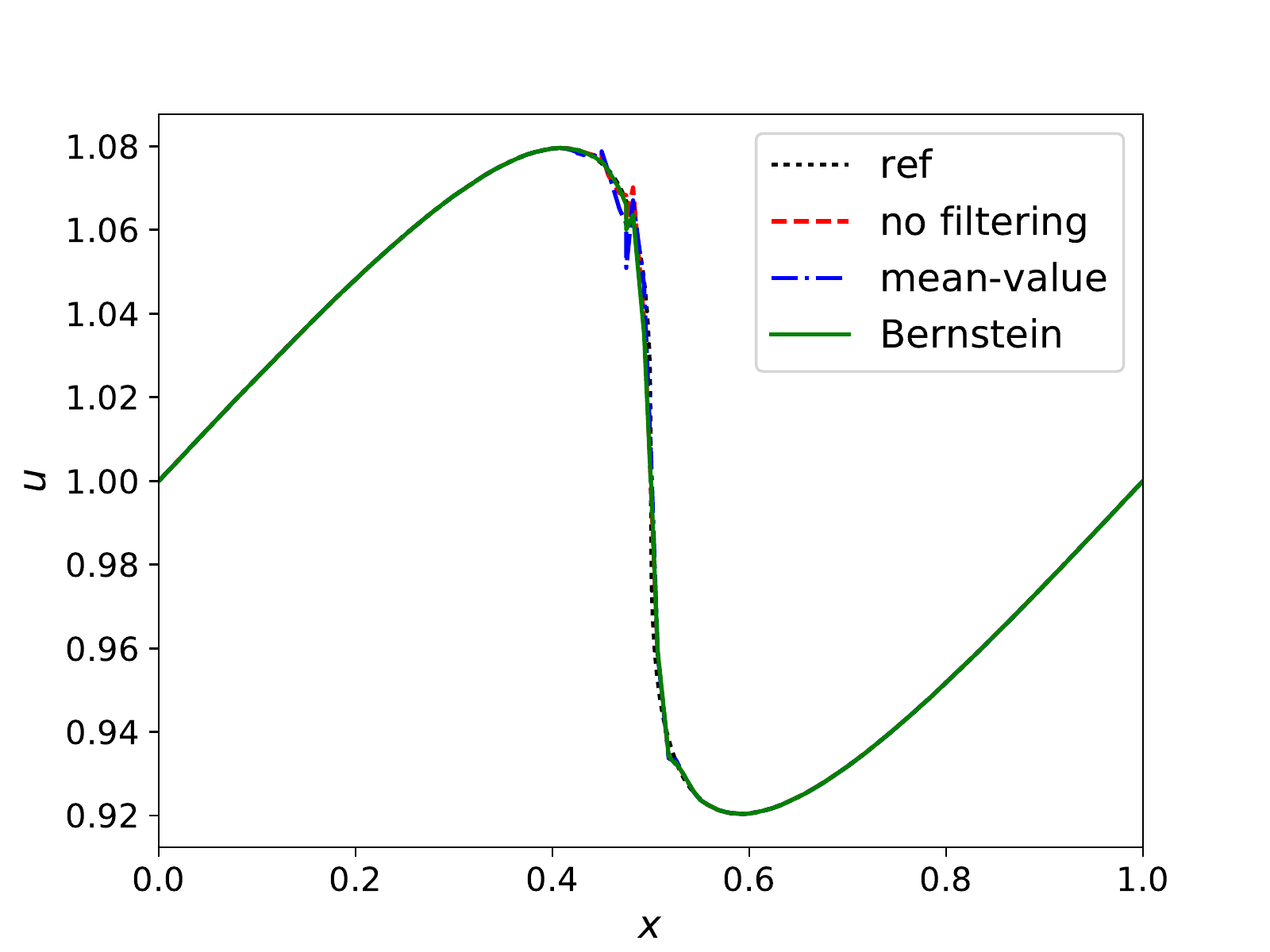}
    \caption{$N=3$ \& $I=40$.}
    \label{fig:burgers_T=2_comp_N3_I40} 
  \end{subfigure}%
  ~ 
  \begin{subfigure}[b]{0.33\textwidth}
    \includegraphics[width=\textwidth]{%
      comp_burgers_T=2_I=40_N=4}
    \caption{$N=4$ \& $I=40$.}
    \label{fig:burgers_T=2_comp_N4_I40} 
  \end{subfigure}%
  ~ 
  \begin{subfigure}[b]{0.33\textwidth}
    \includegraphics[width=\textwidth]{%
      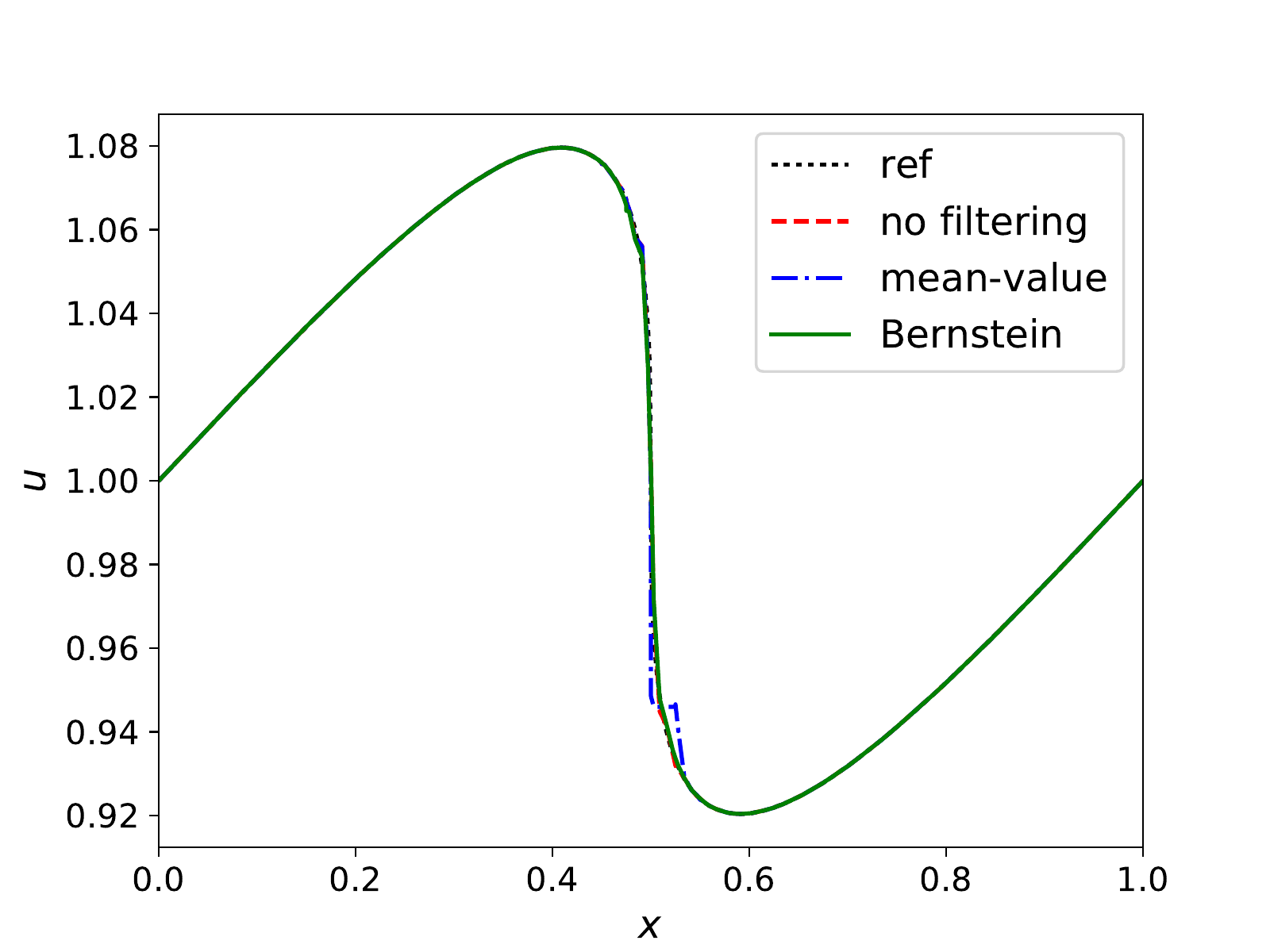}
    \caption{$N=5$ \& $I=40$.}
    \label{fig:burgers_T=2_comp_N5_I40} 
  \end{subfigure}%
  \\ 
  \begin{subfigure}[b]{0.33\textwidth}
    \includegraphics[width=\textwidth]{%
      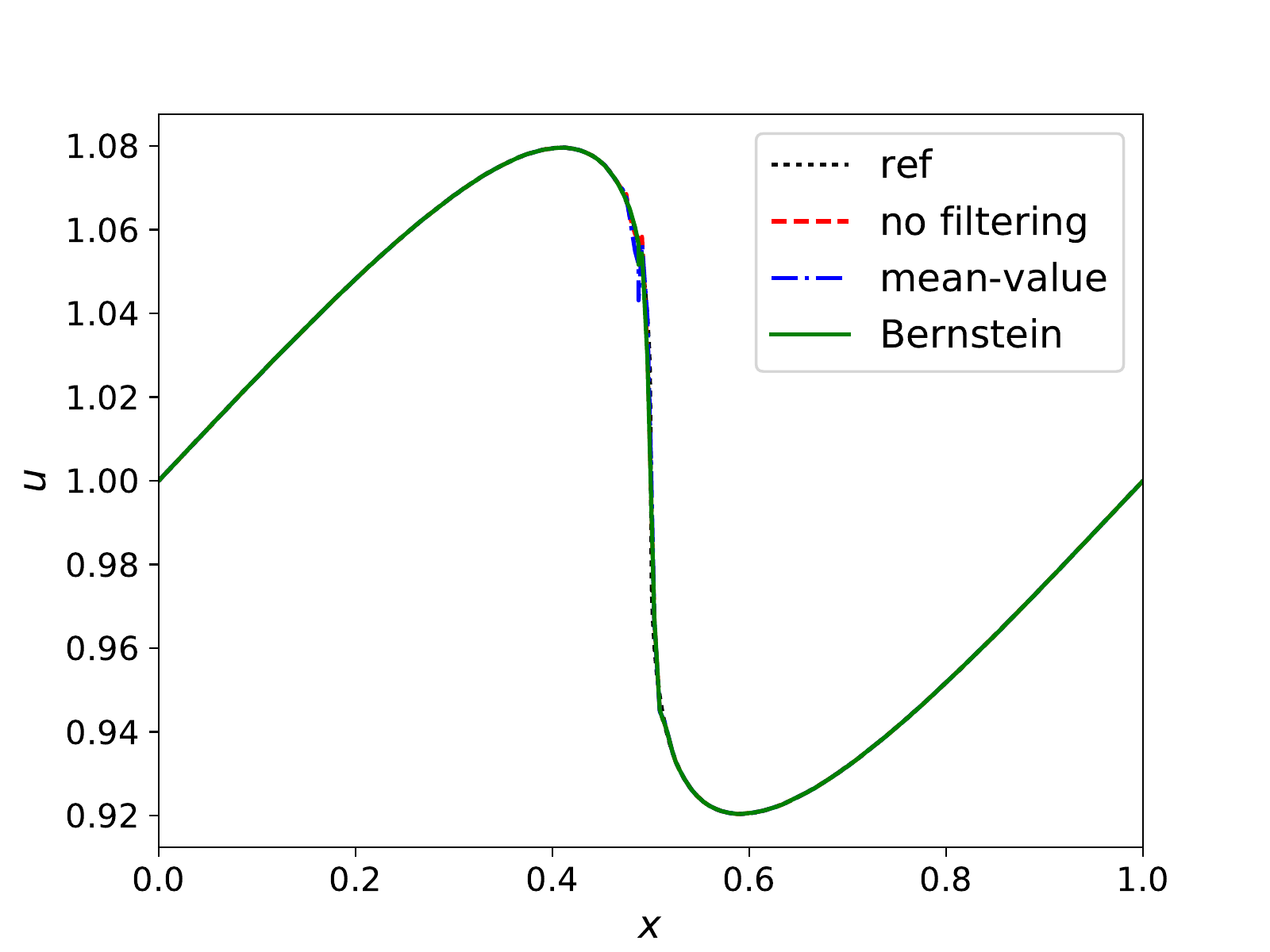}
    \caption{$N=3$ \& $I=80$.}
    \label{fig:burgers_T=2_comp_N3_I80} 
  \end{subfigure}%
  ~ 
  \begin{subfigure}[b]{0.33\textwidth}
    \includegraphics[width=\textwidth]{%
      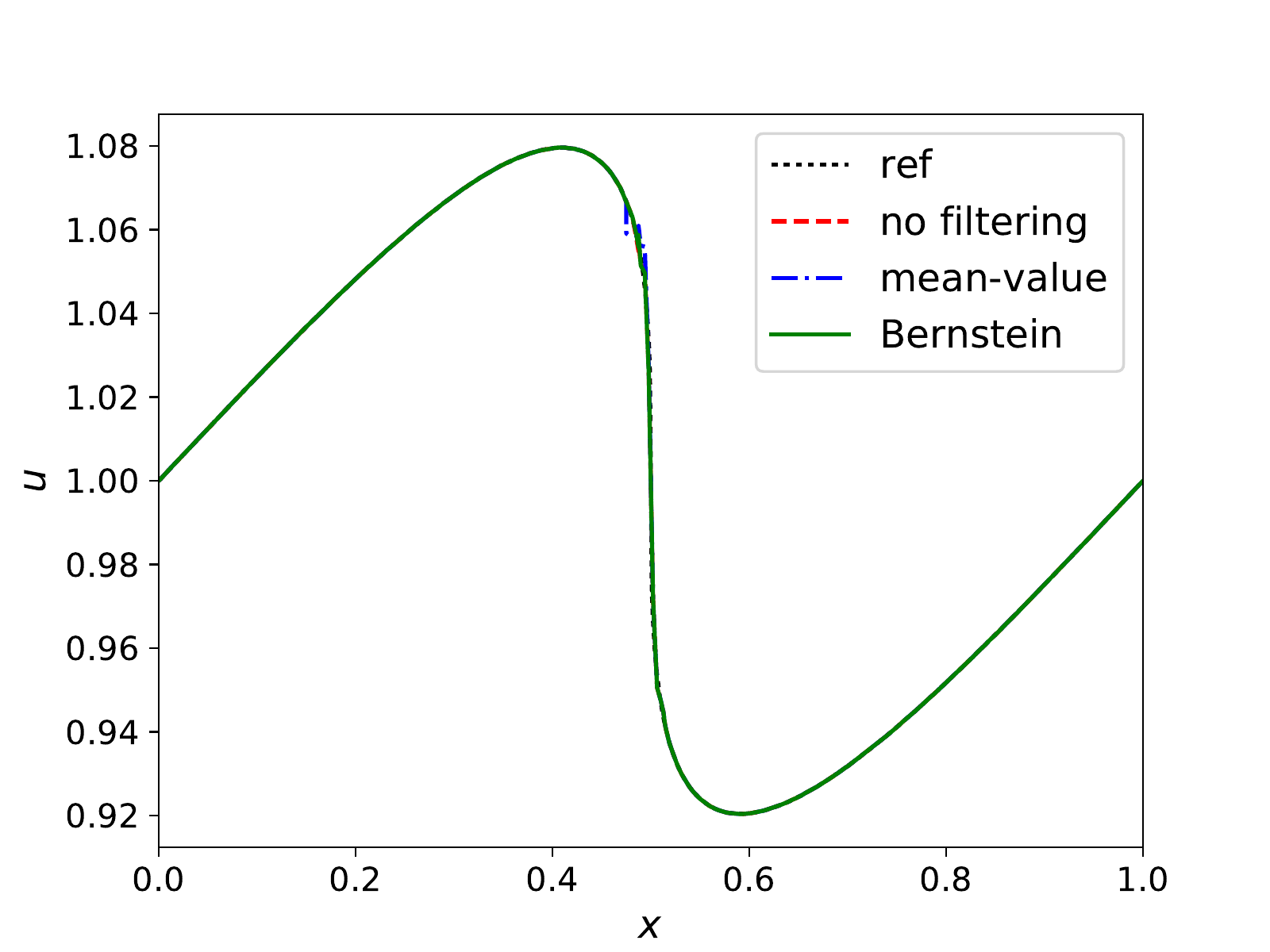}
    \caption{$N=4$ \& $I=80$.}
    \label{fig:burgers_T=2_comp_N4_I80} 
  \end{subfigure}%
  ~ 
  \begin{subfigure}[b]{0.33\textwidth}
    \includegraphics[width=\textwidth]{%
      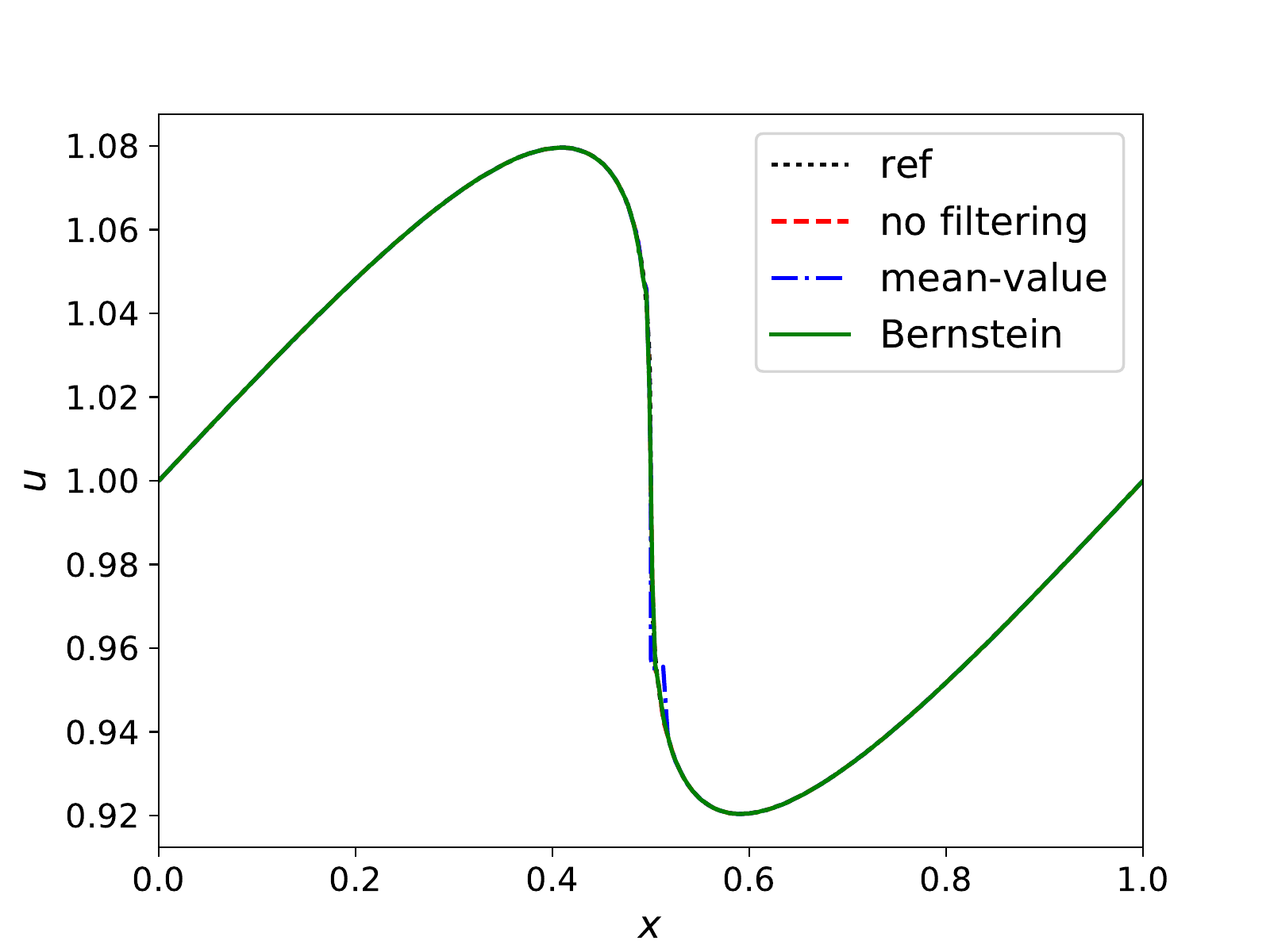}
    \caption{$N=5$ \& $I=80$.}
    \label{fig:burgers_T=2_comp_N5_I80} 
  \end{subfigure}%
  \caption{Comparison of the Bernstein procedure 
    in a DG method with a usual filtering technique and no filtering for the invicid Burgers equation 
    \eqref{eq:problem2} at time $t=2$.}
  \label{fig:burgers_comp_T=2}
\end{figure}


\begin{figure}[!htb]
  \centering
  \begin{subfigure}[b]{0.33\textwidth}
    \includegraphics[width=\textwidth]{%
      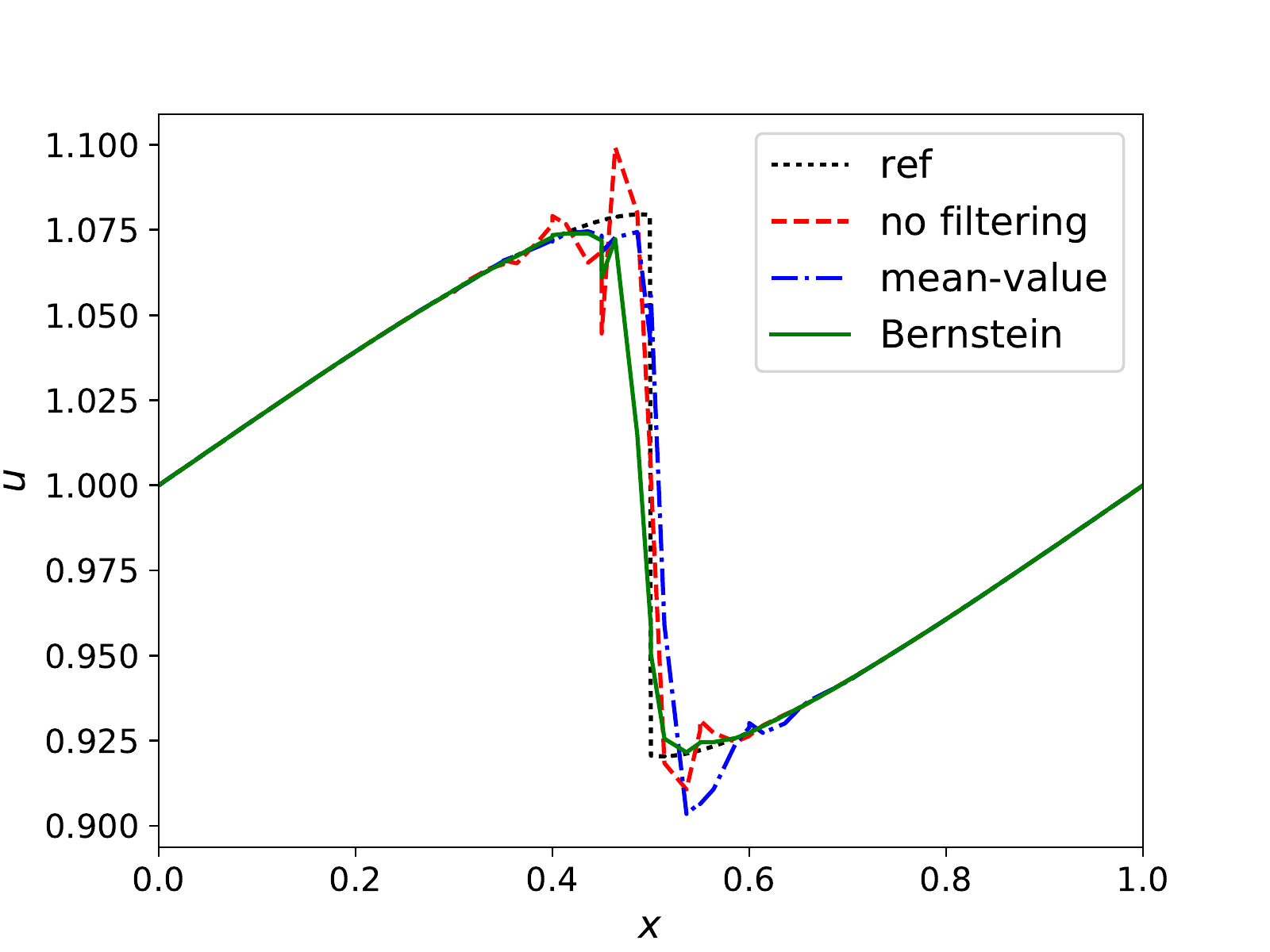}
    \caption{$N=3$ \& $I=20$.}
    \label{fig:burgers_T=3_comp_N3_I20} 
  \end{subfigure}%
  ~ 
  \begin{subfigure}[b]{0.33\textwidth}
    \includegraphics[width=\textwidth]{%
      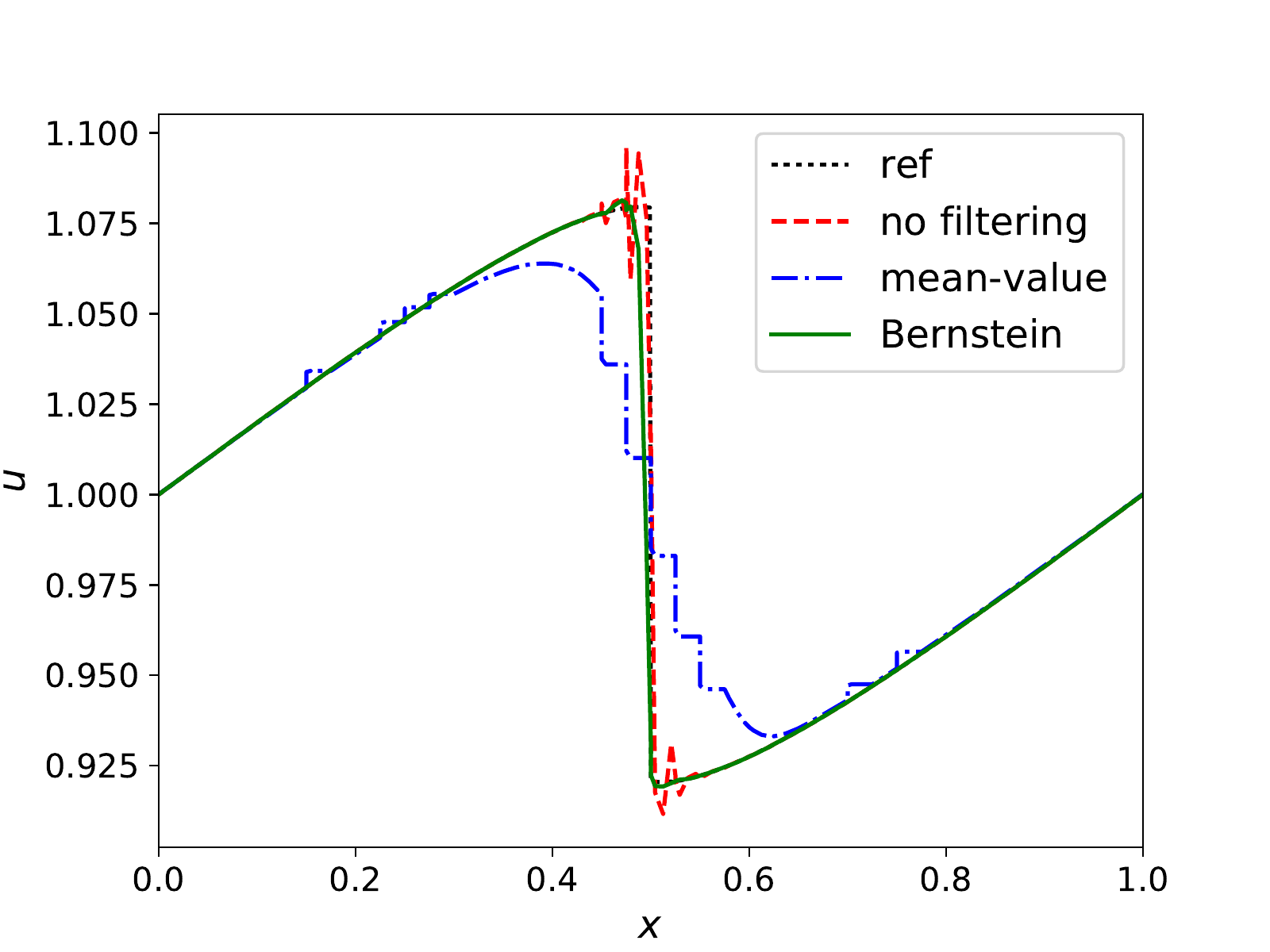}
    \caption{$N=4$ \& $I=20$.}
    \label{fig:burgers_T=3_comp_N4_I20} 
  \end{subfigure}%
  ~ 
  \begin{subfigure}[b]{0.33\textwidth}
    \includegraphics[width=\textwidth]{%
      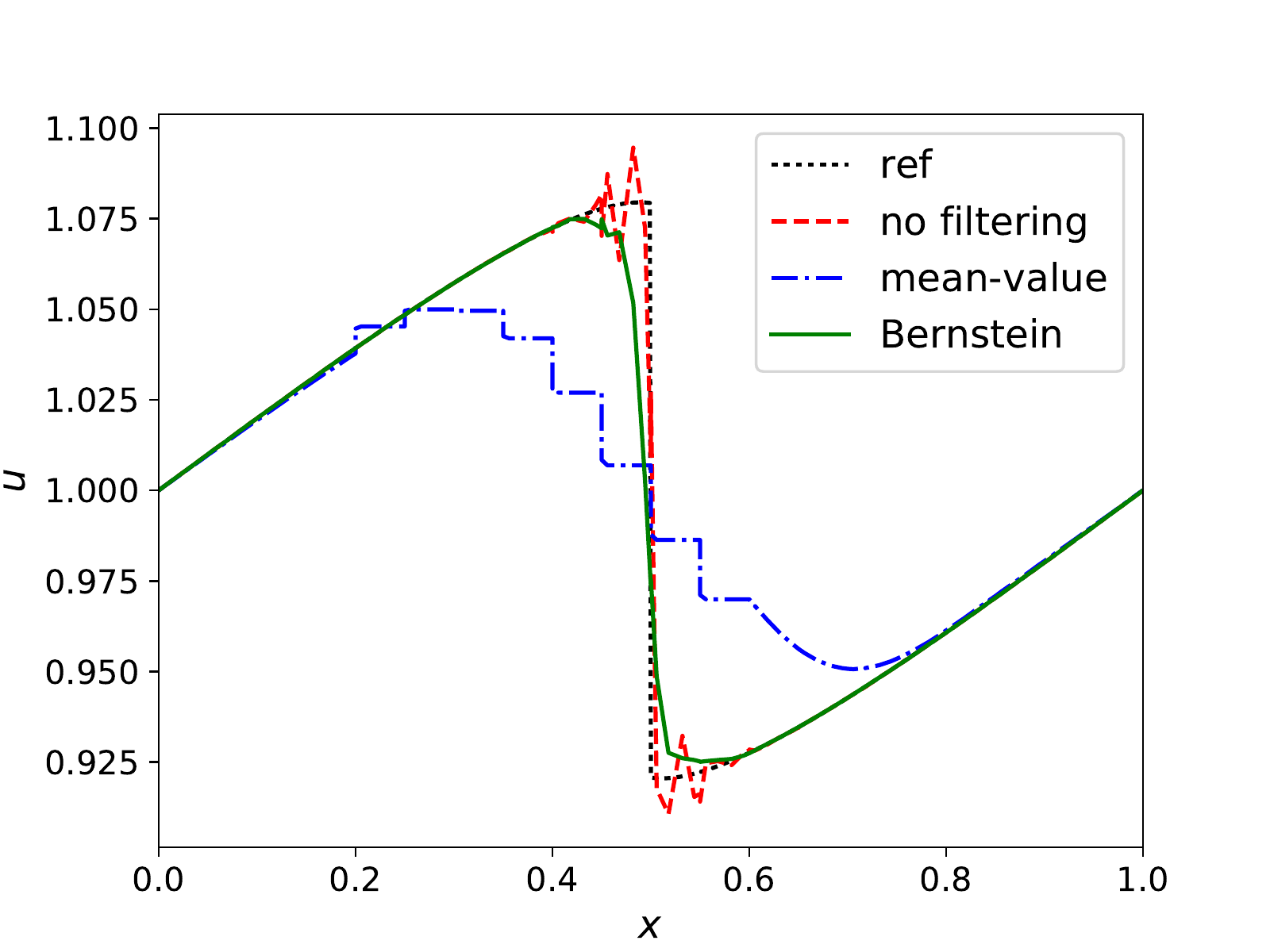}
    \caption{$N=5$ \& $I=20$.}
    \label{fig:burgers_T=3_comp_N5_I20} 
  \end{subfigure}%
  \\ 
  \begin{subfigure}[b]{0.33\textwidth}
    \includegraphics[width=\textwidth]{%
      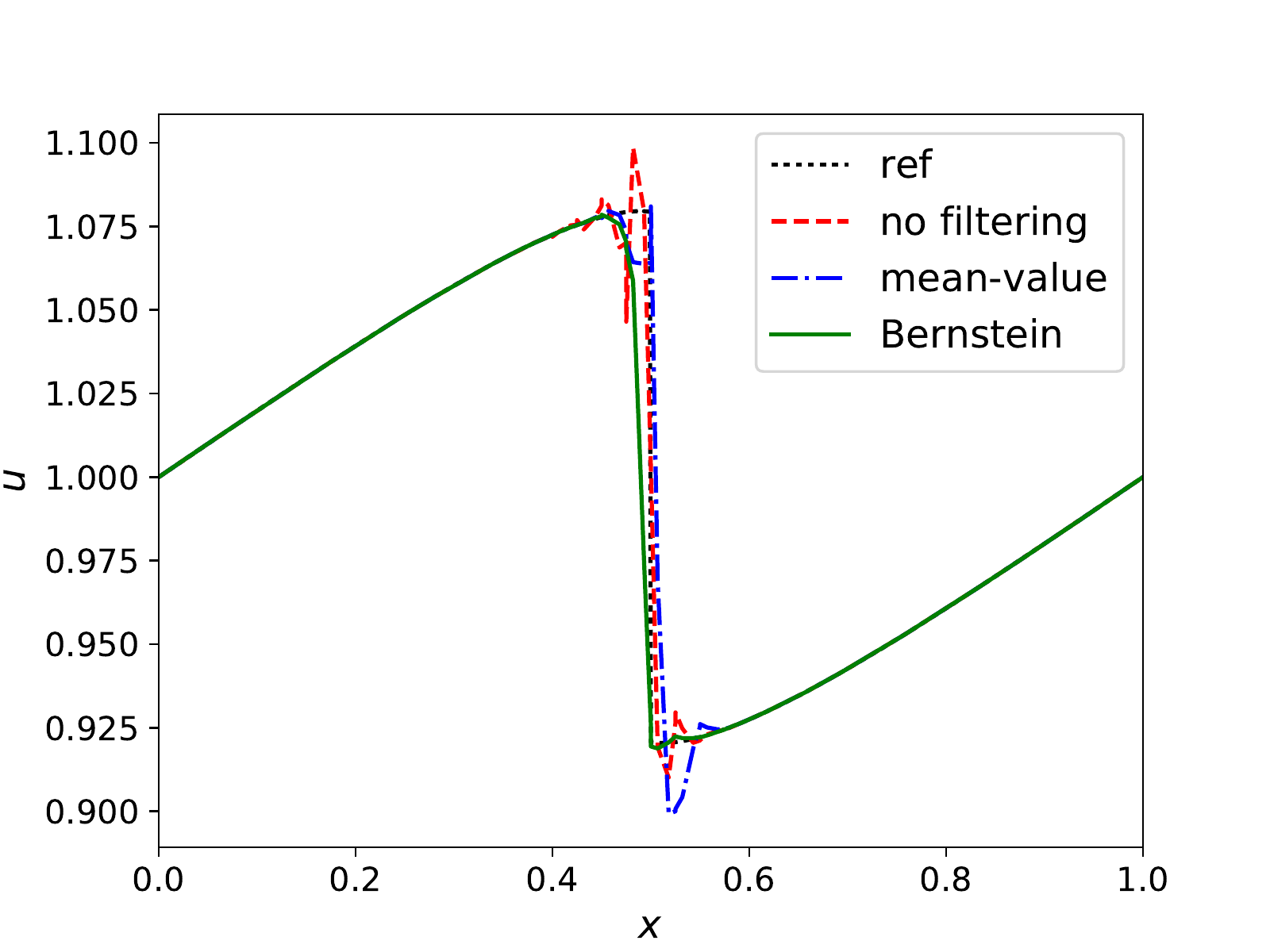}
    \caption{$N=3$ \& $I=40$.}
    \label{fig:burgers_T=3_comp_N3_I40} 
  \end{subfigure}%
  ~ 
  \begin{subfigure}[b]{0.33\textwidth}
    \includegraphics[width=\textwidth]{%
      comp_burgers_T=3_I=40_N=4}
    \caption{$N=4$ \& $I=40$.}
    \label{fig:burgers_T=3_comp_N4_I40} 
  \end{subfigure}%
  ~ 
  \begin{subfigure}[b]{0.33\textwidth}
    \includegraphics[width=\textwidth]{%
      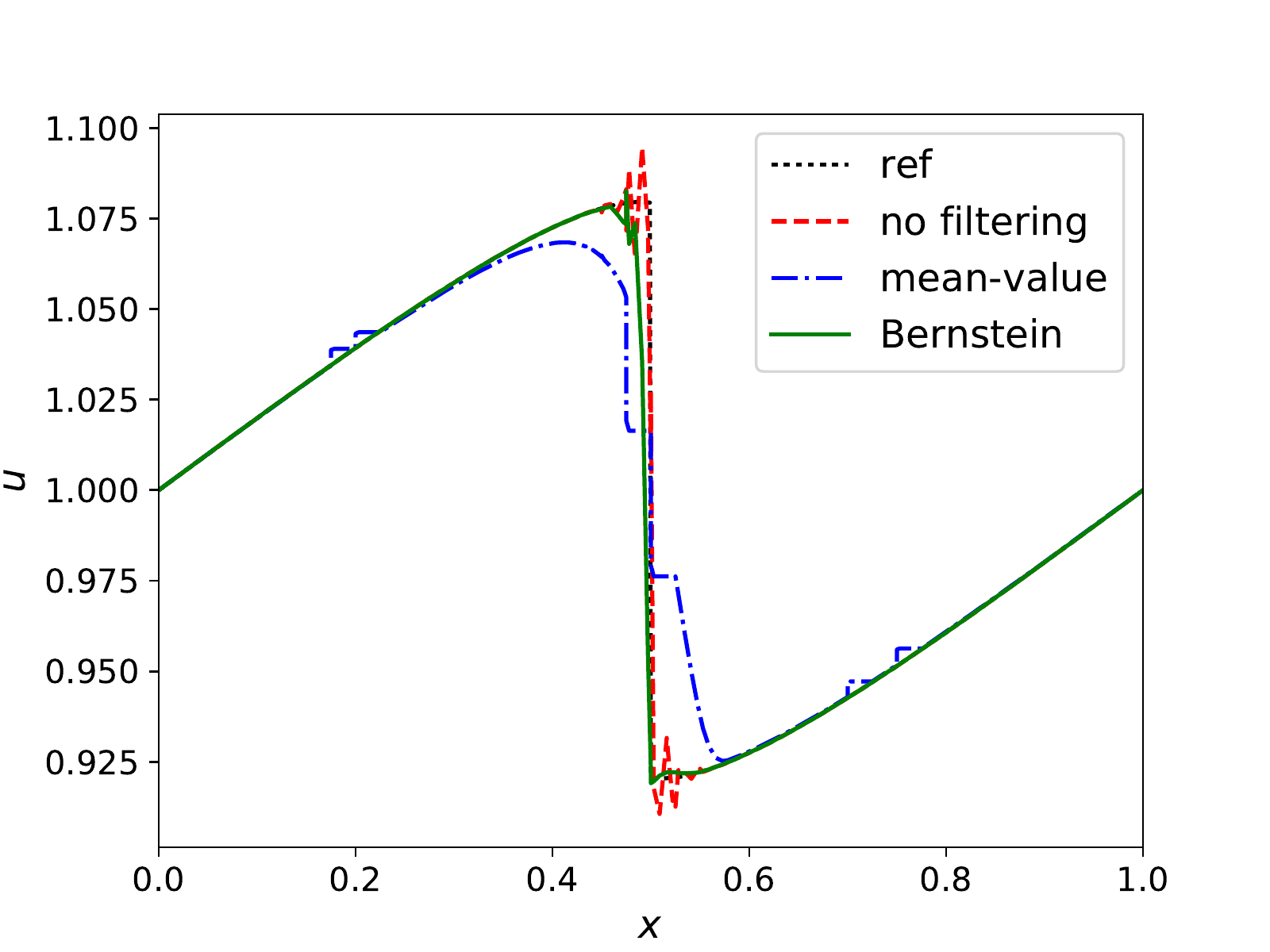}
    \caption{$N=5$ \& $I=40$.}
    \label{fig:burgers_T=3_comp_N5_I40} 
  \end{subfigure}%
  \\ 
  \begin{subfigure}[b]{0.33\textwidth}
    \includegraphics[width=\textwidth]{%
      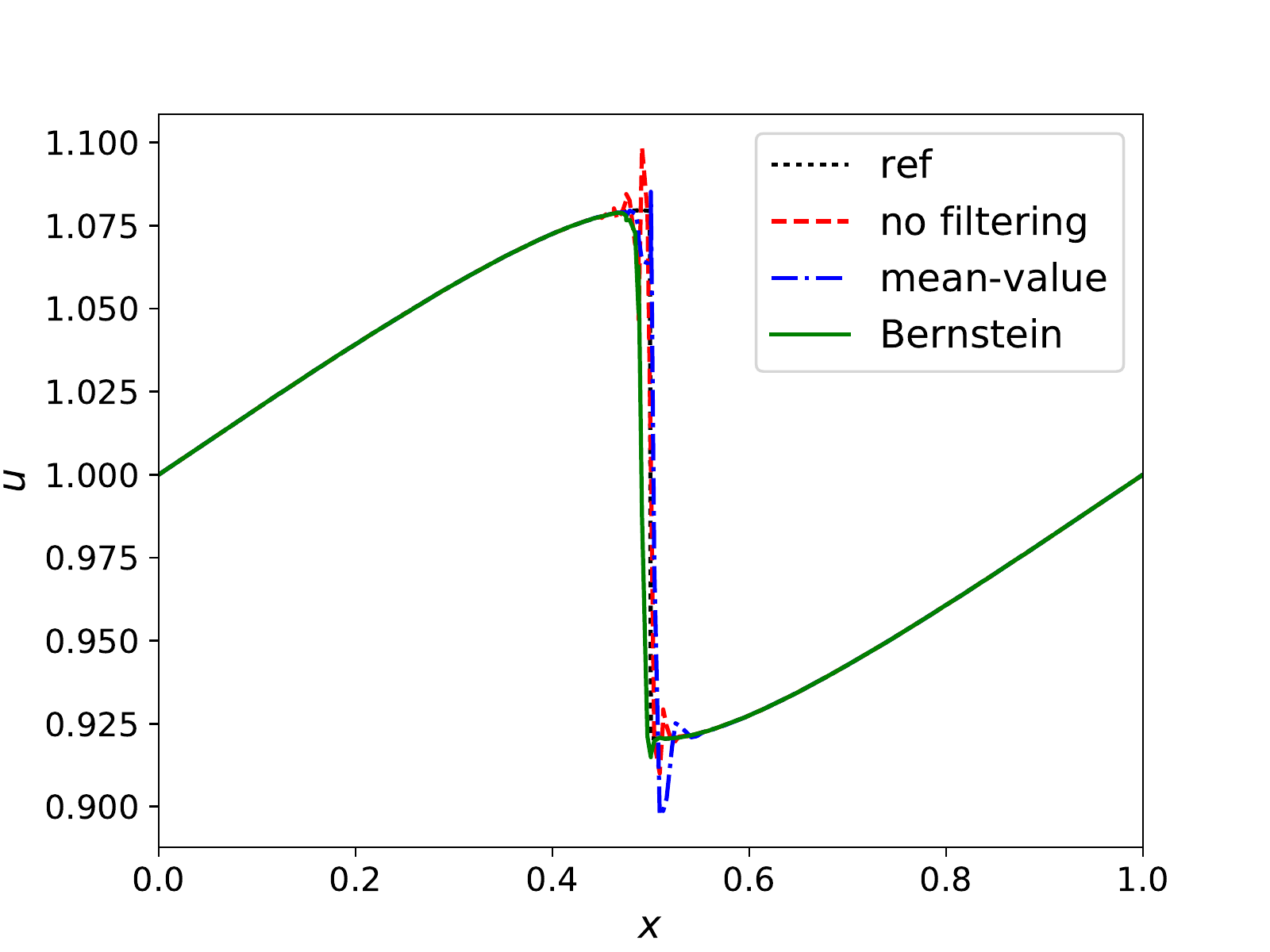}
    \caption{$N=3$ \& $I=80$.}
    \label{fig:burgers_T=3_comp_N3_I80} 
  \end{subfigure}%
  ~ 
  \begin{subfigure}[b]{0.33\textwidth}
    \includegraphics[width=\textwidth]{%
      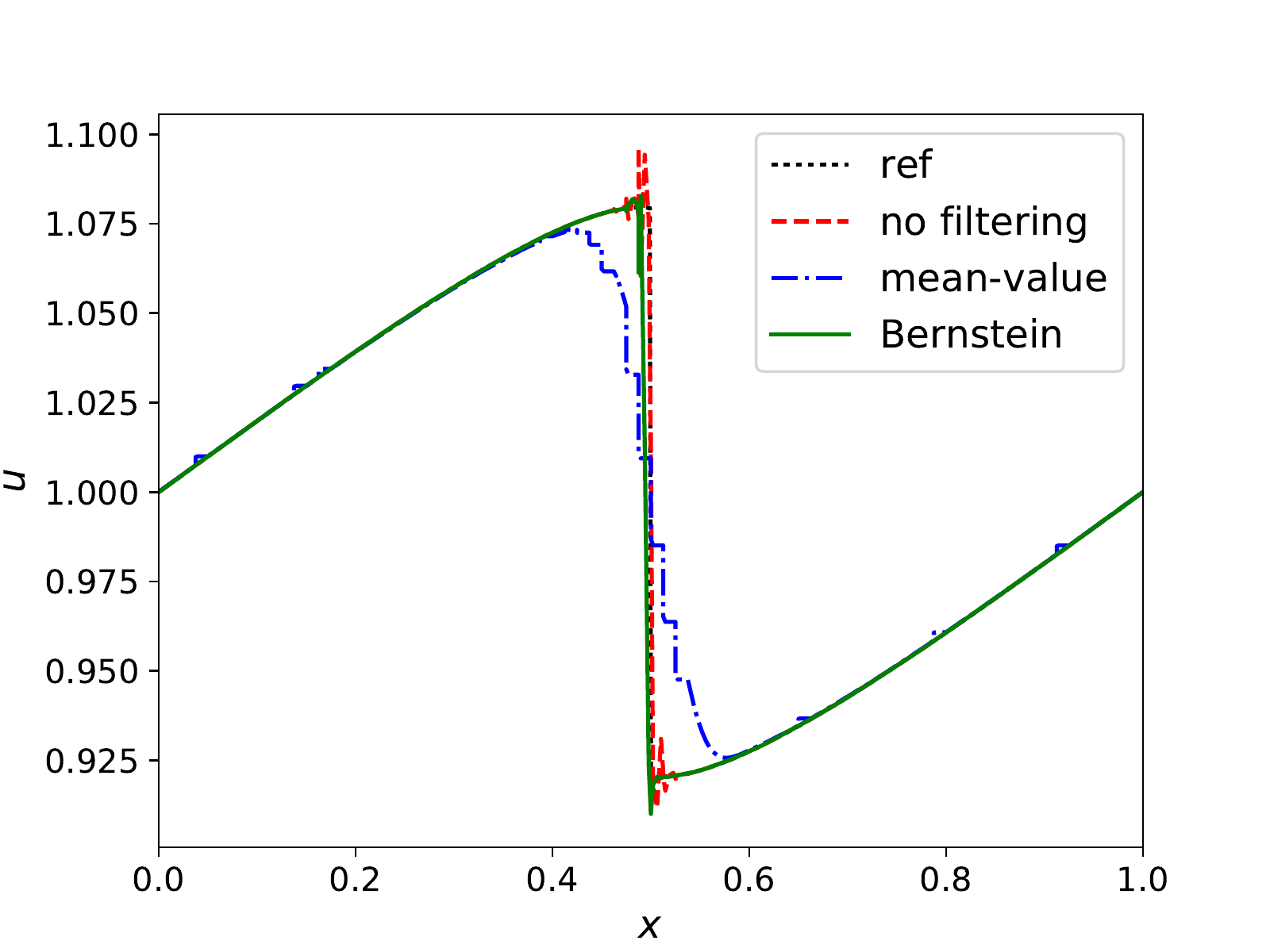}
    \caption{$N=4$ \& $I=80$.}
    \label{fig:burgers_T=3_comp_N4_I80} 
  \end{subfigure}%
  ~ 
  \begin{subfigure}[b]{0.33\textwidth}
    \includegraphics[width=\textwidth]{%
      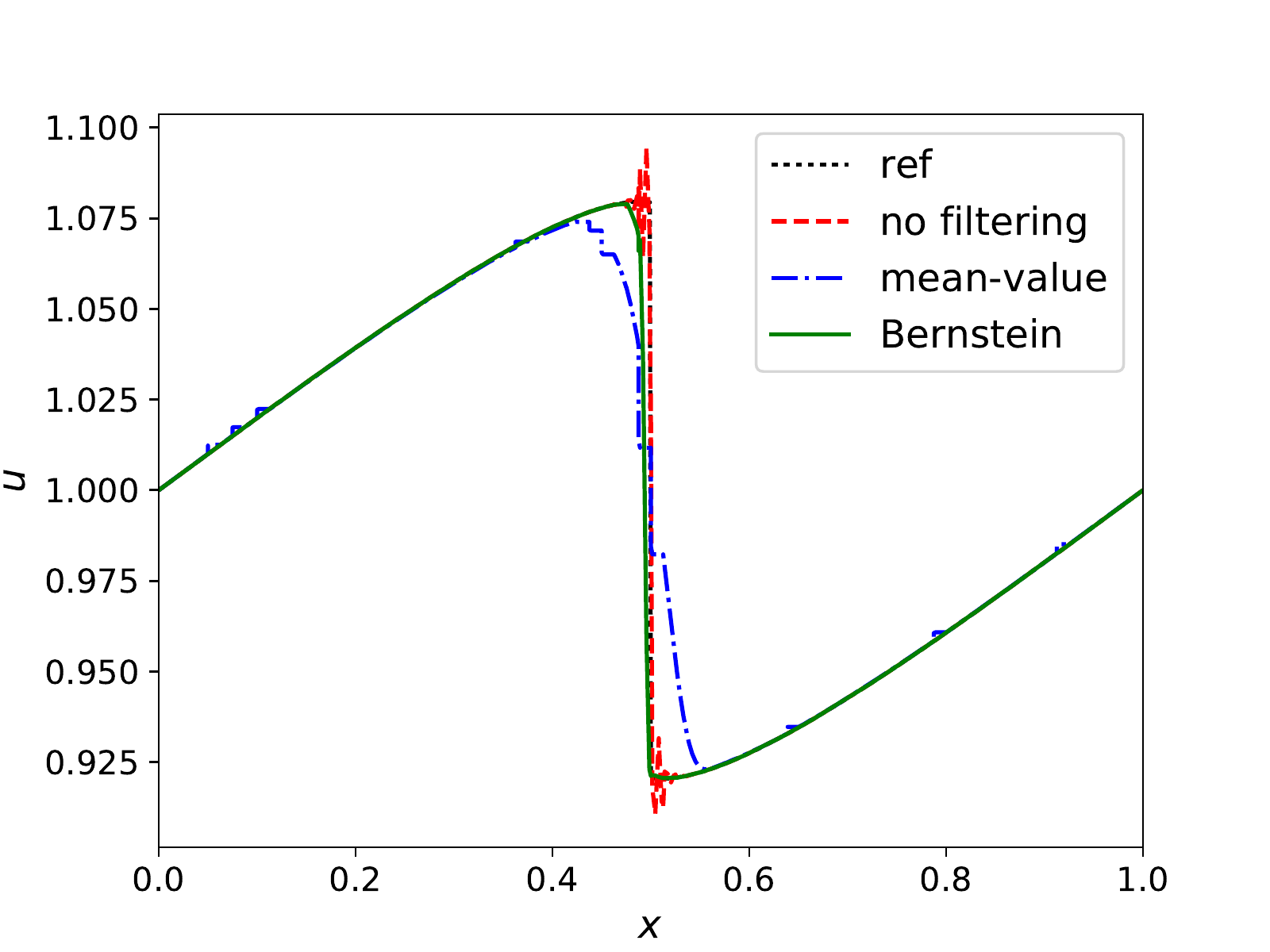}
    \caption{$N=5$ \& $I=80$.}
    \label{fig:burgers_T=3_comp_N5_I80} 
  \end{subfigure}%
  \caption{Comparison of the Bernstein procedure 
    in a DG method with a usual filtering technique and no filtering for the invicid Burgers equation 
    \eqref{eq:problem2} at time $t=3$.}
  \label{fig:burgers_comp_T=3}
\end{figure}


\begin{figure}[!htb]
  \centering
  \begin{subfigure}[b]{0.33\textwidth}
    \includegraphics[width=\textwidth]{%
      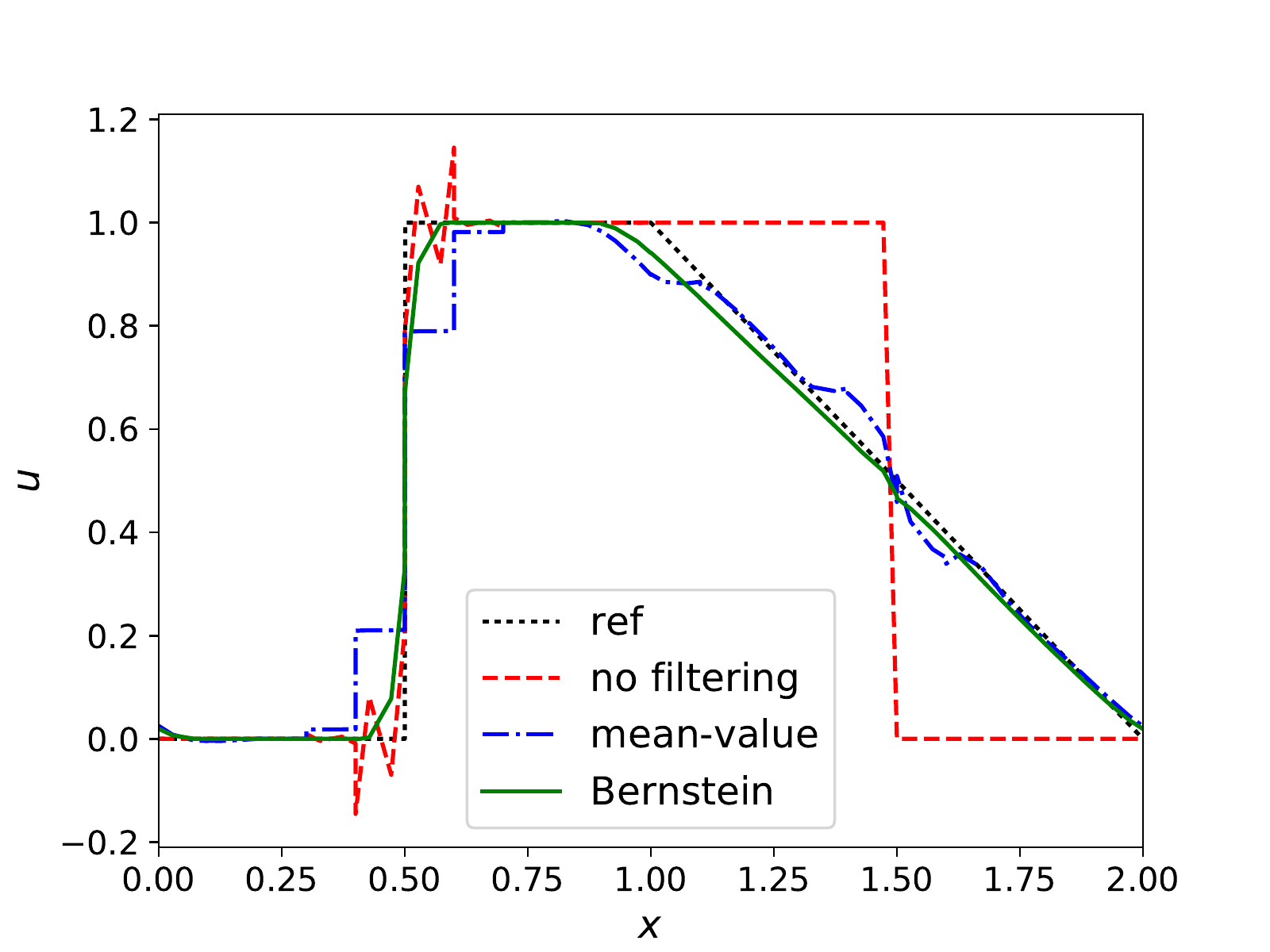}
    \caption{$N=3$ \& $I=20$.}
    \label{fig:linear+burgers_T=05_I=20_N=3} 
  \end{subfigure}%
  ~ 
  \begin{subfigure}[b]{0.33\textwidth}
    \includegraphics[width=\textwidth]{%
      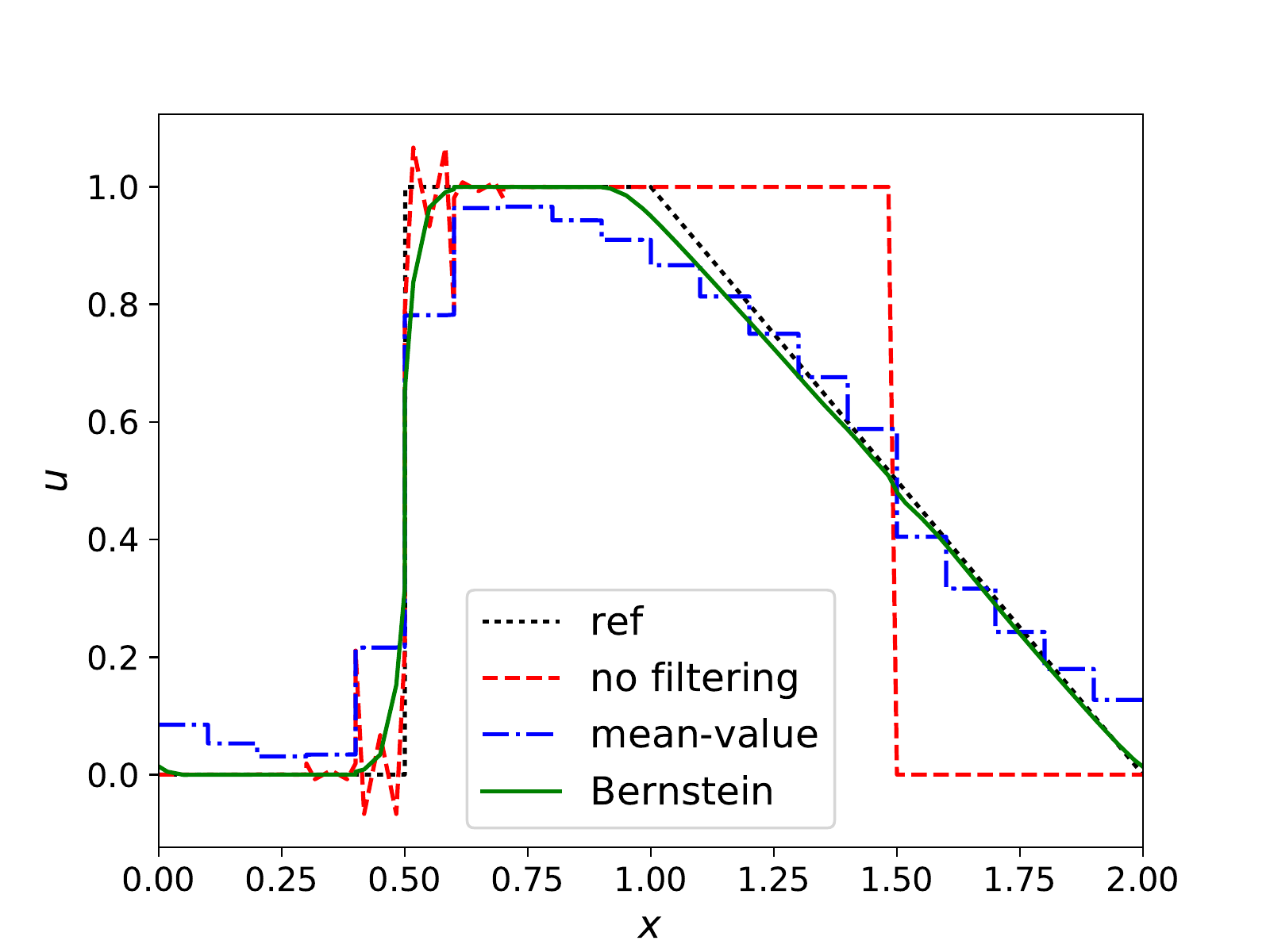}
    \caption{$N=4$ \& $I=20$.}
    \label{fig:linear+burgers_T=05_I=20_N=4} 
  \end{subfigure}%
  ~ 
  \begin{subfigure}[b]{0.33\textwidth}
    \includegraphics[width=\textwidth]{%
      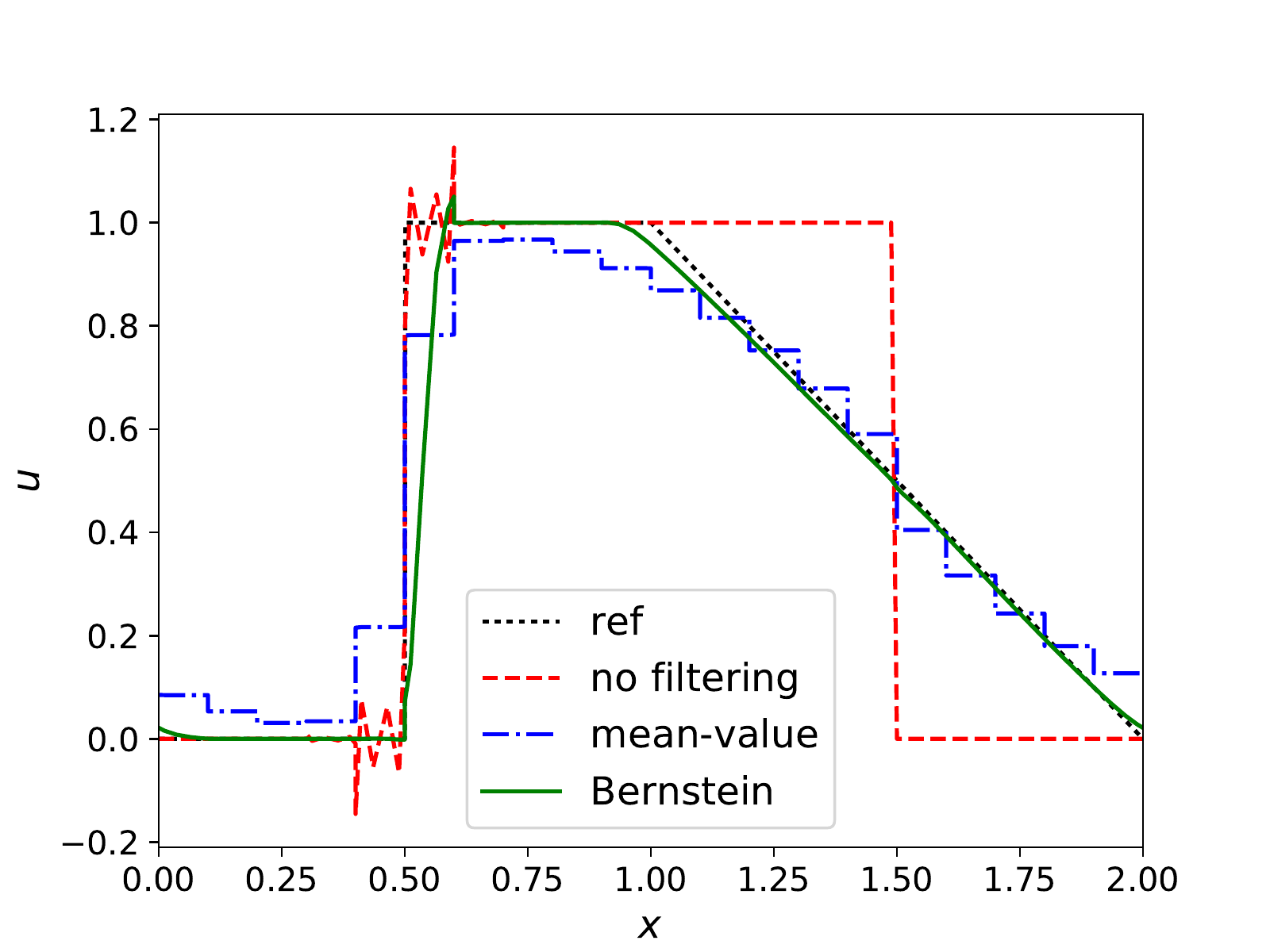}
    \caption{$N=5$ \& $I=20$.}
    \label{fig:linear+burgers_T=05_I=20_N=5} 
  \end{subfigure}%
  \\
  \begin{subfigure}[b]{0.33\textwidth}
    \includegraphics[width=\textwidth]{%
      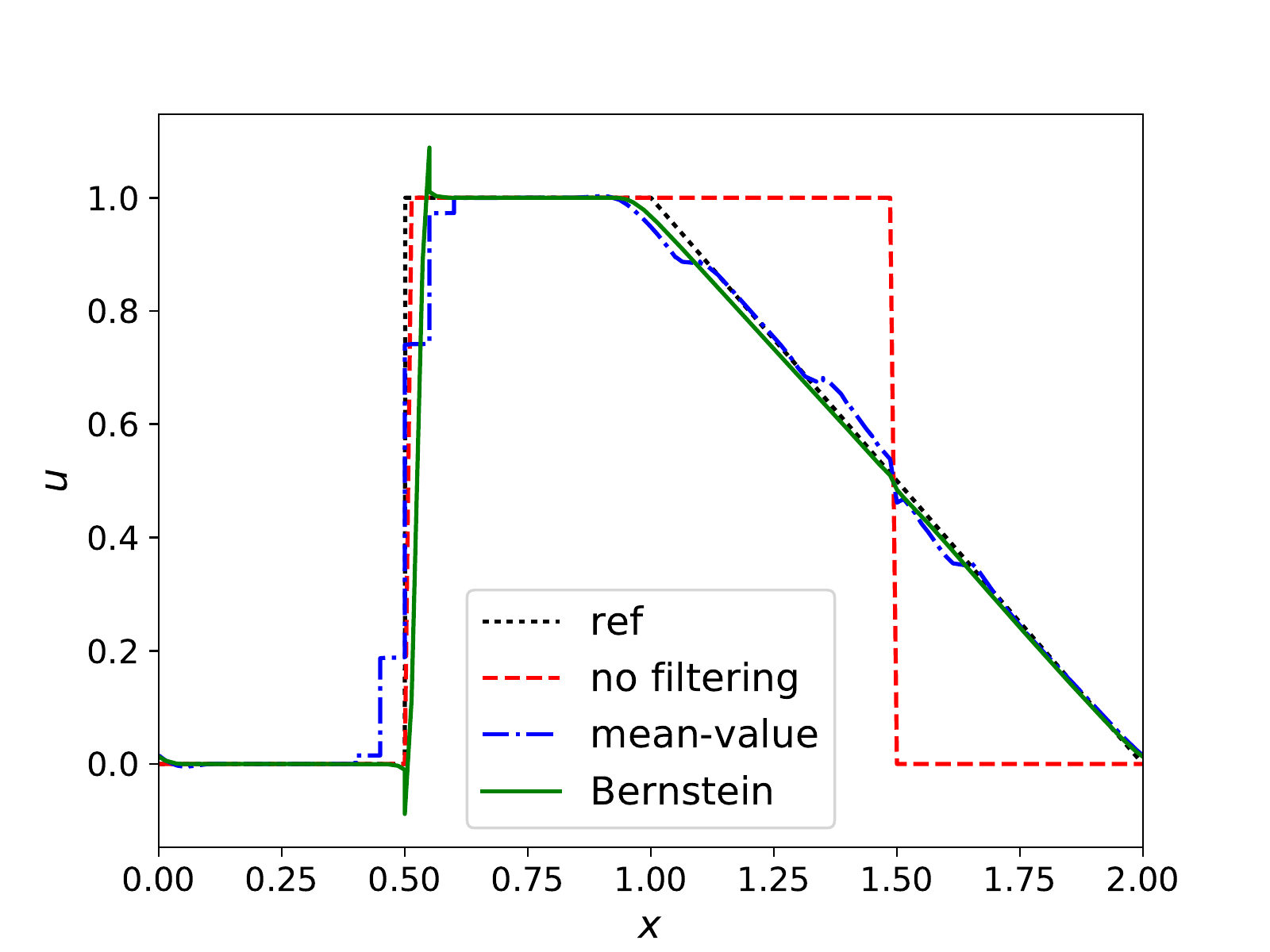}
    \caption{$N=3$ \& $I=40$.}
    \label{fig:linear+burgers_T=05_I=40_N=3} 
  \end{subfigure}%
  ~ 
  \begin{subfigure}[b]{0.33\textwidth}
    \includegraphics[width=\textwidth]{%
      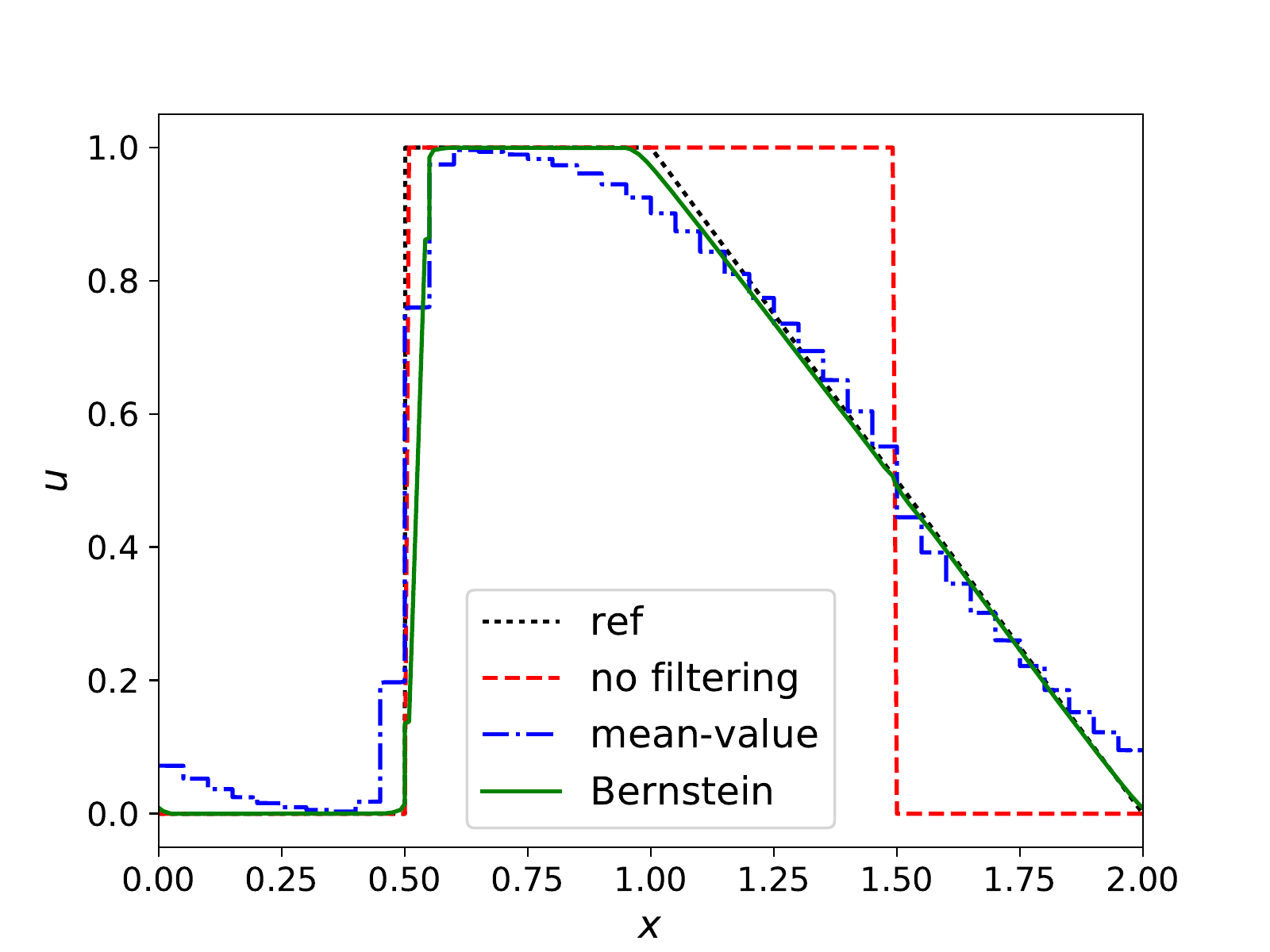}
    \caption{$N=4$ \& $I=40$.}
    \label{fig:linear+burgers_T=05_I=40_N=4} 
  \end{subfigure}%
  ~ 
  \begin{subfigure}[b]{0.33\textwidth}
    \includegraphics[width=\textwidth]{%
      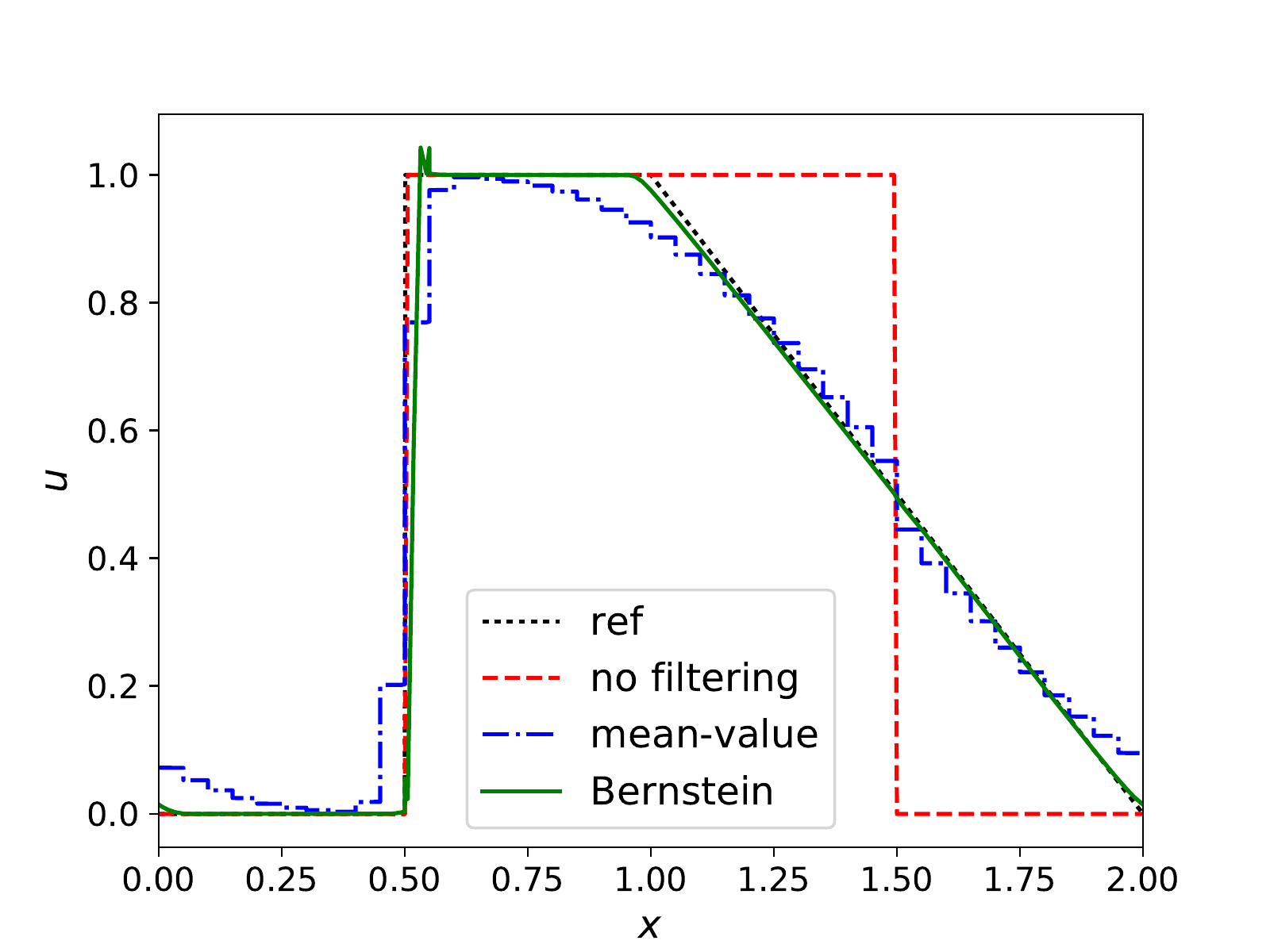}
    \caption{$N=5$ \& $I=40$.}
    \label{fig:linear+burgers_T=05_I=40_N=5} 
  \end{subfigure}%
  \\
  \begin{subfigure}[b]{0.33\textwidth}
    \includegraphics[width=\textwidth]{%
      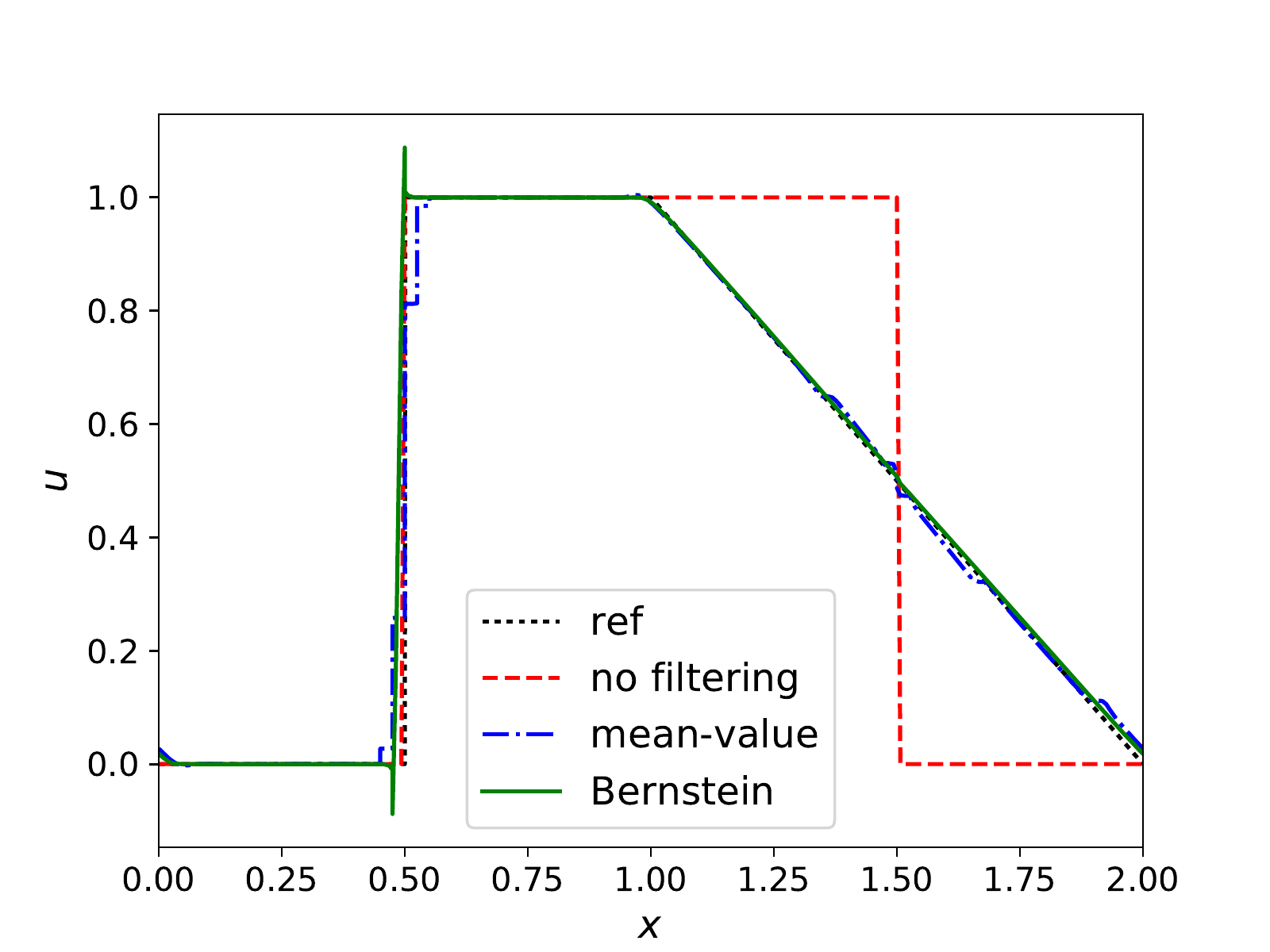}
    \caption{$N=3$ \& $I=80$.}
    \label{fig:linear+burgers_T=05_I=80_N=3} 
  \end{subfigure}%
  ~ 
  \begin{subfigure}[b]{0.33\textwidth}
    \includegraphics[width=\textwidth]{%
      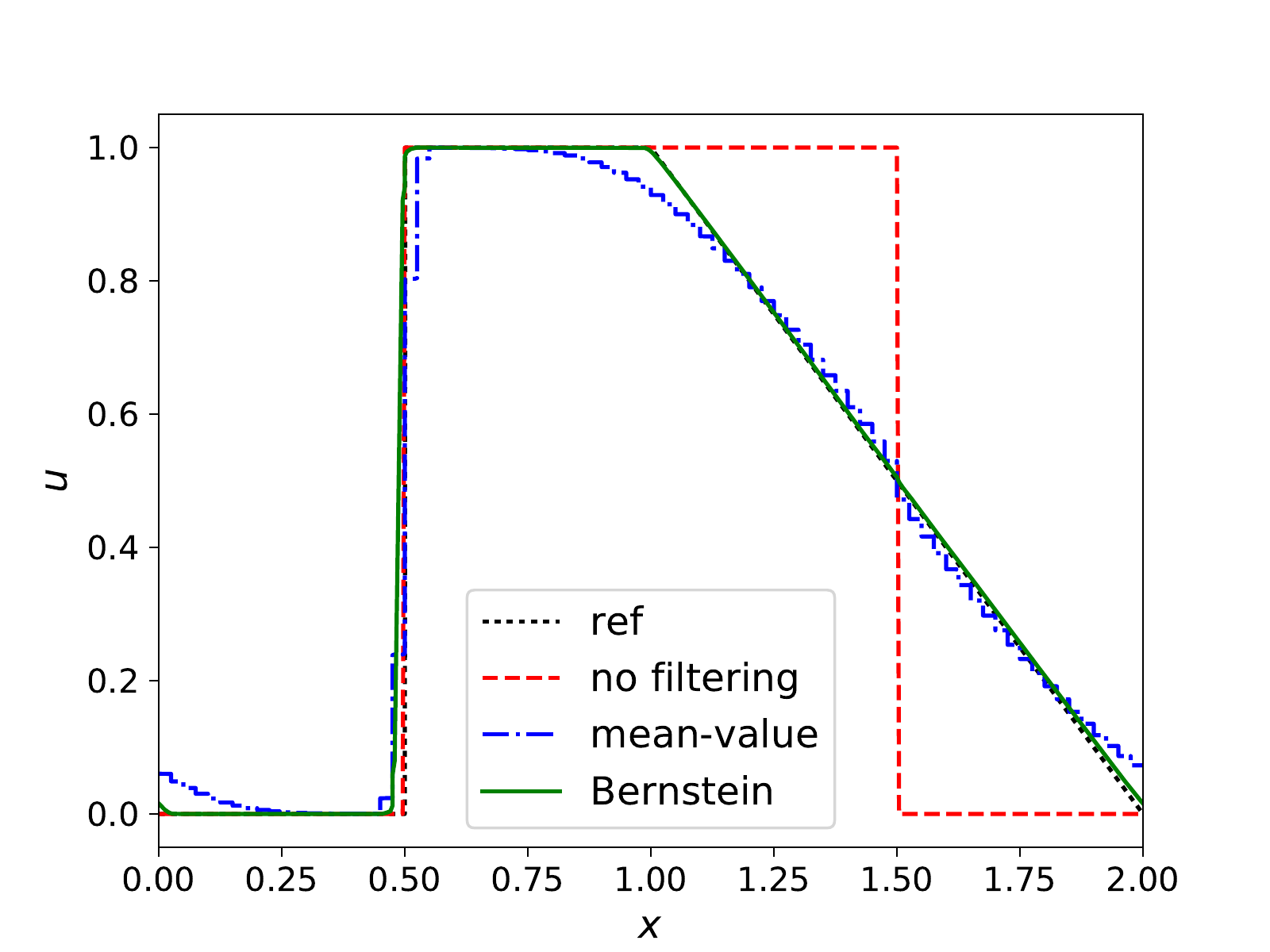}
    \caption{$N=4$ \& $I=80$.}
    \label{fig:linear+burgers_T=05_I=80_N=4} 
  \end{subfigure}%
  ~ 
  \begin{subfigure}[b]{0.33\textwidth}
    \includegraphics[width=\textwidth]{%
      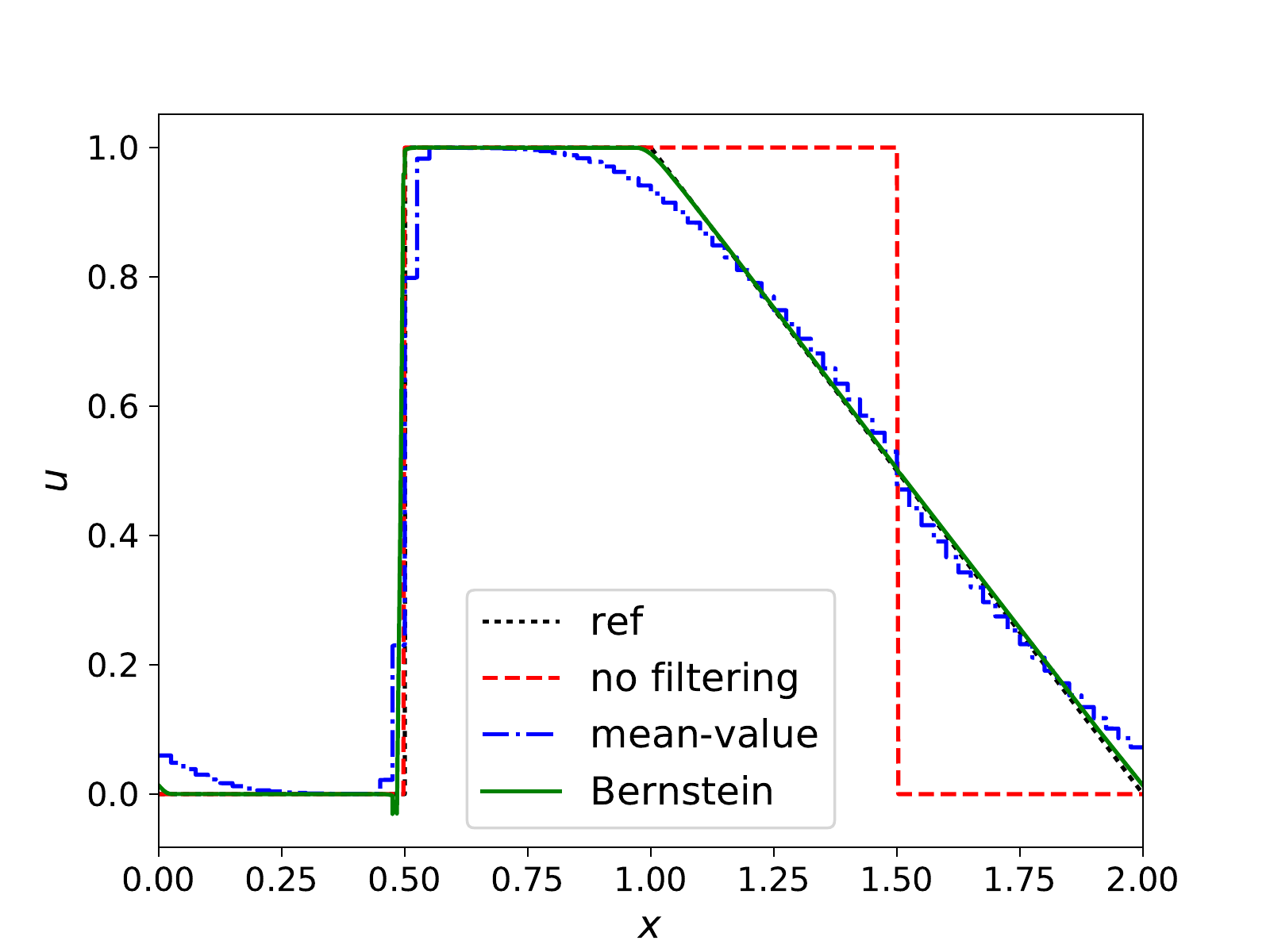}
    \caption{$N=5$ \& $I=80$.}
    \label{fig:linear+burgers_T=05_I=80_N=5} 
  \end{subfigure}%
  \caption{Comparison of the Bernstein procedure 
    in a DG method with a usual filtering technique and no filtering for \eqref{eq:problem3} at time $t=0.5$.}
  \label{fig:linear+burgers_comp_T=05}
\end{figure}


\begin{figure}[!htb]
  \centering
  \begin{subfigure}[b]{0.33\textwidth}
    \includegraphics[width=\textwidth]{%
      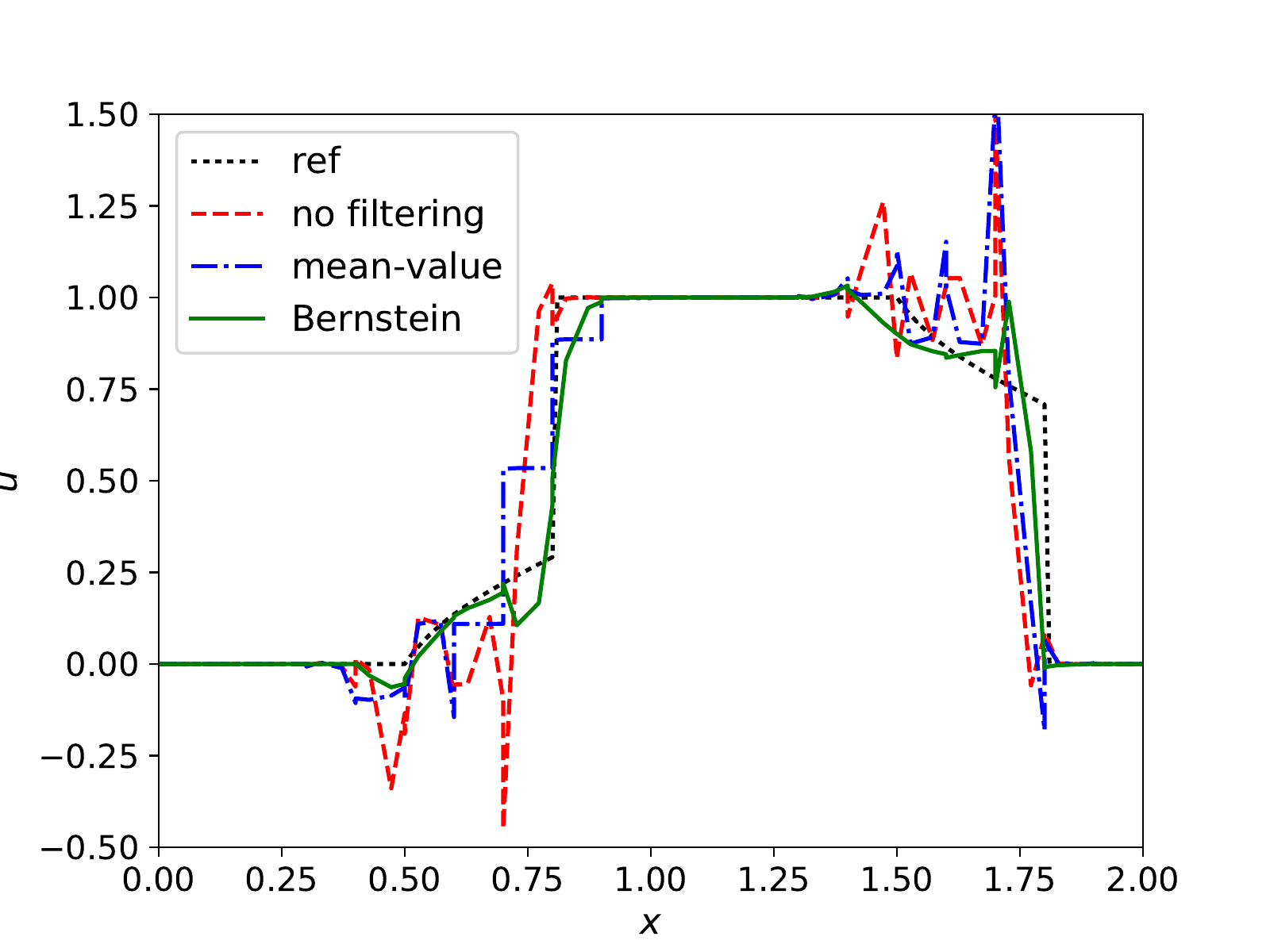}
    \caption{$N=3$ \& $I=20$.}
    \label{fig:buckley-leverett_T=025_I=20_N=3} 
  \end{subfigure}%
  ~ 
  \begin{subfigure}[b]{0.33\textwidth}
    \includegraphics[width=\textwidth]{%
      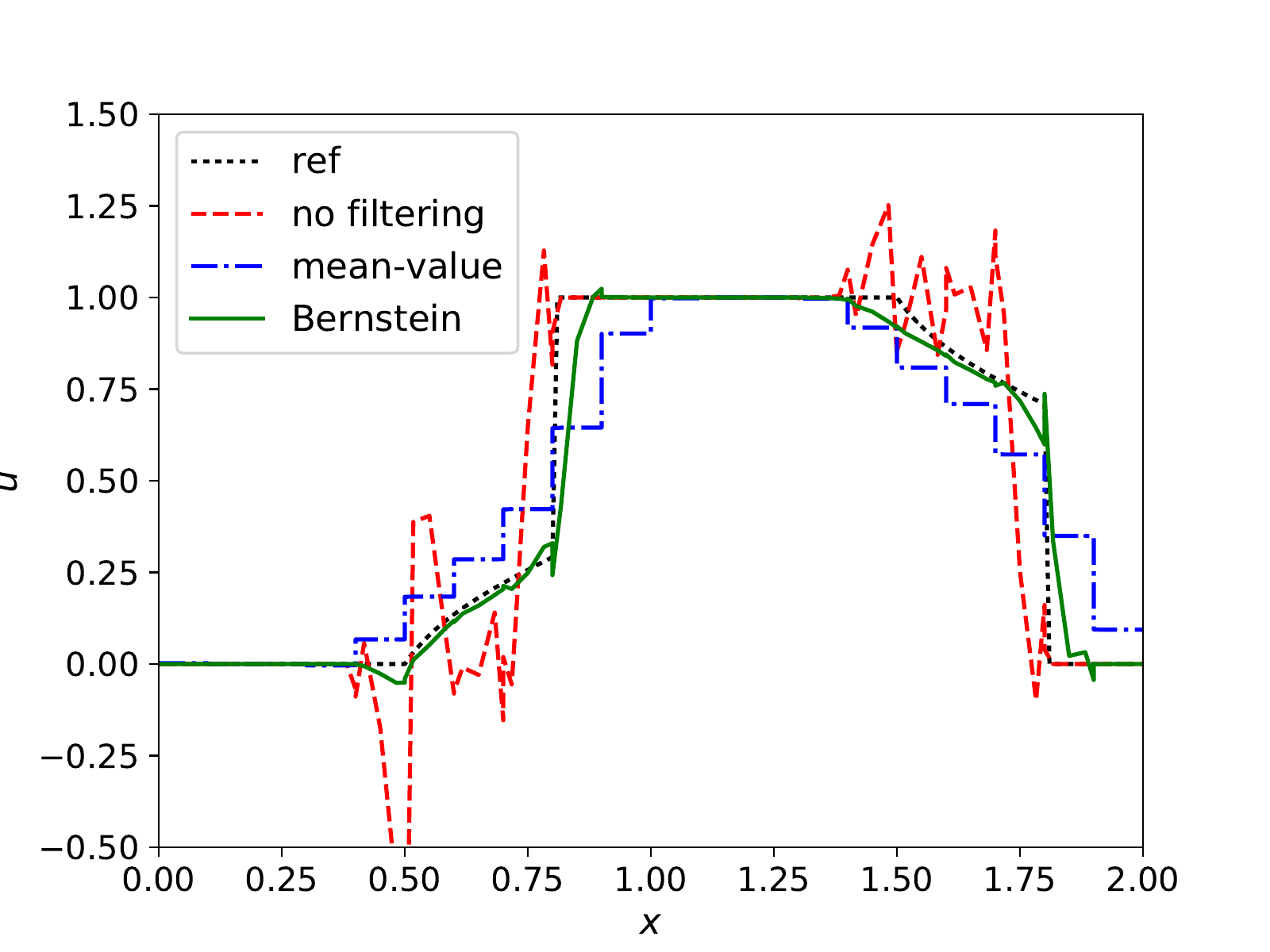}
    \caption{$N=4$ \& $I=20$.}
    \label{fig:buckley-leverett_T=025_I=20_N=4} 
  \end{subfigure}%
  ~ 
  \begin{subfigure}[b]{0.33\textwidth}
    \includegraphics[width=\textwidth]{%
      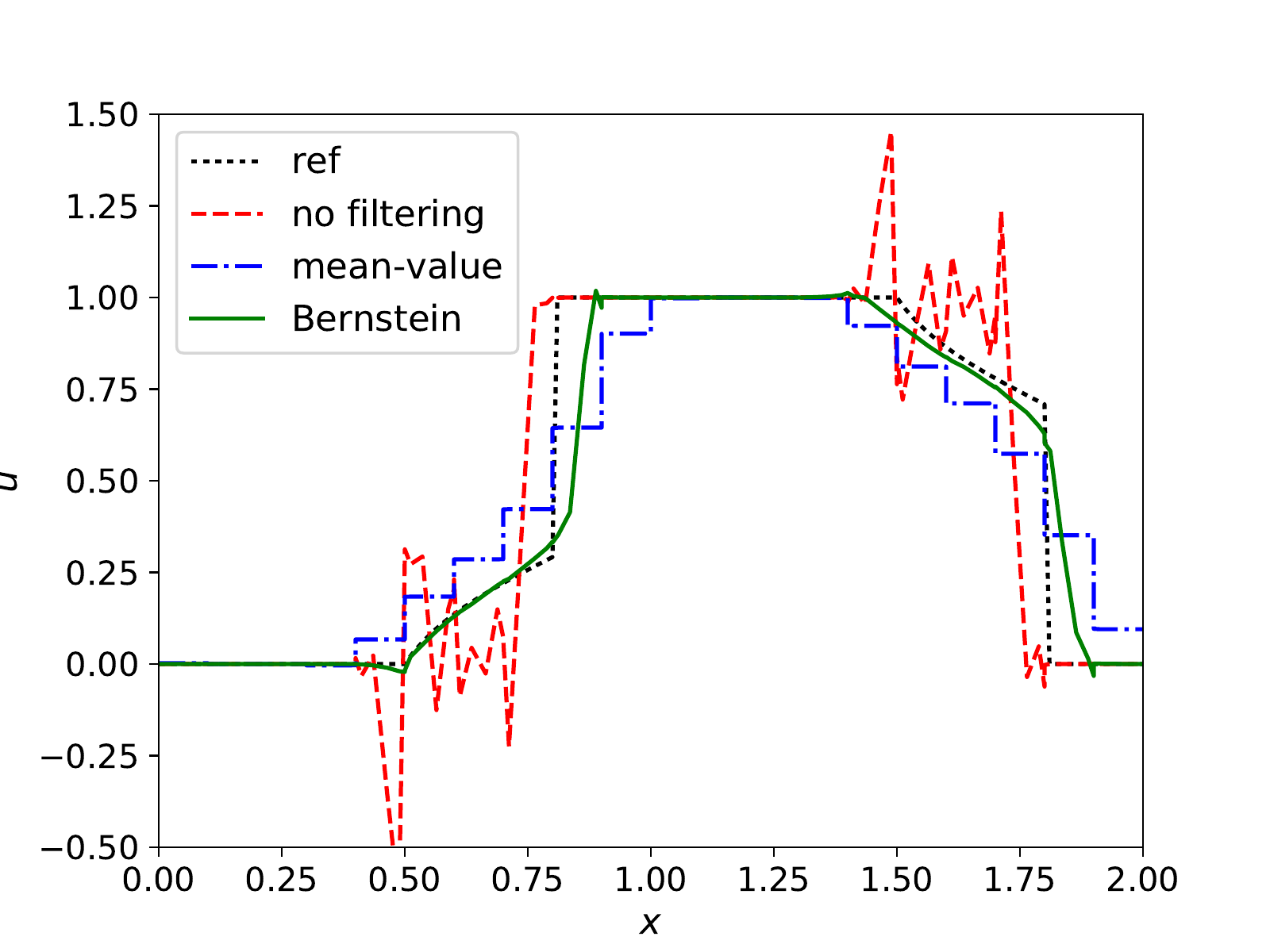}
    \caption{$N=5$ \& $I=20$.}
    \label{fig:buckley-leverett_T=025_I=20_N=5} 
  \end{subfigure}%
  \\ 
  \begin{subfigure}[b]{0.33\textwidth}
    \includegraphics[width=\textwidth]{%
      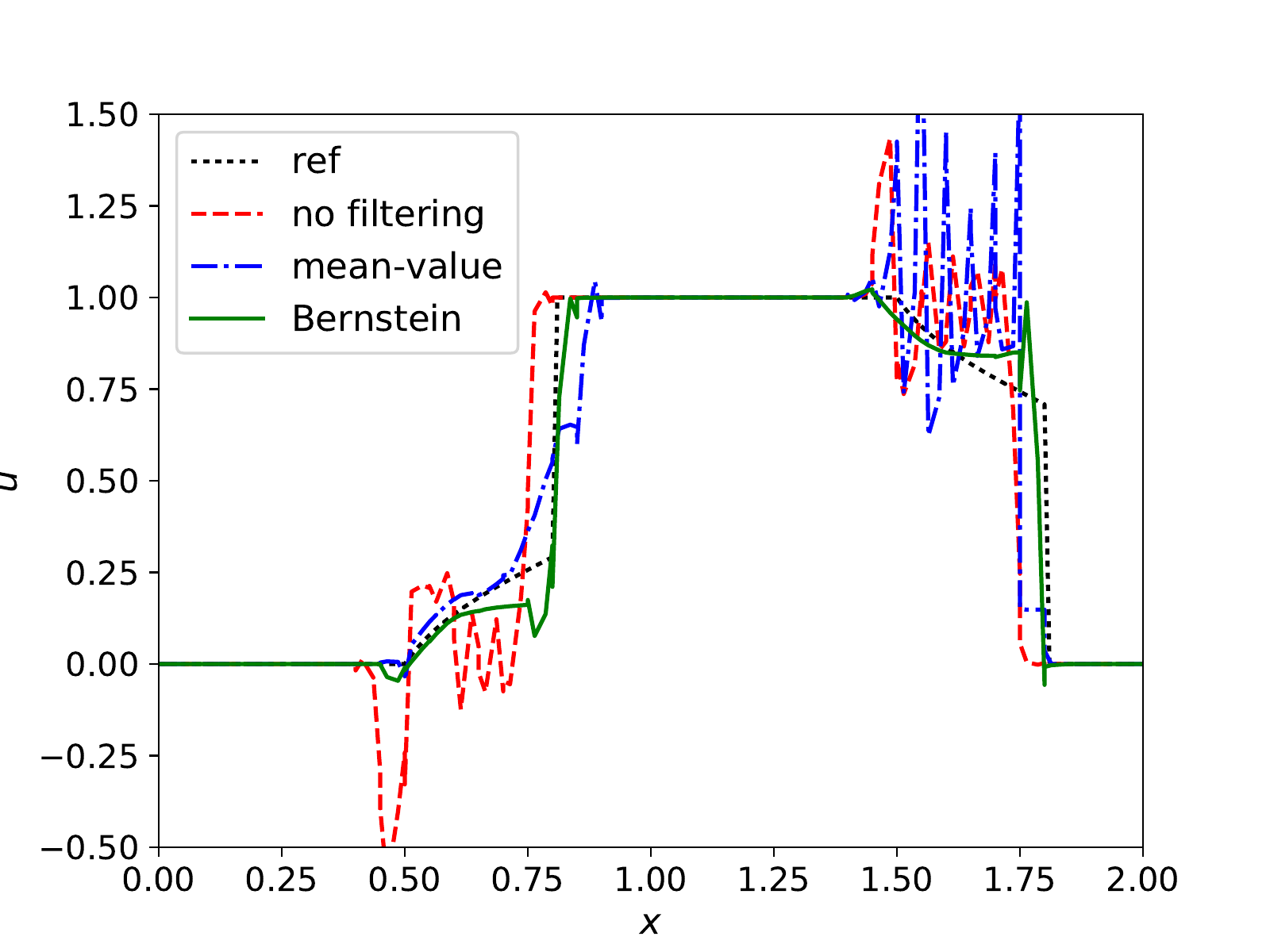}
    \caption{$N=3$ \& $I=40$.}
    \label{fig:buckley-leverett_T=025_I=40_N=3} 
  \end{subfigure}%
  ~ 
  \begin{subfigure}[b]{0.33\textwidth}
    \includegraphics[width=\textwidth]{%
      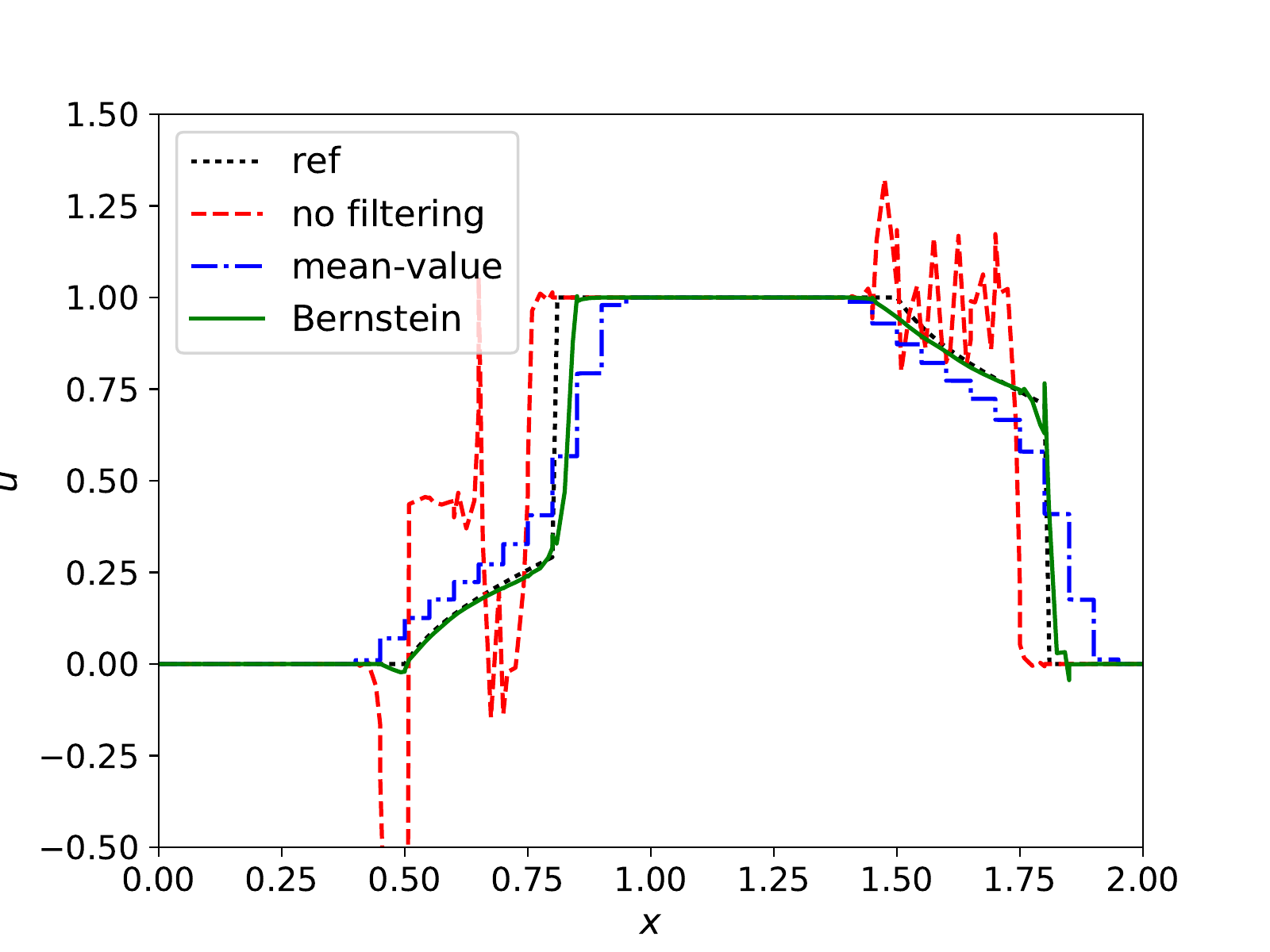}
    \caption{$N=4$ \& $I=40$.}
    \label{fig:buckley-leverett_T=025_I=40_N=4} 
  \end{subfigure}%
  ~ 
  \begin{subfigure}[b]{0.33\textwidth}
    \includegraphics[width=\textwidth]{%
      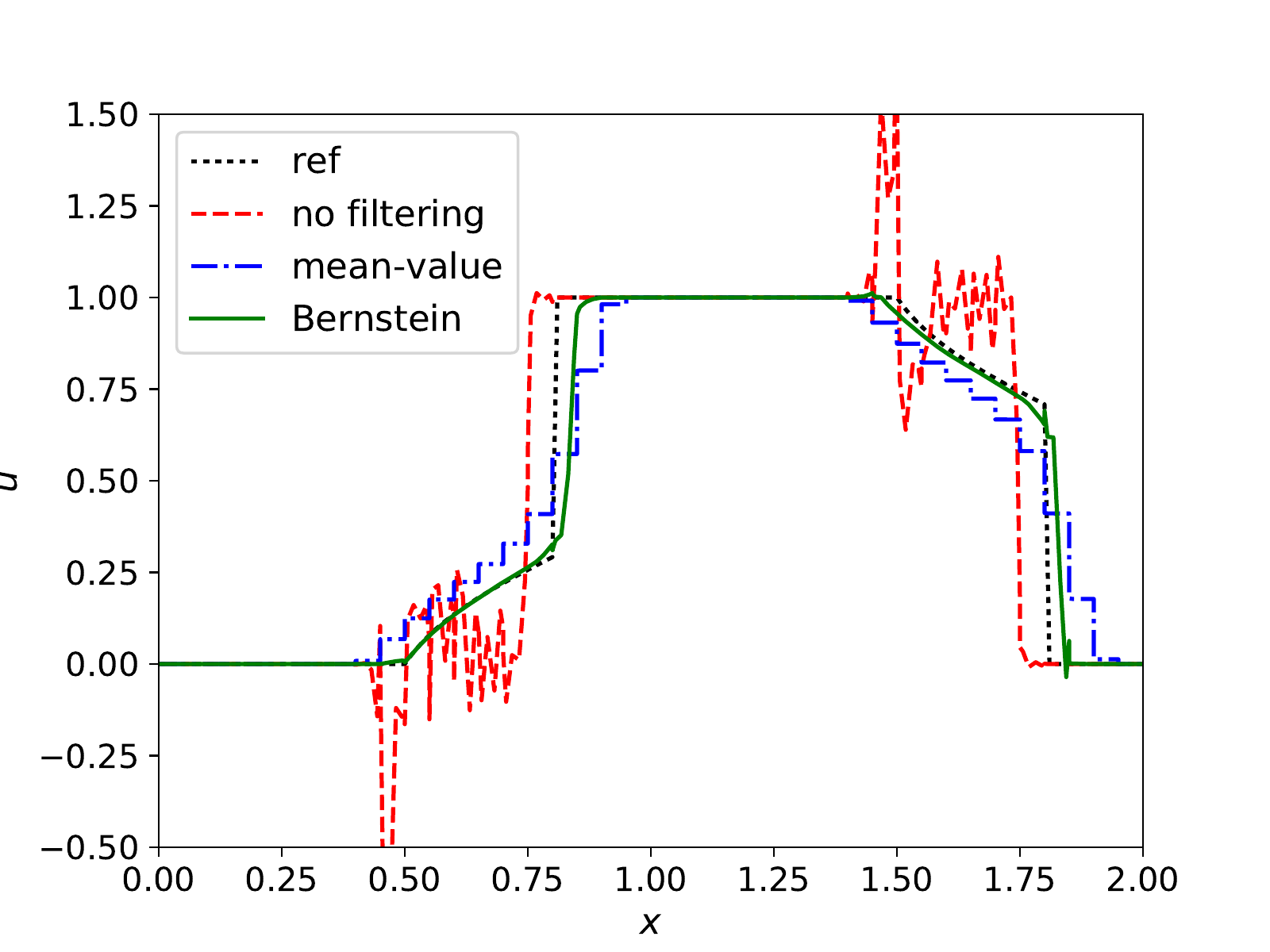}
    \caption{$N=5$ \& $I=40$.}
    \label{fig:buckley-leverett_T=025_I=40_N=5} 
  \end{subfigure}%
  \\
  \begin{subfigure}[b]{0.33\textwidth}
    \includegraphics[width=\textwidth]{%
      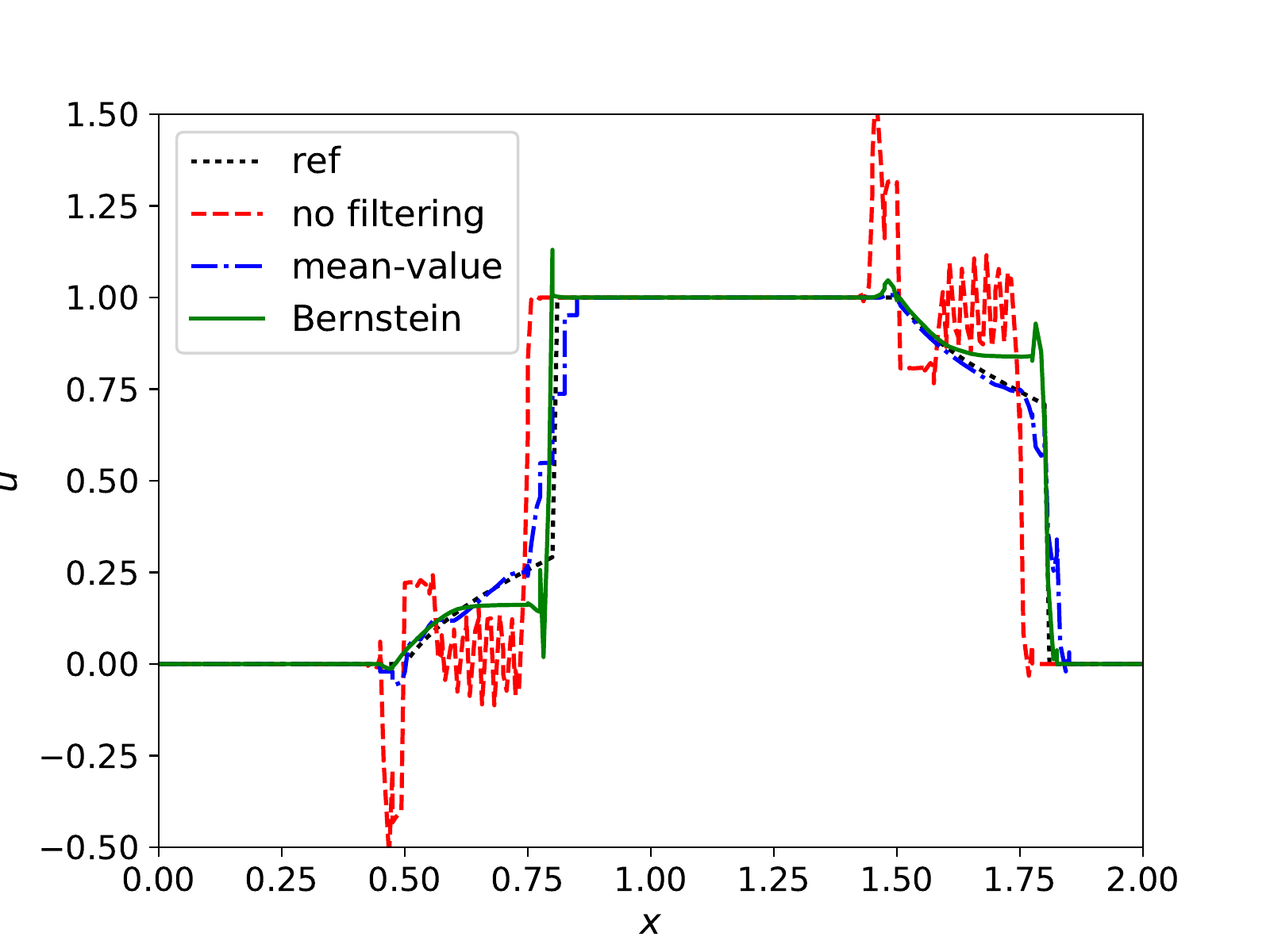}
    \caption{$N=3$ \& $I=80$.}
    \label{fig:buckley-leverett_T=025_I=80_N=3} 
  \end{subfigure}%
  ~ 
  \begin{subfigure}[b]{0.33\textwidth}
    \includegraphics[width=\textwidth]{%
      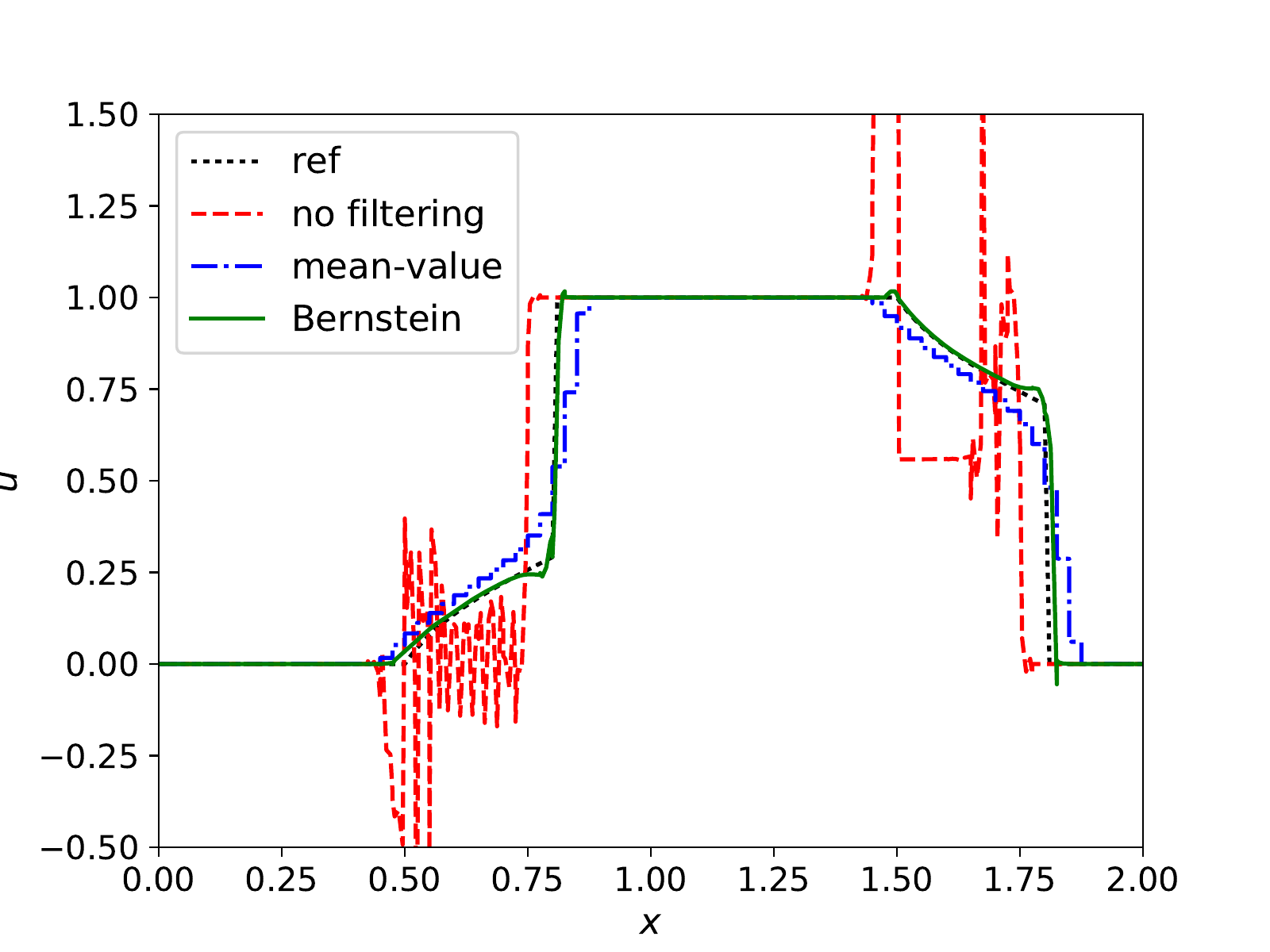}
    \caption{$N=4$ \& $I=80$.}
    \label{fig:buckley-leverett_T=025_I=80_N=4} 
  \end{subfigure}%
  ~ 
  \begin{subfigure}[b]{0.33\textwidth}
    \includegraphics[width=\textwidth]{%
      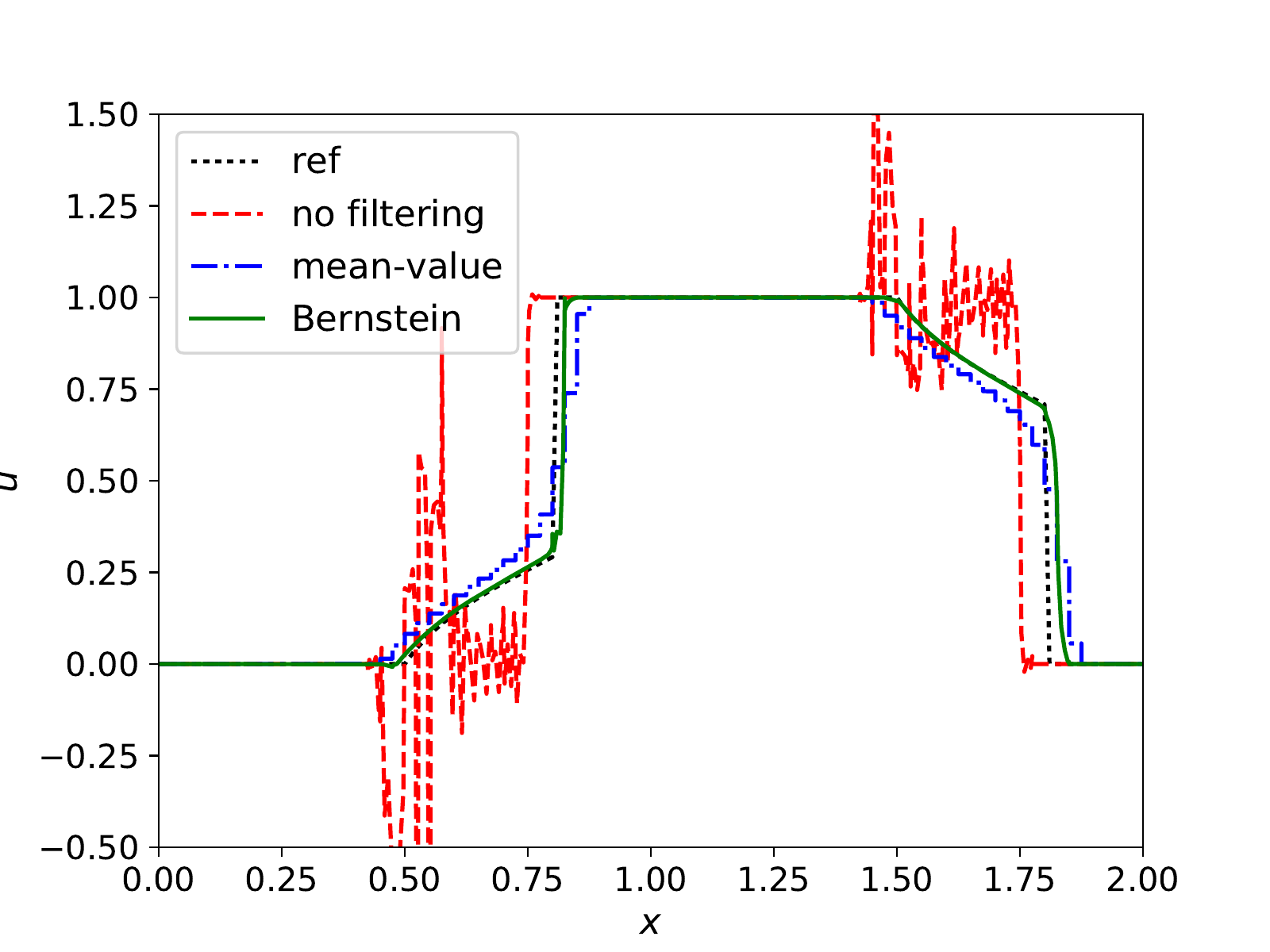}
    \caption{$N=5$ \& $I=80$.}
    \label{fig:buckley-leverett_T=025_I=80_N=5} 
  \end{subfigure}%
  \caption{Comparison of the Bernstein procedure 
    in a DG method with a usual filtering technique and no filtering for \eqref{eq:problem4} at time $t=0.25$.}
  \label{fig:buckley-leverett_comp_T=025}
\end{figure} 


\section*{Acknowledgements}
The author would like to thank Thomas Sonar (TU Braunschweig, Germany) for helpful 
advice. 
This work is supported by the German Research Foundation (DFG, Deutsche Forschungsgemeinschaft) under Grant SO 
363/15-1.

\bibliographystyle{abbrv}
\bibliography{literature}

\end{document}